\documentclass[mathpazo]{cicp}

\usepackage{amsmath,amssymb,amsthm}  
\usepackage{graphicx}                  
\usepackage{cases}
\usepackage{indentfirst}                    
\usepackage{subfigure}                
\usepackage[colorlinks=true,linkcolor=black,citecolor=blue,urlcolor=blue]{hyperref}
\usepackage{table-fct}
\usepackage{diagbox} 
\usepackage{booktabs}
\usepackage{makecell}
\usepackage{pythonhighlight}
\usepackage{bm}                        
\usepackage{array, threeparttable}     
\usepackage{algorithmicx}
\usepackage[ruled]{algorithm}
\usepackage{algpseudocode}
\usepackage{mathrsfs} 
\usepackage{placeins}
\begin{document}
	
	\newtheorem{corollary}[theorem]{Corollary}
	\newtheorem{prop}[theorem]{Proposition}
	\newcommand{\figref}[1]{Fig.~\ref{#1}}
	\renewcommand{\proofname}{\textbf{Proof}}

\title{An adaptive time-stepping strategy for the modified phase field crystal model with a strong nonlinear vacancy potential}

\author[O.~Author]{Wanrong Hao, Yunqing Huang\corrauth}
\address{National Center for Applied Mathematics in Hunan, School of Mathematics and Computational Science, Xiangtan University,
Xiangtan 411105, China}
\email{{\tt haowanrong@smail.xtu.edu.cn(W.Hao), huangyq@xtu.edu.cn} (Y.Huang) }


\begin{abstract}
This paper develops three linear and energy-stable schemes for a modified phase field crystal model with a strong nonlinear vacancy potential (VMPFC model). This sixth-order phase-field model enables realistic crystal growth simulation.
Starting from a Crank-Nicolson scheme based on the stabilized-SAV (S-SAV) method, we optimize it via the generalized positive auxiliary variable (GPAV) and modified exponential scalar auxiliary variable (ESAV) methods, thereby reducing computational complexity or eliminating the requirement for the nonlinear free energy potential to be bounded from below.
The newly developed Energy-Variation Moving Average (EV-MA) adaptive time-stepping strategy resolves numerical instabilities and mitigates the high parameter sensitivity of the conventional adaptive time algorithm during rapid energy decay in the strongly nonlinear system. Unlike conventional instantaneous energy-derivative monitors, the EV-MA technique incorporates a moving average of the energy variation. Additionally, the rate of change between adjacent time steps is constrained by a maximum change factor. This design effectively dampens spurious oscillations and enhances the robustness of time step selection. Extensive numerical experiments are conducted to validate the accuracy and energy stability of the proposed schemes. The EV-MA strategy is also demonstrated to perform robustly across a wide range of parameters.
\end{abstract}


\maketitle

\section{Introduction}
\label{sec1}
Crystallization is a fundamental technique for purifying solid compounds. This process is typically induced by cooling a saturated solution, which reduces the solubility of solute and promotes the formation of pure crystals. A key stage in this process is crystal growth, during which atoms, ions, or molecules arrange themselves into a highly ordered crystalline lattice. This phenomenon is a fundamental phase transition, involving mass transfer from a liquid or vapor phase to a solid crystalline phase.
However, the growth process is susceptible to disturbances, often arising from the exclusion of impurities, which can introduce crystallographic defects such as vacancies, grain boundaries, and dislocations. Consequently, real-world crystals are never structurally perfect. These defects, when present in sufficient density, profoundly influence the mechanical, chemical, electronic, and physical properties of materials. Furthermore, the formation and dynamics of these defects are critical to many technological processes, including diffusion, oxidation, precipitation, and annealing.

To elucidate crystal growth mechanisms, numerous theoretical models have been developed and extensively used in simulations. Among these, the classical phase field crystal (PFC) model, proposed by Elder and Grant in \cite{1elder2004modeling,2elder2002modeling}, describes the dynamics of atomic-scale crystal growth on diffusive time scales.
The PFC model utilizes a phase field variable—representing a coarse-grained time-averaged atomic density field—to act as an order parameter for the liquid-to-crystal transition. This system captures the fundamental distinction between the phases: the nearly uniform density distribution in the liquid state and the 
periodic ordered structure in the crystalline solid.
The PFC model operates exclusively on diffusive timescales, inherently lacking the capacity to capture elastic interactions such as the deformation properties of nanocrystalline solids. 
To address this limitation, the authors in  \cite{3stefanovic2006phase,4stefanovic2009phase} extended the PFC model to a modified PFC (MPFC) model. This makes it possible to self-consistently simulate rapid elastic relaxation over large length scales. 
Formulated as a sixth-order, nonlinear, damped wave equation, the MPFC model captures the viscoelastic response to perturbations in the density field.

Vacancies, crucial point defects in crystalline materials, arise from insufficient atoms occupying lattice sites, profoundly influencing material properties. Traditional PFC and MPFC models fail to capture these defects. Their free energy functionals lack penalties for negative order parameter values, leading to perfectly periodic, vacancy-free predictions.
The MPFC model with strong nonlinear vacancy potential (VMPFC) proposed by Chan et al.\cite{5chan2009molecular} achieves a positive density constraint.
It introduces a strong nonlinear penalty potential $F_{vac}(\phi)$ into the free energy functional,  penalizing negative values of $\phi$ and thereby stabilizing vacancies. This allows the model to form periodic structures in some regions while maintaining near-zero or zero density in others, which are accurately interpreted as vacancies. Initially developed for elastic and plastic transformations in nanocrystalline materials, the VMPFC model has since proven valuable across diverse applications, including defect dynamics, grain coarsening, and solidification processes.

Numerous studies have focused on developing numerical approximations for the PFC model,  such as the convex splitting approach \cite{7wise2009energy,8hu2009stable,9shin2016first,10yang2017linearly}, the invariant energy quadratization (IEQ) approach \cite{16li2019efficient,17liu2020efficient}, and the scalar auxiliary variable (SAV) approach \cite{21li2020stability,22wang2019error,23yang2022linear,45Wang2021Error}.  
The convex splitting approach \cite{6eyre1998unconditionally} decomposes the energy functional into convex and concave terms, with explicit treatment of convex terms and implicit treatment of concave terms.
An unconditionally energy stable finite-difference scheme for the PFC model was developed in \cite{7wise2009energy} utilizing the convex splitting approach.
Hu et al.\cite{8hu2009stable} compared two energy stable finite-difference schemes: an improved one-step scheme based on a convex splitting \cite{7wise2009energy} and a new second-order two-step scheme.
Yang et al.\cite{10yang2017linearly} developed a series of linear, unconditionally energy stable numerical schemes for solving the PFC model.
The invariant energy quadratization (IEQ) approach was proposed by Yang et al. \cite{11yang2016linear,12yang2017efficient,15yang2017numerical}, and was successfully extended to the PFC model \cite{17liu2020efficient}, establishing unconditional energy stable schemes.
Shen et al. proposed the scalar auxiliary variable (SAV) approach  \cite{18shen2018convergence,19shen2018scalar,20shen2019new}, which has been widely used to construct linear energy stable numerical schemes\cite{40li2019energy,41liu2020two}. The SAV approach retains all advantages of the IEQ approach while overcoming most of its shortcomings.
In particular, the SAV approach only requires solving decoupled equations with constant coefficients at each time step and is not restricted to specific forms of the nonlinear part of the free energy.
It has been effectively applied in different spatial discretization schemes for the PFC model: Fourier-spectral SAV scheme \cite{21li2020stability}, finite difference SAV scheme \cite{23yang2022linear} and finite element SAV scheme \cite{22wang2019error}.
For the MPFC model, the convex splitting approach \cite{24baskaran2013energy,25baskaran2013convergence,26wang2011energy}, the IEQ approach \cite{16li2019efficient}, the SAV approach \cite{28li2022efficient,42qi2023error} and many other approaches \cite{27qian2025error,29guo2018high,30dehghan2016numerical} have also been used to construct energy stable numerical schemes.

However, numerical studies for the VMPFC model remain limited\cite{13lee2024linear,14yang2022simple,31zhang2019efficient,32pei2022efficient,33zhang2023highly}. A linear, energy-stable backward difference formula (BDF2) time marching scheme was developed in \cite{31zhang2019efficient}, combining the SAV approach with the traditional stabilization techniques. Three linear, unconditionally energy-stable numerical schemes were constructed in \cite{32pei2022efficient} using the stabilized-IEQ approach, though these schemes require the solution of elliptic systems with complicated variable coefficients at each time step. Through the combination of the modified ESAV approach and the relaxed SAV approach, first-order backward Euler and second-order BDF2 time-accurate schemes were developed in \cite{33zhang2023highly}.
There exist more challenges in designing time-accurate and energy-stable numerical schemes for the VMPFC model. The main difficulties arise from its strong nonlinear potential and stiffness issues caused by the nonlinear terms and higher-order differential terms. If the nonlinear vacancy term is discretized in a simple fully-implicit or explicit way, the resulting schemes become either energy-unstable or computationally expensive, rendering them inefficient in practice. However, the energy conservation law is vital to capture the correct dynamics and establish the existence and convergence of approximate solutions.

In this paper, we first employ the traditional SAV approach with stabilization techniques to construct an energy-stable Crank-Nicolson scheme for the VMPFC model.
Each time step involves solving two decoupled linear equations with constant coefficients, yielding a fast, non-iterative, and easily implementable scheme.
However, we find that the nonlocal auxiliary variable introduced by the traditional SAV approach must be bounded below.
To overcome the lower-bound constraint and enhance efficiency, we propose two improved methods: the Generalized Positive Auxiliary Variable (GPAV) and the modified Exponential Scalar Auxiliary Variable (ESAV) methods. Using these methods combined with stabilization techniques, we develop two decoupled CN schemes. We rigorously prove unconditional energy stability for all schemes. 
Based on the moving average of the energy variation, we then propose a novel adaptive time-stepping algorithm. To the authors' knowledge, this is the first adaptive time strategy for the VMPFC model.	
This algorithm can provide an effective adaptive time strategy for some models that are highly sensitive to varying time step sizes during rapid energy decay phases.
Finally, various 2D and 3D numerical experiments validate the energy stability and accuracy of the proposed schemes.

The rest of this paper is organized as follows. In Section 2, we introduce the VMPFC model and its corresponding energy law. In Section 3, we first consider the stabilized-SAV method for the VMPFC model and give a detailed proof of the unconditional energy stability. Then, the two improved methods based on GPAV and ESAV methods are proposed to construct energy stable schemes. We prove that these schemes are unconditionally energy stable, uniquely solvable, and mass conservative. 
In addition, a novel EV-MA adaptive time-stepping strategy is proposed.
In Section 4, some numerical experiments demonstrate the energy stability and accuracy of our proposed schemes, as well as the advantages of the novel adaptive algorithm. Some conclusions are presented in Section 5.

\section{Governing systems}
\label{sec2}
Let $(\cdot,\cdot)_{K}$ and $\|\cdot\|_{K}$ ($K\geq0$) denote the inner product and norm in the Sobolev space $H^{K}(\Omega)$, respectively. Note that $H^{0}(\Omega)=L^{2}(\Omega)$. We denote the $L^{2}$ inner product and its norm by $(\cdot,\cdot)$ and $\|\cdot\|$, respectively. We define the spaces $L_0^2\left( \Omega  \right) = \left\{ {\phi  \in {L^2}\left( \Omega  \right):\left( {\phi ,1} \right) = 0} \right\}$, $L_{per}^2\left( \Omega  \right)$$=\{\phi\in L^{2}(\Omega):\phi$ is periodic on $\partial\Omega\}$, $H_{per}^K\left(\Omega  \right)$$=\{\phi\in H^{K}(\Omega):\phi$ is periodic on $\partial\Omega\}$, The dual space of $H_{per}^K$ is denoted by $H_{per}^{-K}.$
The VMPFC model is derived from the following free energy functional:
\begin{equation}\label{2.1}
	E(\phi)=\int_{\Omega}{\Big(}\frac{1}{2}(\Delta\phi)^{2}-|\nabla\phi|^2+\frac{1}{2}\phi^2+F(\phi)+F_{vac}(\phi)\Big)d\bm{x},
\end{equation}
where $\Omega$ is a domain in $\mathbb{R}^d(d=1,2,3)$. $\phi(\bm{x},t)$ is introduced to describe the local atomic density field as a phase field variable. The double-well potential is given by $F(\phi)=\frac{1}{4}\phi^{4}-\frac{\epsilon}{2}\phi^{2}$, where $\epsilon\in(0,1)$.
$F_{vac}(\phi)=\frac{h_{vac}}{3}(|\phi|^{3}-\phi^{3})$ is the vacancy potential, which penalizes the negative values of $\phi$.
The penalization parameter $h_{vac}$ satisfies $h_{vac}\gg1$.
By taking the first-order derivatives of $F(\phi)$ and $F_{vac}(\phi)$ with respect to $\phi$,  we obtain $f\left( \phi  \right) = F'\left( \phi  \right) = {\phi ^3} - \varepsilon \phi $ and $ {f_{vac}}\left( \phi  \right) = {F'_{vac}}\left( \phi  \right) = {h_{vac}}\left( {\left| \phi  \right| - \phi } \right)\phi $.
The VMPFC model can be regarded as the $H^{-1}(\Omega)$ gradient flow of the energy functional \eqref{2.1}:
\begin{equation}\label{2.2}
	\alpha\phi_{tt}+\beta\phi_{t}=M\Delta\mu,
\end{equation}
\begin{equation}\label{2.3}
	\mu=(\Delta+1)^{2}\phi+f(\phi)+f_{vac}(\phi),
\end{equation} \\
where $\mu :=\frac{{\partial E\left( \phi  \right)}}{{\partial \phi }}$ is the chemical potential, \( M \) (with \( M > 0 \)) is the mobility constant, and \( \alpha \), \( \beta \) are non-negative parameters.
The system \eqref{2.2}-\eqref{2.3} degenerates back to:
(i) the MPFC model \cite{3stefanovic2006phase,4stefanovic2009phase} when \( h_{vac} = 0 \),
(ii) the classical PFC model \cite{1elder2004modeling,2elder2002modeling} when \( \alpha = h_{vac} = 0 \).
We adopt the periodic boundary conditions, as commonly done in studies of classical PFC model \cite{8hu2009stable,9shin2016first,10yang2017linearly} and the MPFC model \cite{16li2019efficient,24baskaran2013energy,27qian2025error}.
The theoretical and numerical analysis also applies to physical boundary conditions such as Neumann-type~\cite{28li2022efficient}, where boundary integrals vanish in the variational formulation. The initial conditions are given by:
\begin{equation}
	\phi(\bm{x},0)=\phi^{0} , \phi_{t}(\bm{x},0)=\phi_{t}^{0} =0.
\end{equation}

Subsequently, the conservation property of mass for the VMPFC model \eqref{2.2}-\eqref{2.3} holds provided that  $\psi^{0}\in L_0^2\left( \Omega  \right)$. In fact, by taking the $L^{2}$ inner product of \eqref{2.2} with 1, we get
$\alpha\frac{d}{dt}\int_{\Omega}\phi_{t}(\bm{x},t)d\bm{x}+\beta\int_{\Omega}\phi_{t}(\bm{x},t)d\bm{x}=0.$
Then, we have
$\int_{\Omega}\phi_{t}(\bm{x},t)dx=e^{\frac{-\beta}{\alpha}t}\int_{\Omega}\phi_{t}(\bm{x},0)d\bm{x}.$
So if $\psi^{0}\in L_0^2\left( \Omega  \right)$, we can obtain
\begin{equation}\label{2.5}
	\int_{\Omega}\phi_{t}(x,t)d\bm{x}=\int_{\Omega}\phi_{tt}(x,t)d\bm{x}=0.
\end{equation}

Now, we derive the energy dissipation law for the VMPFC model \eqref{2.2}-\eqref{2.3}, which is also an important property for dissipative system. The inverse Laplace operator $-\Delta^{-1}$ and $H_{per}^{-1}$ are defined as follows. Supposing $\phi\in L_0^2\left( \Omega  \right)$, we consider the periodic boundary value problem
$ -\Delta\varphi=\phi$  in $\Omega$
has a unique solution $\varphi$ in the Sobolev space ${H_{per}^2(\Omega)}\cap L_0^2\left( \Omega  \right)$. We define $\varphi:=-\Delta^{-1}\phi$. Suppose that $\phi_{1}$ and  $ \phi_{2}\in L_0^2\left( \Omega  \right)$, we define the $H_{per}^{-1}$ inner product by
$(\phi_{1},\phi_{2})_{-1}:=(\nabla\varphi_{1},\nabla\varphi_{2}),$
where  $\varphi_{1}=\triangle^{-1}\phi_{1}$, $\varphi_{2}=\triangle^{-1}\phi_{2}$, and $\|\phi\|_{-1}=\sqrt{(\phi,\phi)_{-1}}$. Thus, we have the following identity:
\begin{equation} (\phi_{1},\phi_{2})_{-1}=(-\Delta^{-1}\phi_{1},\phi_{2})=(\phi_{1},-\Delta^{-1}\phi_{2}).
\end{equation}
Thanks to a variable $\psi=\phi_{t}$, we can immediately obtain that $\int_{\Omega}\psi d\bm{x}=\int_{\Omega}\psi_{t}d\bm{x}=0$ from \eqref{2.5}. This implies $\psi_{t}\in L_0^2\left( \Omega  \right)$ and $\psi \in L_0^2\left( \Omega  \right)$. Plugging \eqref{2.3} back into \eqref{2.2} and taking the operator $\Delta^{-1}$, we get
\begin{equation} \label{2.7}
	\alpha\Delta^{-1}\psi_t+\beta\Delta^{-1}\psi=M\left((1+\Delta)^2\phi+f(\phi)+f_{vac}(\phi)\right).
\end{equation}
Taking the $L_{2}$ inner product of \eqref{2.7} with $\frac{1}{M}\phi_{t}$ gives the  energy dissipation law:
\begin{equation}\label{2.8}
	\frac{d}{dt}\widetilde{E}(\phi,\psi)=-\frac{\beta}{M}\|\psi\|_{-1}^{2}\leq0,
\end{equation}
where the pseudo energy $\widetilde{E}$ of the VMPFC model is defined as
\begin{equation}\label{2.9}
	\widetilde{E}=E(\phi)+\frac{\alpha}{2M}\|\psi\|_{-1}^{2}.
\end{equation}
This implies that the pseudo energy $\widetilde{E}$ is non-increasing with temporal evolution.

\section{Numerical schemes}
In this section, we will consider some efficient procedures to give three numerical schemes based on the Crank-Nicolson method for the VMPFC model.
Before giving the semi-discrete schemes, we give some definitions.
For a fixed time step $\Delta t > 0$, we define the final time as $T = N \Delta t$ and the discrete time points as $t^{n} = n \Delta t$ , where n ranges from 0 to N.
For a given continuous function $f(x,t)$, we denote $f^{n}$ as the numerical approximation of $f(x,t^{n})$ at time level $t^{n}$. We have the following mathematical definitions:
\begin{equation}
	f^{n+1/2}=\frac{f^{n+1}+f^{n}}{2},\quad f^{\dagger,n+1}=\frac{3}{2}f^{n}-\frac{1}{2}f^{n-1},\quad f^{\star, n+1}=2f^{n}-f^{n-1}.
\end{equation}

\subsection{The traditional S-SAV-CN scheme}
The primary computational challenge in solving the coupled system \eqref{2.2}-\eqref{2.3} stems from the treatment of the nonlinear terms $f(\phi)$ and $f_{vac}(\phi)$.
The recently popular SAV approach \cite{19shen2018scalar} provides valuable inspiration for constructing an efficient scheme that is linear, second-order accurate, and unconditionally energy stable. The key idea of the SAV approach is to define the square root of the nonlinear free energy terms in the energy functional as a simple scalar auxiliary variable. Then, by rewriting the system, the handling of nonlinear terms becomes easier. The advantage is that at each time step, only a few linear equations with constant coefficients need to be solved.
To develop a novel temporal discretization scheme for the system \eqref{2.2}-\eqref{2.3} using the SAV approach, we introduce an auxiliary, nonlocal, time-dependent function $u(t)$ defined as
\begin{equation}\label{3.2}
	u(t)=\sqrt{\int_{\Omega}F(\phi)+F_{vac}(\phi)dx +b},
\end{equation}
where $b$ is a constant ensuring radicand is strictly positive. In other words, this approach requires the assumption that $\int_{\Omega}F(\phi)+F_{vac}(\phi)dx$ is bounded from below.
Through this well-defined variable $u(t)$, the pseudo energy functional 	$\widetilde{E}$ in \eqref{2.9} can be reformulated as the following modified energy:
\begin{equation}\label{3.3}
	\mathcal{E}(u)=\int_{\Omega}(\frac{1}{2} |(\Delta+1)\phi|^2dx+\frac{\alpha}{2M}\|\psi\|_{-1}^{2}+u^{2}-b.
\end{equation}
We rewrite the original VMPFC model \eqref{2.2}-\eqref{2.3} as the following equivalent system:
\begin{equation}\label{3.4}
	\alpha\psi_{t}+\beta\psi=M\Delta\mu,
\end{equation}
\begin{equation}\label{3.5}
	\mu=(\Delta+1)^{2}\phi+Hu,
\end{equation}
\begin{equation}\label{3.6}
	\psi=\phi_{t},
\end{equation}
\begin{equation}\label{3.7}
	u_{t}=\frac{1}{2}\int_{\Omega}H\phi_{t}dx,
\end{equation}
where
\begin{equation}
	H=\frac{f(\phi)+f_{vac}(\phi) }{\sqrt{\int_{\Omega}F(\phi)+F_{vac}(\phi)dx + b}}.
\end{equation}
The system is subject to the periodic boundary conditions and the initial conditions:
\begin{equation}
	\phi(\bm{x},0)=\phi^{0},\psi(\bm{x},0)=\psi^{0}=0,\quad u(t=0)=\sqrt{\int_{\Omega}F(\phi^{0})+F_{vac}(\phi^{0})dx + b }.
\end{equation}

The new system \eqref{3.4}-\eqref{3.7} is also unconditionally energy stable, satisfying an energy dissipation law. By taking the operator $\Delta^{-1}$ on \eqref{3.4} and inserting the expression \eqref{3.5} into \eqref{3.4}, we get
$$\frac{\alpha}{M}\Delta^{-1}\psi_{t}+\frac{\beta}{M}\Delta^{-1}\psi=(\Delta+1)^{2}\phi+Hu.$$
Taking the $L^{2}$ inner product with $\phi_{t}$, combined with \eqref{3.6} and \eqref{3.7}, yields
\begin{equation}
	\frac{d}{dt}\mathcal{E}(u)=-\frac{\beta}{M}\|\psi\|_{-1}^{2}\leq0.
\end{equation}

Combining the SAV approach with traditional stabilization techniques, the second-order stabilized-SAV scheme for solving the equivalent system \eqref{3.4}-\eqref{3.7} can be constructed using the Crank-Nicolson (CN) method.
Given the previously computed values ($\phi^{n-1}$, $\psi^{n-1}$, $u^{n-1}$) and ($\phi^{n}$, $\psi^{n}$, $u^{n}$), the values ($\phi^{n+1}$, $\psi^{n+1}$, $u^{n+1}$) at the next time step are computed using the following S-SAV-CN scheme:
\begin{equation}\label{3.11}
	\alpha \frac{\psi^{n+1}-\psi^{n}}{\Delta t}+\beta\psi^{n+1/2}=M\Delta\mu^{n+1/2},
\end{equation}
\begin{equation}\label{3.12}
	\mu^{n+1/2}=(\Delta+1)^{2}\phi^{n+1/2}+H^{\dagger ,n+1}u^{n+1/2}+S(\phi^{n+1}-\phi^{\star ,n+1}),
\end{equation}
\begin{equation}\label{3.13}
	\psi^{n+1/2}=\frac{\phi^{n+1}-\phi^{n}}{\Delta t},
\end{equation}
\begin{equation}\label{3.14}
	u^{n+1}-u^{n}=\frac{1}{2}\int_{\Omega}H^{\dagger, n+1}(\phi^{n+1}-\phi^{n})dx,
\end{equation}
where $H^{\dagger,n+1}=H(\phi^{\dagger, n+1})$, and $S$ is a positive stabilization parameter.
\begin{remark}\label{sav-re1}
	We need to solve ($\phi^{1}$,$\psi^{1}$,$u^{1}$), which can be obtained using the following first-order scheme based on backward-Euler formulation:
	\begin{equation}\label{3.15}
		\alpha \frac{\psi^{1}-\psi^{0}}{\Delta t}+\beta\psi^{1}=M\Delta\mu^{1},
	\end{equation}
	\begin{equation}\label{3.16}
		\mu^{1}=(\Delta+1)^{2}\phi^{1}+H^{0}u^{1}+S(\phi^{1}-\phi^{0}),
	\end{equation}
	\begin{equation}\label{3.17}
		\psi^{1}=\frac{\phi^{1}-\phi^{0}}{\Delta t},
	\end{equation}
	\begin{equation}\label{3.18}
		u^{1}-u^{0}=\frac{1}{2}\int_{\Omega}H^{0}(\phi^{1}-\phi^{0})dx,
	\end{equation}
	where $H^{0}=H(\phi^{0})$, and $S$ is a positive stabilization parameter. 
\end{remark}
\begin{remark}\label{sav-re2}
	We introduce two stabilization terms $S(\phi^{1}-\phi^{0})$ and $S(\phi^{n+1}-\phi^{\star ,n+1})$ in the numerical schemes \eqref{3.11}-\eqref{3.14} and \eqref{3.15}-\eqref{3.18}.  	
	This method is a widely used stabilization method in phase field model numerical calculations (see \cite{34feng2015long,36xu2006stability,37shen2010numerical,31zhang2019efficient,32pei2022efficient,33zhang2023highly}).	
	The two introduced numerical errors are of orders $S\Delta t\phi_{t}(\cdot)$ and $S\Delta t^{2}\phi_{tt}(\cdot)$, matching the orders of the errors introduced by the explicit treatment of the nonlinear term.	   
	The stabilization term is crucial for energy stability, especially when $h_{vac}$ and $\Delta t$ are relatively large, as demonstrated in our later numerical experiments. To get a balance between
	stability and accuracy, one could take $S \sim h_{\text{vac}} O\left( \left|\min_{x \in \Omega} \left(\phi(x, t), 0 \right)\right|\right)$
	(cf. \cite{31zhang2019efficient,32pei2022efficient})
\end{remark}

\begin{theorem}\label{sav-th1}The S-SAV-CN scheme \eqref{3.11}-\eqref{3.14} satisfies mass conservation.
\end{theorem}
\begin{proof} By taking the $L^{2}$ inner product of \eqref{3.15} with
	1 and using integration by parts, we obtain $(\psi^1,1)=\frac{\alpha}{\alpha+\beta\Delta t}(\psi^0,1).$
	It holds $(\psi^{1}, 1)=0$ since $(\psi^{0}, 1)=0$.	
	Then from \eqref{3.17}, we obtain $(\phi^1,1)=(\phi^0,1).$
	Taking the $L^{2}$ inner product of \eqref{3.11} with 1 and using integration by parts, it holds that
	$$(\psi^{n+1},1)=\frac{2\alpha-\beta\Delta t}{2\alpha+\beta\Delta t}(\psi^{n},1)=0\quad , n\geq1.$$
	Considering \eqref{3.13}, we obtain
	$$(\phi^{n+1}-\phi^n,1)=0 \quad \text{if and only if} \quad \begin{cases} (\psi^{n+1},1)=0 & \text{for } n=0 \\ (\psi^{n+1/2},1)=0 & \text{for } n\geq 1 \end{cases}$$	Therefore,
	\begin{equation}
		\int_{\Omega}\phi^{n+1}dx=\int_{\Omega}\phi^{n}dx=\cdots=\int_{\Omega}\phi^{0}dx\  \text{for any } n,
	\end{equation}
	which implies the mass conservation property of the scheme \eqref{3.11}-\eqref{3.14}.
\end{proof}

\begin{theorem}\label{sav-th2}The second-order S-SAV-CN scheme \eqref{3.11}-\eqref{3.14} is  uniquely  solvable and unconditionally energy stable in the sense that it satisfies the
	following discrete energy dissipation law:
	\begin{equation}\label{3.20}
		E_{CN}^{n+1}(\phi^{n+1},u^{n+1})\leq E_{CN}^{n}(\phi^{n},u^{n})-\delta t\frac{\beta}{M}\|\frac{\psi^{n+1}+\psi^{n}}{2}\|_{-1}^{2},
	\end{equation}
	where, for any integer $k\geq0$, the discrete energy $E_{CN}^{n}(\phi^{n},u^{n})$ is defined as
	\begin{equation}\label{3.21}
		E_{CN}^{n}(\phi^{n},u^{n})=\frac{1}{2}(\Delta+1)^{2}\|\phi^{n}\|^{2}+\frac{S}{2}\|\phi^{n}-\phi^{n-1}\|^{2}+(u^{n})^{2}+\frac{\alpha}{2M}\|\psi^{n}\|_{-1}^{2}.
	\end{equation}
\end{theorem}
\begin{proof} We begin by proving the uniqueness of the solution for the S-SAV-CN scheme defined by equations \eqref{3.11}-\eqref{3.14}. We rewrite \eqref{3.13} and \eqref{3.14} as follows:
	\begin{equation}
		\psi^{n+1}=\frac{2}{\delta t}\phi^{n+1}+g_{1}^{n}  \ , \ u^{n+1}=\frac{1}{2}\int_{\Omega}H^{\dagger, n+1}\phi^{n+1}dx+g_{2}^{n},
	\end{equation}
	where
	$g_{1}^{n}=-\frac{2}{\delta t}\phi^{n}-\psi^{n}$, $g_{2}^{n}=-\frac{1}{2}\int_{\Omega}H^{\dagger, n+1}\phi^{n}dx+u^{n}.$ Then, \eqref{3.11} and \eqref{3.12} can be written as:
	\begin{equation}\label{3.23}
		P(\phi^{n+1})-\frac{1}{4}\Delta H^{\dagger,n+1}\int_{\Omega}H^{\dagger,n+1}\phi^{n+1}dx=\widetilde{g}^{n},
	\end{equation}
	where $P(\phi)=\frac{2}{M\delta t}(\frac{\alpha}{\delta t}+\frac{\beta}{2})\phi-\Delta(\Delta+1)^{2}\frac{\phi}{2}-S\Delta\phi,$ and	
	$$\widetilde{g}^{n}=\psi^{n}(\frac{\alpha}{M\delta t}-\frac{\beta}{2M})+\Delta(\Delta+1)^{2}\frac{\phi_{}^{n}}{2}-(\frac{\alpha}{M\delta t}+\frac{\beta}{2M})g_{1}^{n}+\frac{1}{2}\Delta H^{\dagger,n+1}g_{2}^{n}-S\Delta\phi^{\star,n+1}.$$
	We define a linear operator $P^{-1}$
	such that for any $f\in L^{2}(\Omega)$, $\phi=P^{-1}(f)$ is defined by $P(\phi)=f$. Applying $P^{-1}$ to both sides of \eqref{3.23} readily yields:
	\begin{equation}\label{3.24}
		\phi^{n+1}-\frac{1}{4}P^{-1}(\Delta H^{\dagger,n+1})\int_{\Omega}H^{\dagger,n+1}\phi^{n+1}dx=P^{-1}(\widetilde{g}^{n}).
	\end{equation}
	Taking the $L^{2}$ inner product of both sides of \eqref{3.24} with $H^{\dagger,n+1}$, we get
	\begin{equation}
		(H^{\dagger,n+1},\phi^{n+1})-\frac{1}{4}(P^{-1}(\Delta H^{\dagger,n+1}),H^{\dagger,n+1})\int_{\Omega}H^{\dagger,n+1}\phi^{n+1}dx=(P^{-1}(\widetilde{g}^{n}),H^{\dagger,n+1}).
	\end{equation}
	It follows that
	\begin{equation}
		(H^{\dagger,n+1},\phi^{n+1})=\frac{(P^{-1}(\widetilde{g}^{n}),H^{\dagger,n+1})}{1-\frac{1}{4}(P^{-1}(\Delta H^{\dagger,n+1}),H^{\dagger,n+1})}.
	\end{equation}
	where $P^{-1}(-\Delta)$ is positive definite operator. To summarize, the S-SAV-CN scheme \eqref{3.11}-\eqref{3.14} can be easily implemented as follows:
	\begin{itemize}
		\item[(i)] We compute ($P^{-1}(\widetilde{g}^{n}),H^{\dagger,n+1})$. This requires solving a sixth order equation
		$(P^{-1}(\widetilde{g}^{n}))=\psi_{1}$ with the periodic boundary conditions;
		\item[(ii)]  We compute $(P^{-1}(\Delta H^{\dagger,n+1}),H^{\dagger,n+1})$. This requires solving a sixth order equation $(P^{-1}(\Delta H^{\dagger,n+1}))=\psi_{2}$ with the periodic boundary conditions;
		\item [(iii)] We compute $\phi^{n+1}$ from \eqref{3.24} and $\psi^{n+1}$ from \eqref{3.13}, given the known values of $(H^{\dagger,n+1},\phi^{n+1})$, $\psi_{1},$ and $\psi_{2}$.
	\end{itemize}
	Hence, instead of solving equations that might need some costly iterative solvers \cite{32pei2022efficient}, the total cost of solving scheme
	\eqref{3.11}-\eqref{3.14} at each time step is just solving two decoupled sixth-order linear equations with constant coefficients.	
	Next, we  prove the unconditional energy stability of the scheme \eqref{3.11}-\eqref{3.14} as follows.
	By applying $\frac{1}{M}\Delta^{-1}$ to \eqref{3.11} and using \eqref{3.12}, we can obtain
	\begin{equation}\label{3.27}
		\frac{\alpha}{M\delta t}\Delta^{-1}(\psi^{n+1}-\psi^{n})+\frac{\beta}{2M}\Delta^{-1}(\psi^{n+1}+\psi^{n})=(\Delta+1)^{2}\frac{\phi^{n+1}+\phi^{n}}{2}+H^{\dagger ,n+1}\frac{u^{n+1}+u^{n}}{2}+S(\phi^{n+1}-\phi^{\star ,n+1}).			
	\end{equation}
	By taking the $L^{2}$ inner product of the above equation with $\phi^{n+1}-\phi^{n}$, we can obtain
	\begin{equation}\label{3.28}
		\begin{split}
			&(\frac{\alpha}{M\delta t}\Delta^{-1}(\psi^{n+1}-\psi^{n})+
			\frac{\beta}{2M}\Delta^{-1}(\psi^{n+1}+\psi^{n}),\phi^{n+1}-\phi^{n})=\\
			&((\Delta+1)^{2}\frac{\phi^{n+1}+\phi^{n}}{2},\phi^{n+1}-\phi^{n})+(H^{\dagger,n+1}\frac{u^{n+1}+u^{n}}{2},\phi^{n+1}-\phi^{n})+S(\phi^{n+1}-\phi^{\star,n+1},\phi^{n+1}-\phi^{n}).
		\end{split}
	\end{equation}
	Using integration by parts and \eqref{3.13}, the left term of \eqref{3.28} can be rewritten as
	\begin{equation}\label{3.29}
		\begin{split}
			&(\frac{\alpha}{M\delta t}\Delta^{-1}(\psi^{n+1}-\psi^{n})+\frac{\beta}{2M}\Delta^{-1}(\psi^{n+1}+\psi^{n}),\phi^{n+1}-\phi^{n})\\	
			&=-\frac{\alpha}{2M}(-\Delta^{-1}(\psi^{n+1}-\psi^{n}),\psi^{n+1}+\psi^{n})-\frac{\beta \delta t}{4M}(-\Delta^{-1}(\psi^{n+1}+\psi^{n}),\psi^{n+1}+\psi^{n})\\
			&=-\frac{\alpha}{2M}(\|\psi^{n+1}\|_{-1}^{2}-\|\psi^{n}\|_{-1}^{2})-\frac{\beta \delta t}{4M}\|\psi^{n+1}+\psi^{n}\|_{-1}^{2}.
		\end{split}
	\end{equation}
	Using integration by parts, \eqref{3.14}, and the identity $(a-2b+c)(a-b)=\frac{1}{2}((a-b)^{2}-(b-c)^{2}+(a-2b+c)^{2})$, the right term of \eqref{3.28} can be rewritten as
	\begin{equation}\label{3.30}
		\begin{split}
			&((\Delta+1)^{2}\frac{\phi^{n+1}+\phi^{n}}{2},\phi^{n+1}-\phi^{n})+(H^{\dagger ,n+1}\frac{u^{n+1}+u^{n}}{2},\phi^{n+1}-\phi^{n})+S(\phi^{n+1}-\phi^{\star ,n+1},\phi^{n+1}-\phi^{n})\\
			&=\frac{1}{2}(\Delta+1)^{2}(\|\phi^{n+1}\|^{2}-\|\phi^{n}\|^{2})+((u^{n+1})^{2}-(u^{n})^{2})\\
			&+\frac{S}{2}(\|\phi^{n+1}-\phi^{n}\|^{2}-\|\phi^{n}-\phi^{n-1}\|^{2}+\|\phi^{n+1}-2\phi^{n}+\phi^{n-1}\|^{2}),\\
		\end{split}
	\end{equation}
	Combining the \eqref{3.28}-\eqref{3.30}, we arrive at
	\begin{equation}
		\begin{split}
			&\frac{S}{2}(\|\phi^{n+1}-\phi^{n}\|^{2}-\|\phi^{n}-\phi^{n-1}\|^{2}+\|\phi^{n+1}-2\phi^{n}+\phi^{n-1}\|^{2})	+\frac{(\Delta+1)^{2}}{2}(\|\phi^{n+1}\|^{2}-\|\phi^{n}\|^{2})+((u^{n+1})^{2}-(u^{n})^{2})\\
			&+\frac{\alpha}{2M}(\|\psi^{n+1}\|_{-1}^{2}-\|\psi^{n}\|_{-1}^{2})=-\frac{\beta \delta t}{4M}\|\psi^{n+1}+\psi^{n}\|_{-1}^{2},
		\end{split}
	\end{equation}
	which implies the desired result \eqref{3.20} after we drop a positive term.
\end{proof}

\subsection{The S-GPAV-CN scheme}
In this subsection, inspired by the generalized positive auxiliary
variable (GPAV) numerical approach \cite{38yang2020roadmap,27qian2025error,44qian2025decoupled,46qian2023stability}, we propose a new, linear, and unconditionally stable scheme for the VMPFC model. 
The principal computational advantage is that the number of equations to be solved per time step is half that of the SAV method \eqref{3.11}-\eqref{3.14}.
We redefine the energy functional and introduce the following auxiliary variable:
\begin{equation}\label{3.32}
	R(t) = \sqrt{E_{1}(t)},
\end{equation}
where
\begin{equation}\label{3.33}
	E_{1}(t) = E_{1}[\phi,\psi] = \widetilde{E}(\phi,\psi) + \int_{\Omega} c_0 dx = \int_{\Omega} \left( \frac{1}{2} |(\Delta+1)\phi|^2  + F(\phi)+ F_{vac}(\phi) + c_0 \right) dx + \frac{\alpha}{2M} \|\psi\|_{-1}^2.
\end{equation}
The constant $c_0$ is chosen such that $E_{1}(t) > 0$ for $0 \leq t \leq T$. By taking derivative of \eqref{3.32} with respect to t and considering \eqref{2.9}, we obtain the following energy equation:
\begin{equation}\label{3.34}
	\frac{dR}{dt} = \frac{1}{2R} \frac{dE_{1}}{dt} = \frac{1}{2R} \frac{d\widetilde{E}(\phi,\psi)}{dt} = -\frac{\beta}{2MR} \|\psi\|_{-1}^2 = -\frac{\beta}{2M\sqrt{E_{1}}} \|\psi\|_{-1}^2.
\end{equation}
According to the above definition of $R(t)$ in \eqref{3.32}, we can define $\xi = \frac{R}{\sqrt{E_{1}}}$.
It is easy to obtain that $\xi\equiv1$  in the time-continuous case. The system \eqref{2.2}-\eqref{2.3} can be reformulated as the following equivalent system:
\begin{equation}\label{3.35}
	\alpha \psi_t + \beta \psi = M \Delta \mu, 
\end{equation}
\begin{equation}
	\mu = (\Delta+1)^2 \phi + \xi(f(\phi)+ f_{vac}(\phi)),
\end{equation}
\begin{equation}	
	\psi = \phi_t, 
\end{equation}
\begin{equation}\label{3.38}
	\frac{dR}{dt} = -\frac{\beta \xi}{2M\sqrt{E_{1}}} \|\psi\|_{-1}^2.
\end{equation}
The boundary conditions of the system \eqref{3.35}-\eqref{3.38} are the same as the original system \eqref{2.2}-\eqref{2.3} and the initial conditions are given by
\begin{equation}\label{3.39}
	\phi(\bm{x},0)=\phi^{0},\psi(\bm{x},0)=\psi^{0}=0,\quad R(t=0) = \sqrt{E_{1}(0)} = \sqrt{E_{1}[\phi^{0}, \psi^{0}]}.
\end{equation}
The auxiliary variable $R$ in ODE \eqref{3.38}, which approximates the energy \eqref{3.33}, obeys the dissipation law \eqref{3.34}. 
The dynamics of $R(t)$ are closely related to the pseudo energy dissipation law \eqref{2.8}.
As a result, the pseudo energy functional $\widetilde{E}(\phi, \psi)$ in \eqref{2.9} admits a reformulation as the modified energy:
\begin{equation}\label{3.40}
	\mathcal{E}(R) = R^2 - \int_{\Omega} c_0 dx.
\end{equation}

Next, The above equivalent system \eqref{3.35}-\eqref{3.38} is very easy to construct a linear, second-order, sequentially solved and unconditionally stable semi-implicit scheme. 
Given $(\phi^n, \psi^n, R^n)$ and $(\phi^{n-1}, \psi^{n-1}, R^{n-1})$ ($n \geq 1$), ($\phi^{n+1}, \psi^{n+1}$, $R^{n+1}$) can be obtained via the following second-order S-GPAV-CN scheme:
\begin{equation}\label{3.41}
	\alpha\frac{\psi^{n+1}-\psi^n}{\Delta t} + \beta \psi^{n+1/2} = M \Delta \mu^{n+1/2},
\end{equation}	
\begin{equation}\label{3.42}
	\mu^{n+1/2} = (\Delta + 1)^2 \phi^{n+1/2} + \xi_1^{n+1/2} (f(\phi^{\dagger, n+1})+ f_{vac}(\phi^{\dagger, n+1}))+S(\phi^{n+1}-\phi^{\star ,n+1}),
\end{equation}	
\begin{equation}\label{3.43}
	\psi^{n+1/2}=\frac{\phi^{n+1}-\phi^{n}}{\Delta t},
\end{equation}	
\begin{equation}\label{3.44}	
	\frac{R^{n+1}-R^n}{\Delta t} = -\frac{\beta \xi_2^{n+1}}{2M \sqrt{E_{1}[\phi^{n+1/2}, \psi^{n+1/2}]}} ||\psi^{n+1/2}||_{-1}^2,
\end{equation} 	
where $\xi_1^{n+1/2}$ and $\xi_2^{n+1}$ are second-order approximations of $\frac{R(t)}{\sqrt{E_{1}(t)}} = 1$ at $t = (n+1/2)\Delta t$ and $t = (n+1)\Delta t$, respectively, defined by
\begin{equation}\label{3.45}	
	\xi_1^{n+1/2} = \frac{{R}^{\dagger,n+1}}{\sqrt{E_{1}[\phi^{\dagger,n+1}, \psi^{\dagger,n+1}]}}, \xi_2^{n+1} = \frac{R^{n+1}}{\sqrt{E_{1}[\phi^{n+1}, \psi^{n+1}]}}.
\end{equation} 
Noting that the $\phi^{\dagger, n+1}$ is any explicit $O(\Delta t^2)$ approximation for $\phi(t^{n+\frac{1}{2}})$ and $R^{\dagger, n+1}$ \eqref{3.45} is any explicit $O(\Delta t^2)$ approximation for $R(t^{n+1/2})$.

\begin{remark}\label{sav-re3}
	For $n=0$, the values of ($\phi^{1}$, $\psi^{1}$, $R^{1}$) can be obtained by the first-order scheme:
\end{remark}

\begin{equation}\label{3.46}
	\alpha\frac{\psi^1 - \psi^0}{\Delta t} + \beta\psi^1 = M\Delta\mu^1,
\end{equation} 	
\begin{equation}\label{3.47}	
	\mu^1 = (\Delta+1)^2\phi^1 + \frac{R^0}{\sqrt{E_{1}[\phi^0,\psi^0]}}(f(\phi^{0})+ f_{vac}(\phi^{0}))+S(\phi^{1}-\phi^{0}), 
\end{equation}  
\begin{equation}\label{3.48}	
	\frac{\phi^1 - \phi^0}{\Delta t} = \psi^1, 
\end{equation} 	
\begin{equation}\label{3.49}	
	\frac{R^1 - R^0}{\Delta t} = -\frac{\beta R^1}{2ME_{1}[\phi^1,\psi^1]}\Vert\psi^1\Vert_{-1}^2.
\end{equation} 
\begin{theorem}\label{sav-th3}
	The S-GPAV-CN scheme \eqref{3.41}- \eqref{3.44} maintains mass conservation as 
	\begin{equation}
		\int_{\Omega}\phi^{n+1}dx=\int_{\Omega}\phi^{n}dx=\cdots=\int_{\Omega}\phi^{0}dx\  \text{for any } n.
	\end{equation}
\end{theorem}
\begin{proof} 
	The proof is similar to Theorem~\ref{sav-th1}
\end{proof}
\begin{theorem}\label{sav-th4}The second-order S-GPAV-CN scheme \eqref{3.41}-\eqref{3.44} admits a unique solution. Given $R^0 > 0$, it is unconditionally energy stable for all time steps $n$, in the following sense
	\begin{equation}	
		0 < R^{n+1} \leq R^n.
	\end{equation} 
\end{theorem}

\begin{proof} We first prove that the S-GPAV-CN scheme \eqref{3.41}-\eqref{3.44} admits a unique solution. The value of $E_{1}[\phi^{\dagger,n+1}, \psi^{\dagger,n+1}]$  can be computed using \eqref{3.33}. Subsequently, we compute $R^{\dagger,n+1}$ and $\xi_1^{n+1/2}$ from \eqref{3.32} and  \eqref{3.45}.
	\eqref{3.41}-\eqref{3.42} can be rewritten as
	\begin{equation}
		\left(\alpha+\frac{\Delta t\beta}{2}\right)\frac{2}{\Delta t}\phi^{n+1}=\hat{g}^n+\Delta tM\Delta\mu^{n+\frac{1}{2}},
	\end{equation}
	\begin{equation}
		\mu^{n+\frac{1}{2}}=(1+\Delta)^{2}\phi^{n+\frac{1}{2}}+\xi_1^{n+1/2}(f(\phi^{\dagger, n+1})+ f_{vac}(\phi^{\dagger, n+1}))+S(\phi^{n+1}-\phi^{\star ,n+1}),
	\end{equation}
	where $\hat{g}^{n}= 2\alpha \psi^{n}+ (\frac{2\alpha}{\Delta t}+ \beta)\phi^{n}.$ Then, we have
	\begin{equation}\label{3.54}
		\mathcal{M}(\phi^{n+1})= g^n,
	\end{equation}
	where
	\begin{equation}
		\begin{cases}
			\mathcal{M}(\phi) = \left( \alpha + \frac{\Delta t \beta}{2} \right) \frac{2}{\Delta t} \phi - \frac{\Delta t M}{2} \Delta (1 + \Delta)^2 \phi -\Delta t M S\Delta\phi, \\
			g^n = \hat{g}^{n} + \frac{\Delta t M}{2} \Delta (1 + \Delta)^2 \phi^n + \Delta t M\xi_1^{n+1/2} \Delta (f(\phi^{\dagger, n+1})+ f_{vac}(\phi^{\dagger, n+1}))-\Delta t MS\Delta\phi^{\star ,n+1}.
		\end{cases}
	\end{equation}
	Next, once the variable $\xi_1^{n+1/2}$ are determined, $\phi^{n+1}$ can be computed by \eqref{3.54} with periodic boundary conditions.	
	Since the linear operator $\mathcal{M}$ is a positive definite operator, \eqref{3.54} has a unique solution $\phi^{n+1}$.
	Then, $\psi^{n+1/2}$ and $R^{n+1}$ can be obtained directly by \eqref{3.43}-\eqref{3.45}.
	The computation is totally decoupled and only one sixth-order linear equation with constant coefficients needs to be solved at each time step, whereas the classical SAV approach requires solving two. Next, we prove the unconditional energy stability of scheme \eqref{3.41}-\eqref{3.44}. The stability pertains to a discrete modified energy, rather than the original free energy \eqref{2.1} or the pseudo energy \eqref{2.9}. From \eqref{3.49}, one finds
	$$R^1 = \frac{R^0}{1 + \frac{\beta \Delta t}{2M E_{1}[\phi^1, \psi^1]} \| \psi^1 \|_{-1}^2} \leq R^0.$$
	By inserting the second expression from \eqref{3.45} into \eqref{3.44}, one gets\\
	$$
	R^{n+1} = \frac{R^n}{1 + \frac{\beta \Delta t}{2M \sqrt{E_{1}[\phi^{n+1/2}, \psi^{n+1/2}] \sqrt{E_{1}[\phi^{n+1}, \psi^{n+1}]}}} \| \psi^{n+1/2} \|_{-1}^2} \quad n \geq 1.
	$$
	Based on $R^0 > 0$, one can conclude by induction that $R^n > 0$ for all $n$ and $R^{n+1} \leq R^n$.
\end{proof}
\begin{remark}\label{esav-re6}	
	The S-GPAV-CN scheme \eqref{3.41}-\eqref{3.44} can be implemented as follows:
	\begin{enumerate}
		\item[(i)] Compute $\phi^{1}$ from \eqref{3.46}-\eqref{3.47}, $\psi^{1}$ from \eqref{3.48}, $R^{1}$ from \eqref{3.49}, and set n=0;
		\item[(ii)] set n=n+1;
		\item[(iii)]  Compute $\xi_1^{n+1/2}$ from \eqref{3.45}, $\phi^{n+1}$ from \eqref{3.41}-\eqref{3.42}, $\psi^{n+1}$ and $\psi^{n+1/2}$ from \eqref{3.43};
		\item[(iv)] Compute $\xi_2^{n+1}$ from \eqref{3.45}; compute $R^{n+1}$ from \eqref{3.44} and go back to step (ii).
	\end{enumerate}	
\end{remark}
\subsection{The S-ESAV-CN scheme}
In this subsection,  we propose a new and unconditional energy stable scheme for the system \eqref{2.2}-\eqref{2.3} of VMPFC model based on the modified exponential SAV (E-SAV) \cite{39liu2020exponential} approach.
By leveraging the inherent positivity of exponential functions, the improved scheme removes the essential lower-bound restriction required by the SAV and GPAV methods, thus being more suitable for dissipative systems. Specifically, we define a time-dependent variable $B$ with the following form:
\begin{equation}\label{3.56}
	B=B(t)=\exp\left(\frac{\widetilde{E}(\phi)}{C}\right),
\end{equation}
where $C\geq \widetilde{E}(\phi^{0})$. $C$ is a large positive constant used to decelerate the rapid growth of the exponential function and $\widetilde{E}$ is the pseudo energy in \eqref{2.9}. It is obvious that $B> 0$ for any t. 
We introduce the assumption-free exponential scalar auxiliary variable $B$, thereby removing the restriction that the nonlinear part of the total energy must be bounded from below. Then, the nonlinear function $F^{\prime}(\phi)+F_{vac}^{\prime}(\phi)$ can be transformed as:
$F^{\prime}(\phi)+F_{vac}^{\prime}(\phi)=\frac{B}{B}(F^{\prime}(\phi)+F_{vac}^{\prime}(\phi))=\frac{B}{\exp\left(\frac{\widetilde{E}}{C}\right)}(f(\phi)+f_{vac}(\phi)).$
By taking the derivative of \eqref{3.56} with respect to t, we obtain
\begin{equation}\label{3.57}
	\frac{dB}{dt}=\frac{1}{C}\exp\left(\frac{\widetilde{E}(\phi)}{C}\right)\frac{d\widetilde{E}(\phi)}{dt}=-\frac{B\beta}{CM}\|\psi\|_{-1}^{2}.
\end{equation}
Thus, \eqref{2.2}-\eqref{2.3} can be rewritten as the following equivalent system:
\begin{equation}\label{3.58}
	\alpha \psi_{t}+\beta\psi=M\Delta\mu,
\end{equation}
\begin{equation}
	\mu=(\Delta+1)^{2}\phi+\xi (f(\phi)+f_{vac}(\phi)),
\end{equation}
\begin{equation}\label{3.60}
	\psi=\phi_{t},
\end{equation}
\begin{equation}\label{3.61}
	\frac{dB}{dt}=-\frac{B\beta}{CM}\|\psi\|_{-1}^{2},
\end{equation}
where the new variable $\xi=\frac{B}{\exp\left(\frac{\widetilde{E}}{C}\right)}\equiv1$  in the time-continuous
case. The boundary conditions of the system \eqref{3.58}-\eqref{3.61} are the same as the original system \eqref{2.2}-\eqref{2.3} and the initial conditions are given by
\begin{equation}\label{3.62}
	\phi(\bm{x},0)=\phi^{0},\psi(\bm{x},0)=\psi^{0}=0,B(t=0)=B^0=\exp\left(\frac{\widetilde{E}(\phi^{0})}{C}\right).
\end{equation}
It should be noted that \eqref{3.57} is also an ODE system in time, thus $B$ does not require boundary conditions. The variable $B$ computed in \eqref{3.61} serves as an approximation of the energy \eqref{2.9} and obeys the dissipation law \eqref{3.57}. It is obvious that the pseudo energy $\widetilde{E}$ can be reformulated as the following modified energy:
\begin{equation}\label{3.63}
	\mathcal{E}(B)= Cln(B).
\end{equation}

A decoupled, linear, unconditionally energy-stable semi-implicit scheme is introduced to solve the equivalent system \eqref{3.58}-\eqref{3.61} based on the stabilized-ESAV approach and Crank-Nicolson formula. Having computed $(\phi^{n},\psi^{n},B^{n})$ and $(\phi^{n-1}, \psi^{n-1}, B^{n-1})$ for $n\geq1$,  $(\phi^{n+1},\psi^{n+1},B^{n+1})$ can be computed using the following second-order S-ESAV-CN scheme:
\begin{equation}\label{3.64}
	\alpha\frac{\psi^{n+1}-\psi^{n}}{\Delta t}+\beta\psi^{n+\frac{1}{2}}=M\Delta\mu^{n+1/2}
\end{equation}
\begin{equation}\label{3.65}
	\mu^{n+1/2}=(\Delta+1)^{2}\phi^{n+1/2}+\xi^{n+1}( f(\phi^{\dagger, n+1})+ f_{vac}(\phi^{\dagger, n+1}))+S(\phi^{n+1}-\phi^{\star ,n+1}),
\end{equation}
\begin{equation}\label{3.66}
	\psi^{n+1/2}=\frac{\phi^{n+1}-\phi^{n}}{\Delta t},
\end{equation}
\begin{equation}\label{3.67}
	\frac{B^{n+1}-B^n}{\Delta t}=-\frac{\beta B^{n+1}}{CM}\|\psi^{n+\frac{1}{2}}\|_{-1}^{2},
\end{equation}
where 
\begin{equation}\label{3.68}
	\xi^{n+1}=\frac{B^{\dagger, n+1}}{\exp\left(\frac{\widetilde{E}(\phi^{\dagger, n+1})}{C}\right)}.
\end{equation}
The $\phi^{\dagger, n+1}$ is any explicit $O(\Delta t^2)$ approximation for $\phi(t^{n+\frac{1}{2}})$ and $B^{\dagger, n+1}$ in \eqref{3.68} is any explicit $O(\Delta t^2)$ approximation for $B(t^{n+1/2})$.
\begin{remark}
	The second-order S-ESAV-CN scheme \eqref{3.64}-\eqref{3.68} requires $(\phi^1,\psi^1,B^1)$, which we computed by the first-order scheme :
	\begin{equation}\label{3.69}
		\alpha\frac{\psi^{1}-\psi^{0}}{\Delta t}+\beta\psi^{1}=M\Delta\mu^{1},
	\end{equation}
	\begin{equation}\label{3.70}
		\mu^{1}=(\Delta+1)^{2}\phi^{1}+\xi^{1}(f(\phi^{0})+ f_{vac}(\phi^{0}))+S(\phi^{1}-\phi^{0}),
	\end{equation}
	\begin{equation}\label{3.71}
		\psi^{1}=\frac{\phi^{1}-\phi^{0}}{\Delta t},
	\end{equation}
	\begin{equation}\label{3.72}
		\frac{B^{1}-B^{0}}{\Delta t}=-\frac{\beta B^{1}}{CM}\|\psi^{1}\|_{-1}^{2},
	\end{equation}
	where
	\begin{equation}\label{3.73}
		\xi^{1}=\frac{B^{0}}{\exp\left(\frac{\widetilde{E}(\phi^{0})}{C}\right)}.
	\end{equation}	
\end{remark}
\begin{remark}\label{esav-re3} The S-ESAV-CN scheme \eqref{3.64}-\eqref{3.68} can be implemented as follows:
	\begin{enumerate}
		\item[(i)] Compute $\xi^{1}$ from \eqref{3.73}, $\phi^{1}$ from \eqref{3.69}-\eqref{3.70}, $\psi^{1}$ from \eqref{3.71}, $B^{1}$ from \eqref{3.72}, and set n=0;
		\item[(ii)] set n=n+1;
		\item[(iii)] Compute $\xi^{n+1}$ from \eqref{3.68} and $\phi^{n+1}$ from \eqref{3.64}-\eqref{3.65};
		\item[(iv)]   Compute $\psi^{n+1}$, $\psi^{n+1/2}$ from \eqref{3.66}; compute $B^{n+1}$ from \eqref{3.67}; go back to step (ii).
	\end{enumerate}	
\end{remark}
We observe that $\phi^{n+1}$ and $B^{n+1}$ can aslo be obtained by solving only one sixth-order linear equation with constant coefficients, thereby improving the computational efficiency compared to the S-SAV method.
\begin{theorem}
	The S-ESAV-CN scheme \eqref{3.64}-\eqref{3.68} maintains mass conservation as 
	\begin{equation}
		\int_{\Omega}\phi^{n+1}dx=\int_{\Omega}\phi^{n}dx=\cdots=\int_{\Omega}\phi^{0}dx\  \text{for any } n.
	\end{equation}
\end{theorem}
\begin{proof} 
	The proof is similar to Theorem~\ref{sav-th1}
\end{proof}
\begin{theorem}
	The second-order S-ESAV-CN scheme \eqref{3.64}-\eqref{3.68} admits a unique solution and is unconditionally energy stable in the sense that
	\begin{equation}
		C \ln(B^{n+1}) \leq C \ln(B^n)\  \text{for any n}.
	\end{equation}
\end{theorem}

\begin{proof} The proof of the well-posedness of the S-ESAV-CN scheme \eqref{3.64}-\eqref{3.68} is similar to Theorem~\ref{sav-th4}.	
	We only prove that the scheme is unconditionally energy stable.
	It is obvious that $\Delta t$, $\beta$, C, M and $\|\psi^{1}\|_{-1}^{2}$ are all non-negative. Thus, from \eqref{3.72}, one  finds
	$B^{1}= \frac{B^0}{(1+\frac{\beta \Delta t}{CM}\|\psi^{1}\|_{-1}^{2})}\leq B^{0}.$
	From \eqref{3.67}, one finds
	$B^{n+1}= \frac{B^n}{(1+\frac{\beta \Delta t}{CM}\|\psi^{n+\frac{1}{2}}\|_{-1}^{2})}\leq B^{n}.$
	Due to the property of the exponential functions, we know that $B^0>0$. We can conclude by induction that $B^{n+1}\leq B^n\leq B^{n-1}...\leq B^0$ for all $n$ and further $C \ln(B^{n+1}) \leq C \ln(B^n)$. This indicates that the discrete modified energy is non-increasing, i.e.unconditionally stable. 
\end{proof}
\subsection{Energy Variation Moving Average (EV-MA) adaptive time stepping}

This subsection presents a novel adaptive time stepping strategy based on the moving average of energy variation (EV-MA), which is particularly suitable for high-order nonlinear models with energy functional sensitive to time step changes. Using the S-SAV-CN scheme as an example, which has been proven to be unconditionally energy stable in Theorem \eqref{sav-th2}, we note that this stability property allows large time steps to be used in simulations of phase transition and crystal growth. Nevertheless, overly large constant time steps often reduce accuracy and may introduce nonphysical solutions at intermediate stages (i.e., incorrect dynamics), unless the temporal variation of the solution is negligible. The authors of \cite{47qiao2011adaptive} introduced an effective and widely applicable adaptive time stepping method based on the temporal derivative of energy \cite{48,35}. Its formulation is given below:
\begin{equation}\label{85}
	\Delta t = \max\left(\Delta t_{\min}, \frac{\Delta t_{\max}}{\sqrt{1 + \alpha_1|E'(t)|^2}}\right),
\end{equation}
where $\Delta t_{\min}$, $\Delta t_{\max}$ are minimum and maximum time steps, and $\alpha_1$ is a suitable parameter. 

However, the strong nonlinear potential of the VMPFC model during its rapid structural transition phase results in high sensitivity to time step changes.
This sensitivity gives rise to instantaneous jumps in the energy derivative, triggering violent step-size variations. As a result, the time steps auto-selected by the adaptive algorithm \eqref{85} exhibit severe oscillations that produce erroneous dynamical outcomes, detailed in Fig.~\ref{fig:4.16} of the next section.
To address this issue, we propose a novel adaptive time-stepping strategy designed to smooth the adaptive time steps and better capture the rapid transition of the solution. To mitigate the influence of instantaneous energy fluctuations and ensure a smooth adaptation of the time steps, the moving average of the energy variation, denoted as $\Delta \bar{E}$, is employed.
At each time step, the energy variation $\Delta {E}$ between consecutive time steps is computed and stored in a finite-length history buffer $E_{\text{hist}}$ of size $\mathit{w_{size}}$. The average of $E_{\text{hist}}$ is then computed and denoted as $\Delta \bar{E}$, where $\Delta\bar{E} = \operatorname{mean}(E_{\text{hist}})$.
The expected time steps is subsequently calculated using the following equation:
\begin{equation}
	\Delta t =\max\left( \Delta t_{\min},\ \dfrac{\Delta t_{\max}}{\sqrt{1 + \alpha_1 \cdot \Delta\bar{E}}} \right).
\end{equation}
To prevent abrupt changes in the time steps, the ratio between adjacent  time steps is constrained by a maximum change factor $\mathit{ratio}_{\max}$.
Furthermore, when the time step exceeds a critical threshold $\Delta t_{cr}$, the stabilization term is activated to maintain numerical stability under large time steps. We summarize the procedure in Algorithm 1.
\section{Numerical experiments}
In this section, we present various 2D and 3D numerical simulations to validate the accuracy, energy stability and efficiency of the proposed numerical schemes. We assume the boundary conditions are periodic and use the Fourier spectral method with $N^{d}$ Fourier modes for spatial discretization. The computational domain is set to $\Omega=[0,L]^{d} (d = 2,3)$.
\subsection{Convergence test}
In this subsection, we test the temporal convergence rates for the S-SAV-CN scheme \eqref{3.11}-\eqref{3.14}, the S-GPAV-CN scheme~\eqref{3.41}-\eqref{3.44} and the S-ESAV-CN scheme \eqref{3.64}-\eqref{3.68} to verify the accuracy and validity. For $S = 0$, the corresponding schemes are referred to as SAV-CN, GPAV-CN, and ESAV-CN, respectively, as the stabilization term vanishes. We use $128 ^{2}$ Fourier  modes so that the spatial discretization errors are negligible compared to the errors introduced by the time discretization. We choose the suitable forcing term such that the exact solution is given by
\begin{equation}
	\phi(x,y,0)=\sin(\frac{8\pi x}{128})\cos\left(\frac{8\pi x}{128} y\right)cos(t), \quad \psi(x,y,0)=0
\end{equation}\label{4.1}
\begin{algorithm}[t]
	\caption{EV-MA Adaptive Time Strategy}
	\label{alg:adaptive_timestep}
	\begin{algorithmic}[1]		
		\Require{
			$\mathit{w_{size}}$,\ 
			$\mathit{ratio_{max}}$,\ 
			$\Delta t_{\min}$,\ 
			$\Delta t_{\max}$,\ 
			$\Delta t_{cr}$,\ 
			$\alpha_1$,\ 
			$S_{cr}$.
		}		
		\Statex
		\Procedure{AdaptiveTimeStep}{}
		\State  $ t=\Delta t^{old}=\Delta t_{\min}$, $S = 0$, Compute $\phi^{1}$ and $E^{\text{old}}$ with BDF1 scheme, $\Delta t^{\text{new}}=\Delta t_{\min}$          
		
		\While{$t < T$}  
		\State  $t = t + \Delta t^{\text{new}}$    
		\If{$\Delta t^{\text{new}} > \Delta t_{cr}$}  
		\State $S = S_{cr}\neq 0$;             \Comment{Activate stabilization term }
		\Else
		\State $S = 0$;
		\EndIf
		\State Compute $\phi^{\text{new}}$ and $E^{\text{new}} $ with CN scheme, $\Delta t^{\text{old}} \gets \Delta t^{\text{new}}$.
		
		\State $\Delta E \gets |E^{\text{new}} - E^{\text{old}}|$ \Comment{Compute energy variation}
		
		\If{$E_{\text{hist}} = \emptyset$} \Comment{Initialize history buffer}
		\State $E_{\text{hist}} \gets []$
		\EndIf
		\State $E_{\text{hist}} \gets [E_{\text{hist}}, \Delta E]$ \Comment{Update history}
		
		\If{$\mathrm{length}(E_{\text{hist}}) > \mathit{w_{size}}$} \Comment{Maintain window size}
		\State $E_{\text{hist}} \gets E_{\text{hist}}[\text{end}-\mathit{w_{size}}+1:\text{end}]$
		\EndIf
		
		\State $\Delta\bar{E} \gets \mathrm{mean}(E_{\text{hist}})$ \Comment{Compute moving average}
		
		\State $\Delta t \gets \max\left( \Delta t_{\min},\ \dfrac{\Delta t_{\max}}{\sqrt{1 + \alpha_1 \cdot \Delta\bar{E}}} \right)$ \Comment{Theoretical time step}
		
		\State $\rho \gets \dfrac{\Delta t}{\Delta t^{\text{old}}}$ \Comment{Compute change ratio}
		\If{$\rho > \mathit{ratio_{max}}$}
		\State $\Delta t^{\text{new}} \gets \Delta t^{\text{old}} \times \mathit{ratio_{max}}$ \Comment{Upper limit}
		\ElsIf{$\rho < 1/\mathit{ratio_{max}}$}
		\State $\Delta t^{\text{new}} \gets \Delta t^{\text{old}} / \mathit{ratio_{max}}$ \Comment{Lower limit}
		\Else
		\State $\Delta t^{\text{new}} \gets \Delta t$ \Comment{Accept calculated value}
		\EndIf    
		\State  $E^{\text{old}} \gets E^{\text{new}}$           		
		\EndWhile  
		\EndProcedure  
	\end{algorithmic}
\end{algorithm}	
\clearpage
\begin{figure}[t]
	\centering
	\subfigure[$h_{vac}= 0$]
	{	
		\includegraphics[width=3.2cm,height=3.2cm]{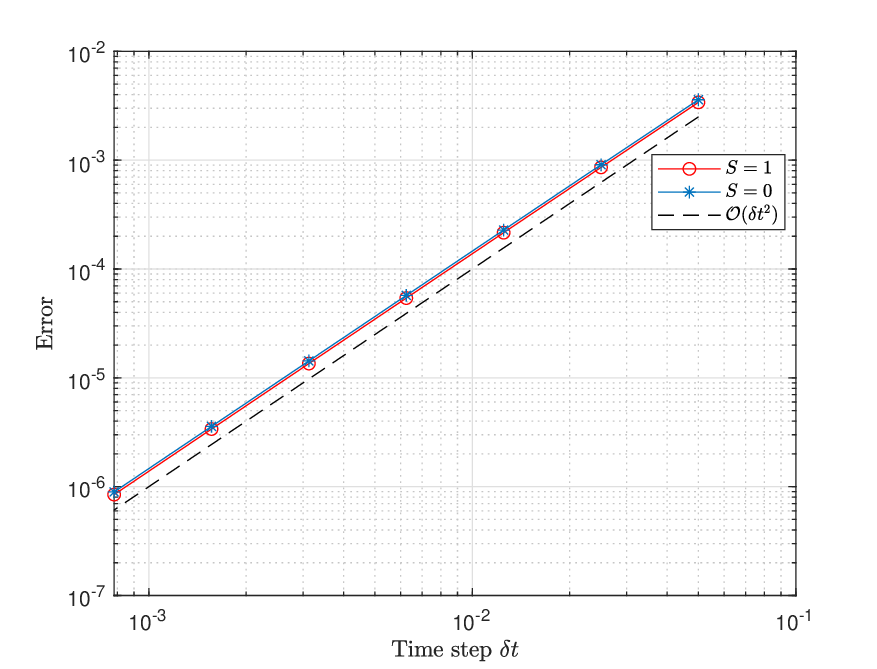}
	}	
	\subfigure[$h_{vac}= 500$]
	{	\includegraphics[width=3.2cm,height=3.2cm]{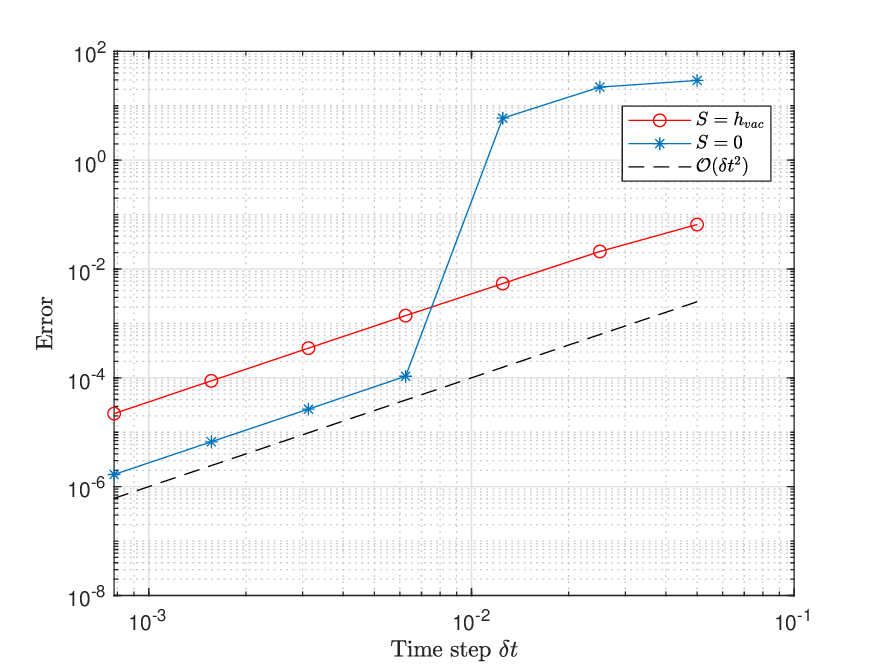}
	}
	\subfigure[$h_{vac}= 1000$]
	{	
		\includegraphics[width=3.2cm,height=3.2cm]{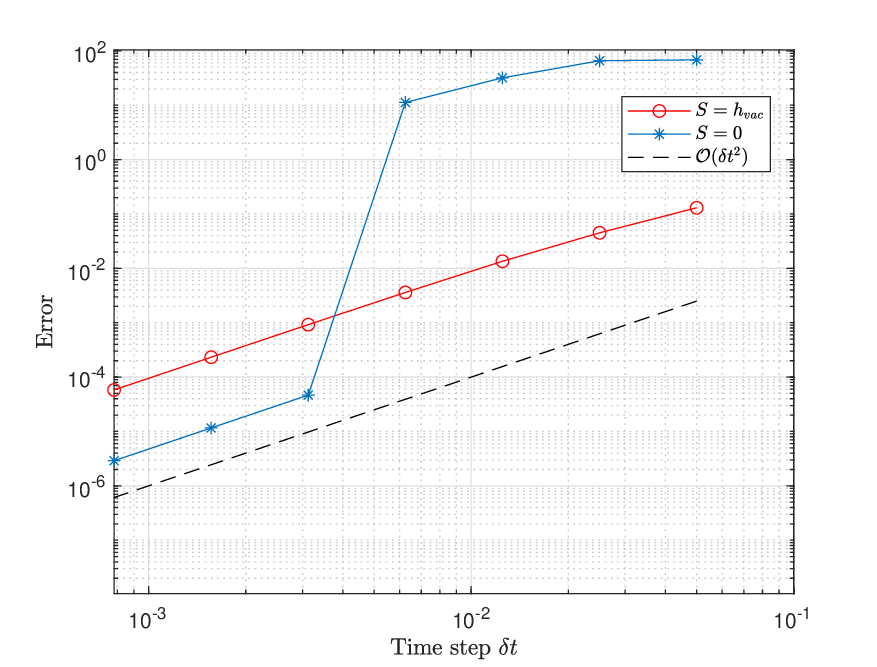}
	}	
	\subfigure[$h_{vac}= 3000$]
	{	\includegraphics[width=3.2cm,height=3.2cm]{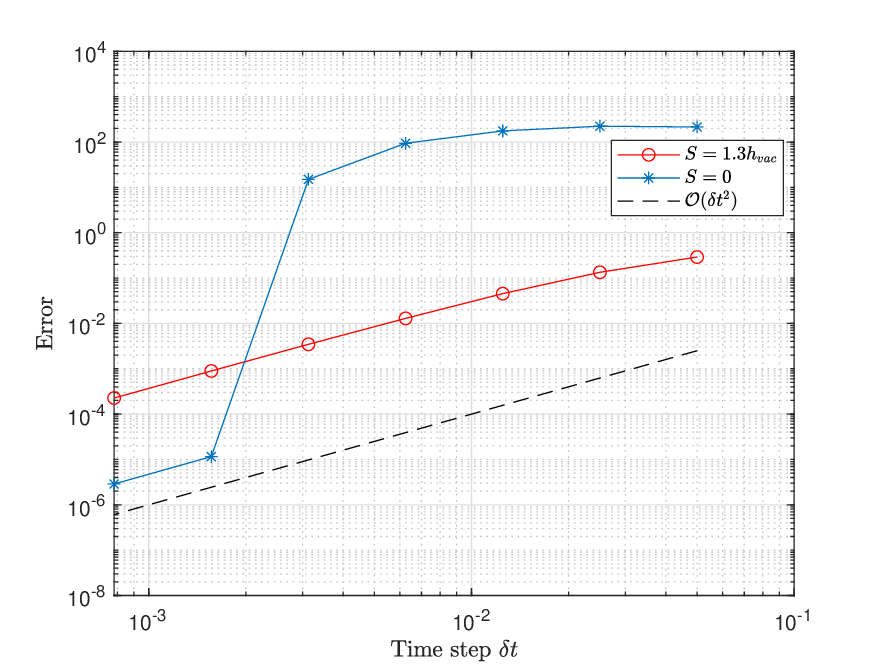}
	}
	\caption{The $L^{2}$ errors of the phase variable $\phi$ with various time steps using the S-SAV-CN scheme.}\label{fig:4.1}
\end{figure}
\begin{figure}[htp]
	\centering
	\subfigure[$h_{vac}= 0$]
	{	
		\includegraphics[width=3.2cm,height=3.2cm]{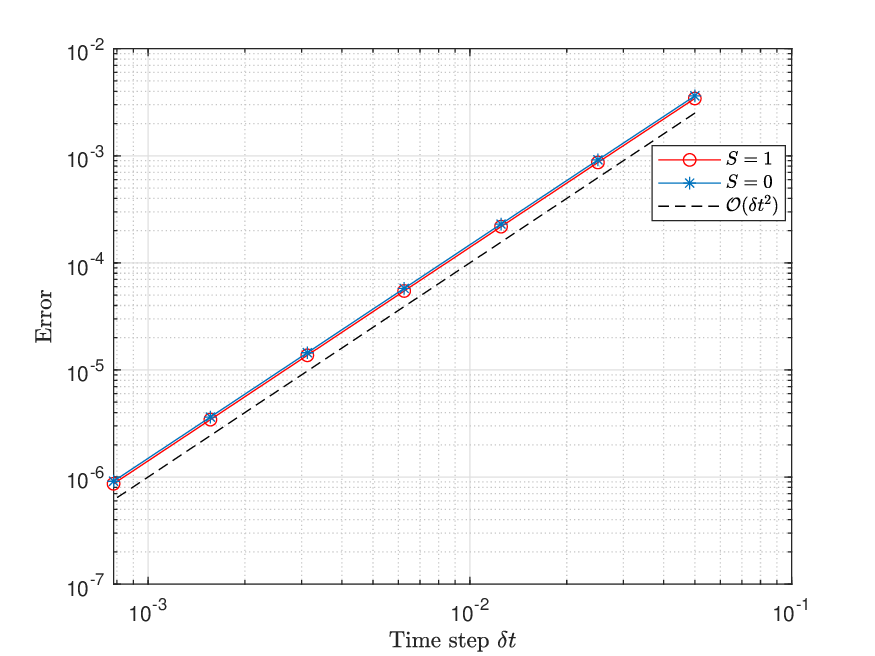}
	}	
	\subfigure[$h_{vac}= 500$]
	{	\includegraphics[width=3.2cm,height=3.2cm]{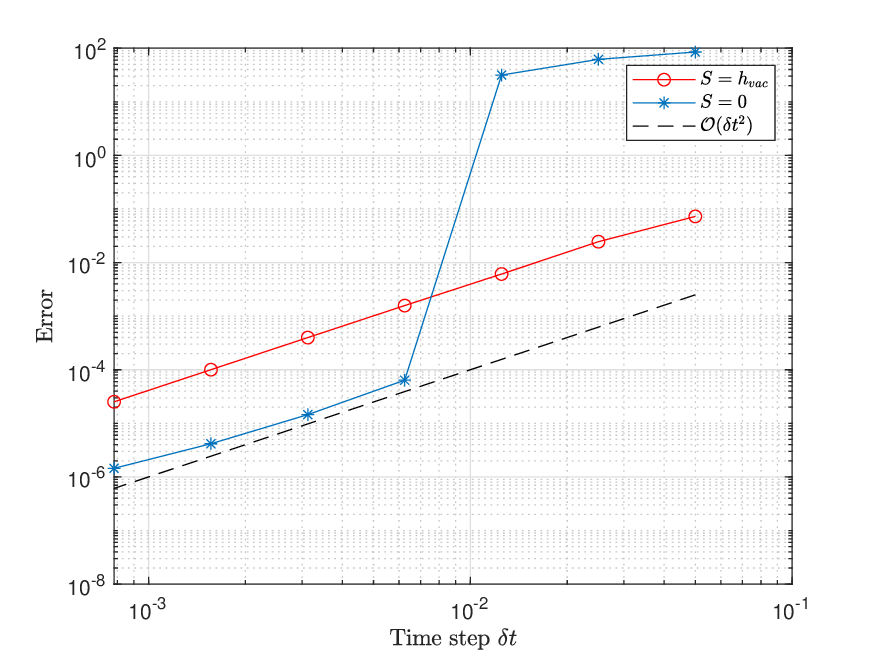}
	}
	\subfigure[$h_{vac}= 1000$]
	{	
		\includegraphics[width=3.2cm,height=3.2cm]{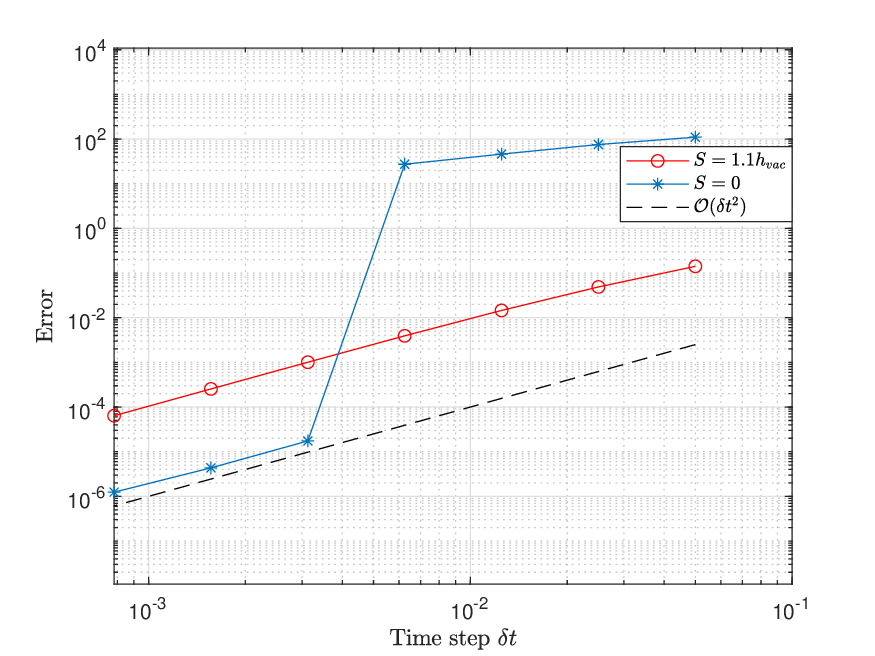}
	}	
	\subfigure[$h_{vac}= 3000$]
	{	\includegraphics[width=3.2cm,height=3.2cm]{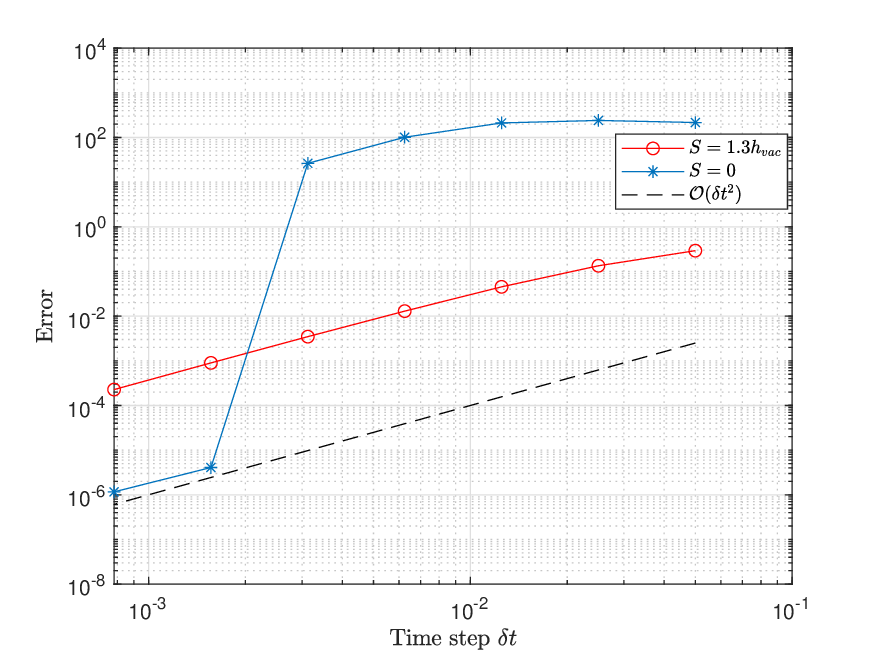}
	}
	\caption{The $L^{2}$ errors of the phase variable $\phi$ with various time steps using the S-GPAV-CN scheme.}\label{fig:4.2}
\end{figure}

\begin{figure}[t]
	\centering
	\subfigure[$h_{vac}= 0$]
	{	
		\includegraphics[width=3.2cm,height=3.2cm]{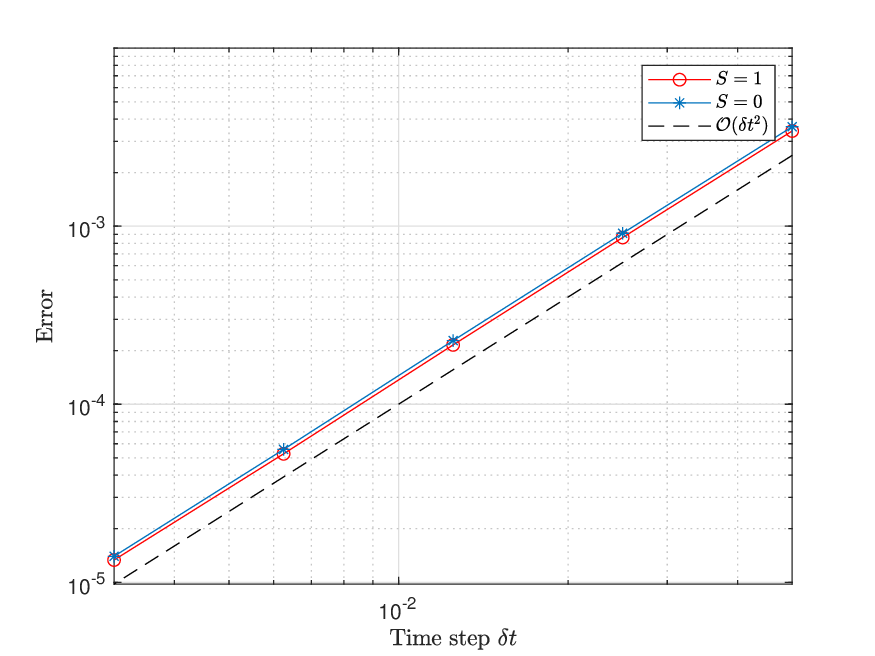}
	}	
	\subfigure[$h_{vac}= 500$]
	{	\includegraphics[width=3.2cm,height=3.2cm]{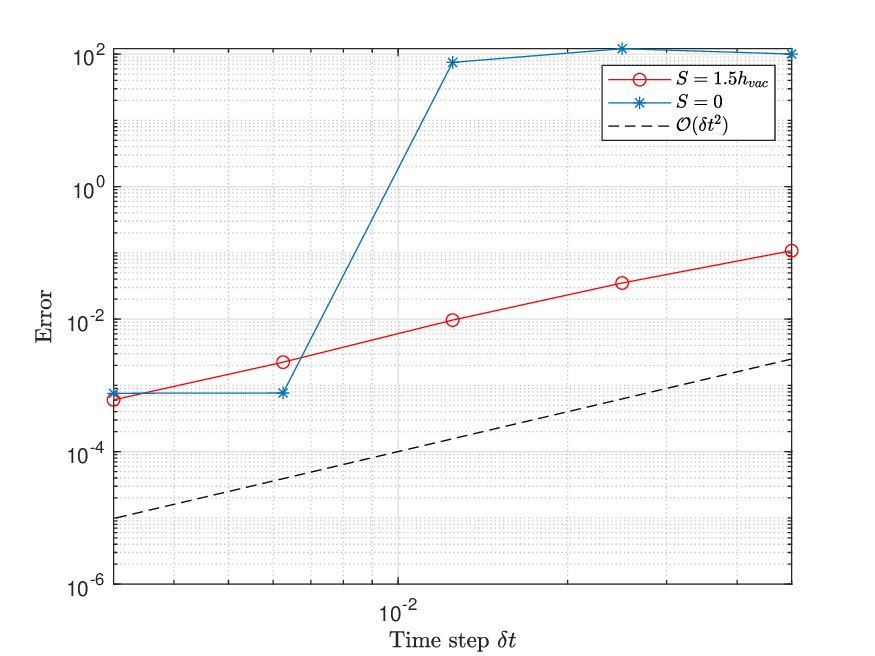}
	}
	\subfigure[$h_{vac}= 1000$]
	{	
		\includegraphics[width=3.2cm,height=3.2cm]{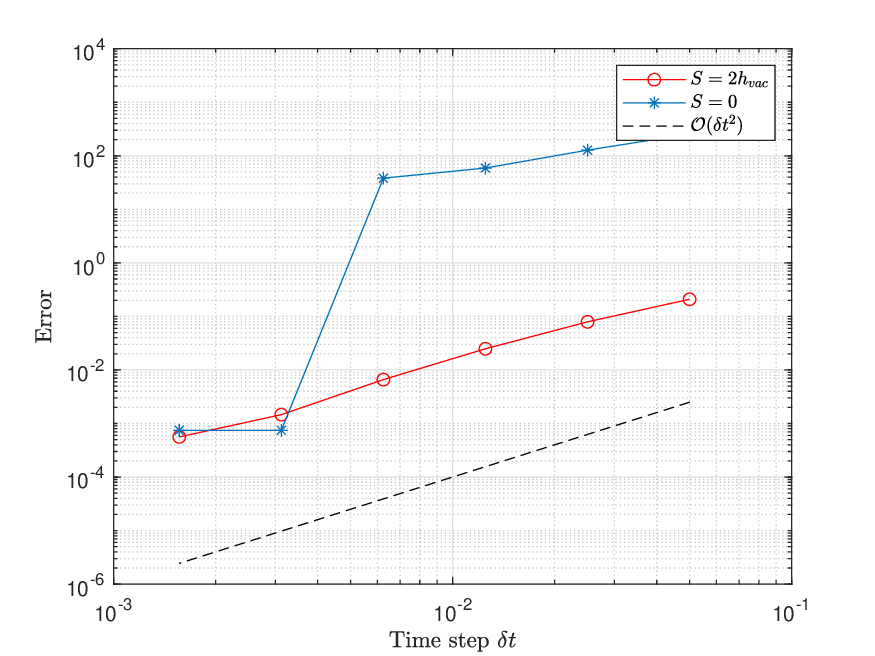}
	}	
	\subfigure[$h_{vac}= 3000$]
	{	\includegraphics[width=3.2cm,height=3.2cm]{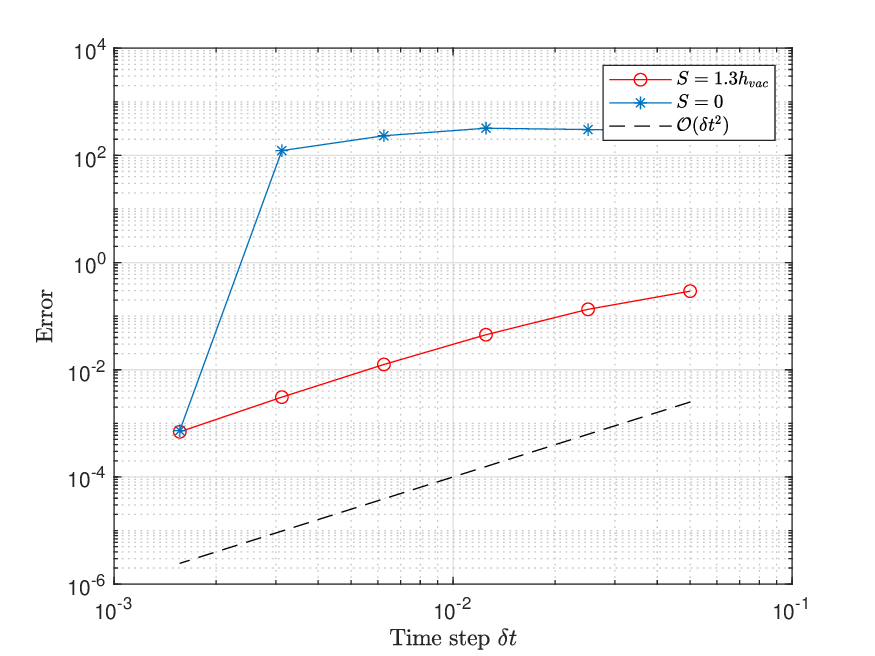}
	}
	\caption{The $L^{2}$ errors of the phase variable $\phi$ with various time steps using the S-ESAV-CN scheme.}\label{fig:4.3}
\end{figure}

\begin{figure}[t]
	\centering
	\subfigure[$S-SAV-CN,S=100$]{
		\includegraphics[width=3.2cm,height=3.2cm]{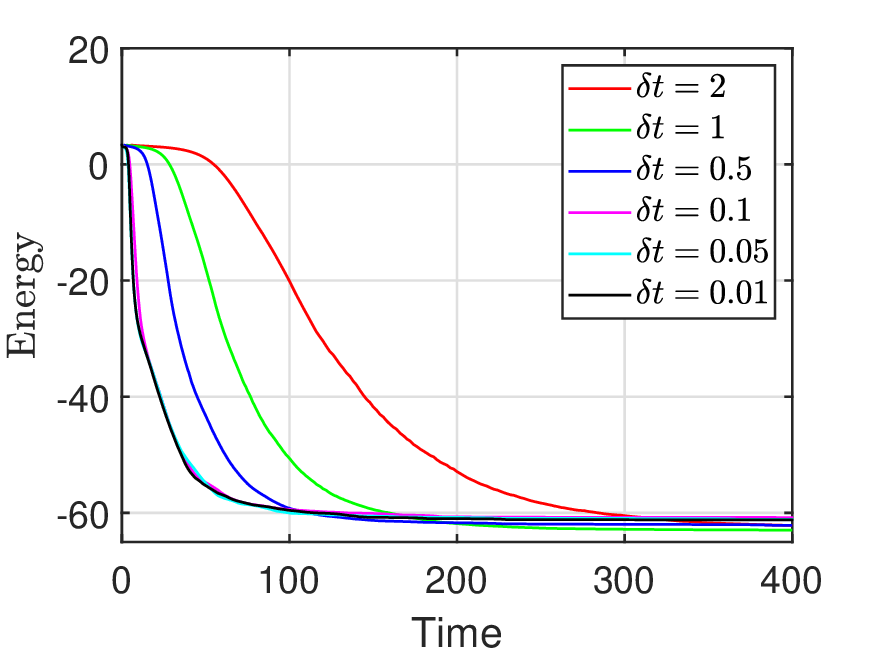}
	}
	\subfigure[$S-GPAV-CN,S=100$]{
		\includegraphics[width=3.2cm,height=3.2cm]{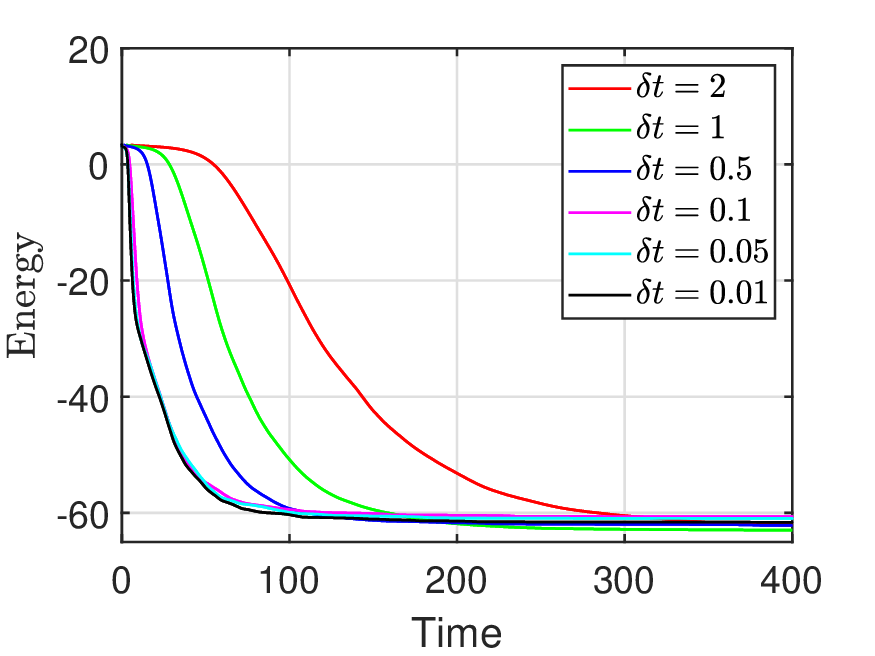}
	}	
	\subfigure[$S-ESAV-CN,S=100$]{
		\includegraphics[width=3.2cm,height=3.2cm]{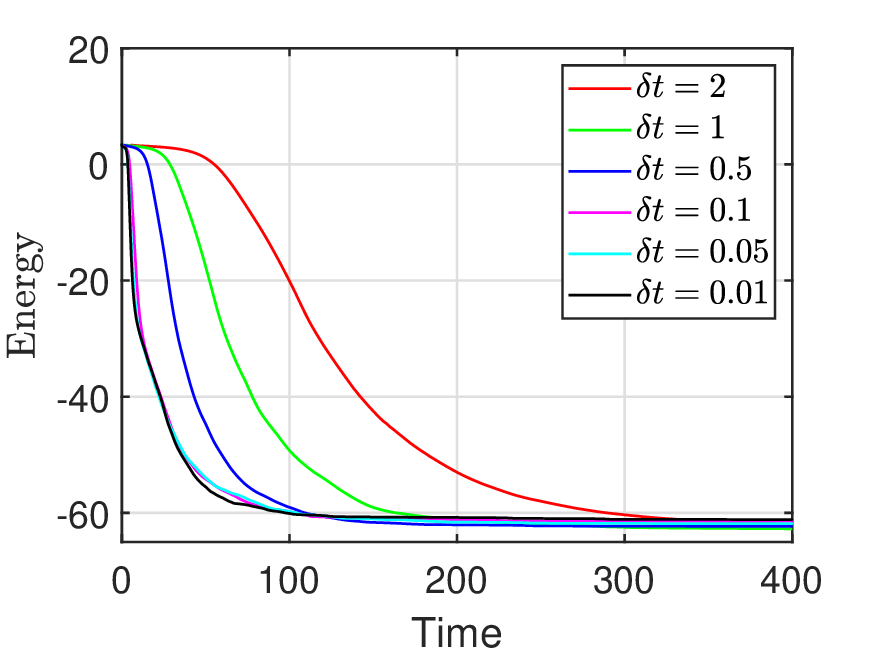}
	}
	\subfigure[$S-SAV-CN,S=0$]{
		\includegraphics[width=3.2cm,height=3.2cm]{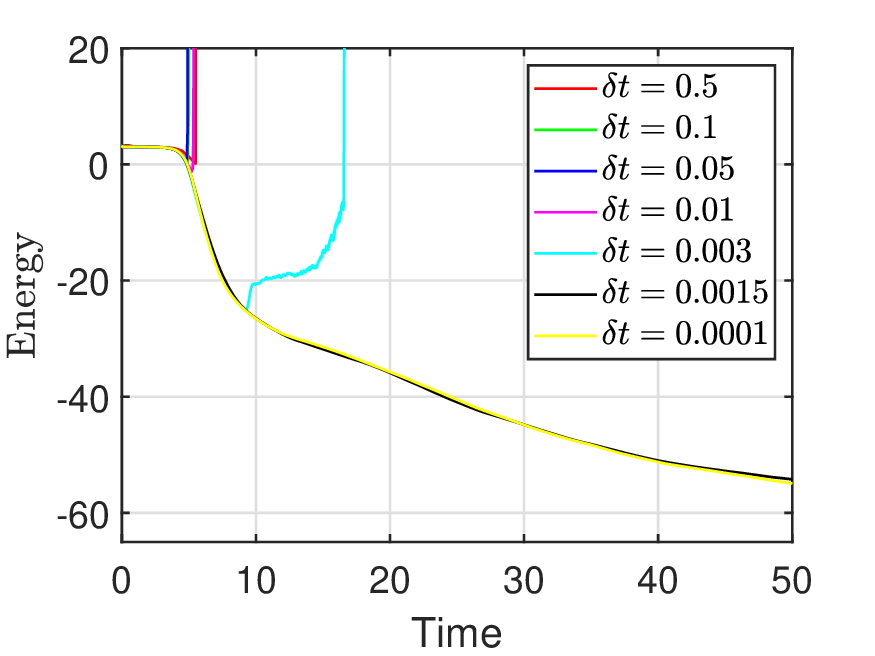}
	}
	\caption{The evolution of pseudo energ $\widetilde{E}$ using the (a) S-SAV-CN, (b) S-GPAV-CN, and (c) S-ESAV-CN schemes with $S=100$, and (d) S-SAV-CN scheme $S=0$.}\label{fig:4.4}
\end{figure}
We set the other parameters as $\alpha=1, \beta=1, M=1, \epsilon=0.025,L = 128, C=10^8, T=1.$
We set $b=10^4$ and $c_0=10^3$ (which can be set to even larger values.) in the S-SAV-CN and S-GPAV-CN schemes to ensure that the radicand in \eqref{3.2} and \eqref{3.32} is quite positive for all parameters $h_{vac}$ and $\varepsilon$.
To decelerate the rapid growth of the exponential function, a value of $C=10^8$ is used in the S-ESAV scheme.
In this example, $\left|\min_{x \in \Omega} (\phi(\mathbf{x}, t), 0 )\right|$ is around 1. Therefore, we could choose $S = 1$ in the computations for the S-SAV and S-GPAV schemes.
For larger $hvac$ values (e.g., 3000) and the S-ESAV scheme, the parameter S could be slightly larger to achieve a balance between
stability and accuracy.
Figs.~\ref{fig:4.1}-\ref{fig:4.3} show the $L^{2}$ errors between the numerical and exact solutions of the phase variable $\phi$ at $T=1$, computed with varying time steps. The temporal discretization employs time step sizes $\Delta t = \frac{2^{-k}}{10}$ for $k={1,2,\ldots,7}$. When the penalization parameter $h_{vac} = 0$ ( Figs.~\ref{fig:4.1}-\ref{fig:4.3}(a)), this term $f_{vac}(\phi)$ vanishes and the stabilization term becomes unnecessary. However, for consistency, we still choose $S = 1$ in the computations. For nonzero cases, $h_{vac} = 500$, $h_{vac} = 1000$ and $h_{vac} = 3000$ (shown in (b), (c), and (d) of Figs.~\ref{fig:4.1}-\ref{fig:4.3}), the stabilization term is activated with S.

On the one hand, the proposed numerical schemes exhibit poor accuracy for larger time steps when $S=0$. In contrast, with the stabilization term, these schemes not only provide accurate approximations to the exact solutions but also  preserve second-order accuracy across all tested time steps. 
Furthermore, when $h_{vac}=0$, the stabilization term is unnecessary, but when $h_{vac}$  reaches sufficiently large values, it becomes crucial for maintaining numerical stability.
Ultimately, these results clearly demonstrate the significant influence of the stabilization term.
In summary, the stabilizing term becomes indispensable when employing a relatively large step size or a high value of $h_{vac}$. On the other hand, the convergence behavior of the S-SAV-CN and S-GPAV-CN schemes are extremely similar, and both schemes achieve slightly better accuracy than the S-ESAV-CN scheme.
\begin{figure}[t]
	\centering
	\subfigure[$t=10$]{
		\includegraphics[width=3.4cm,height=2.8cm]{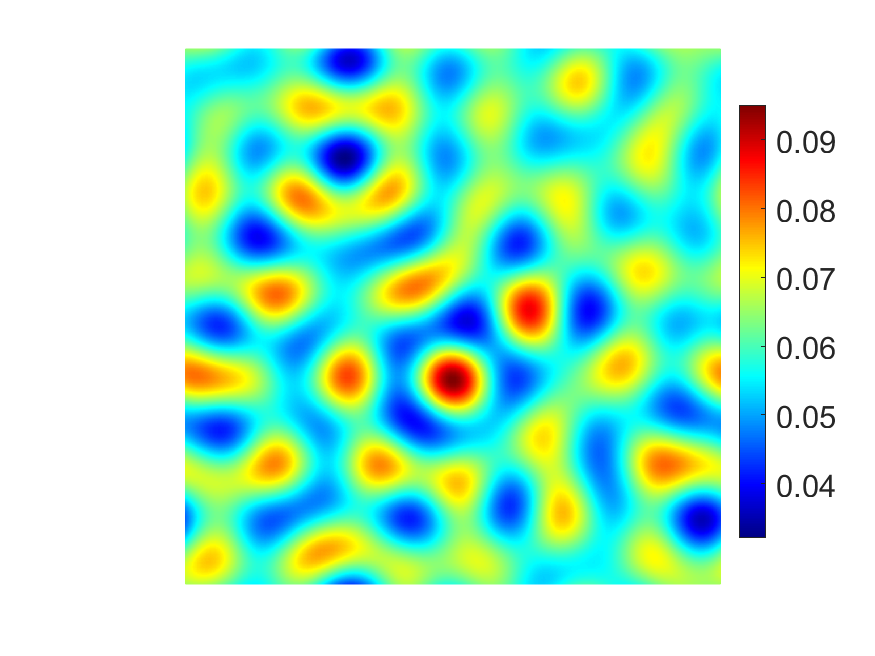}
	}
	\subfigure[$t=500$]{
		\includegraphics[width=3.4cm,height=2.8cm]{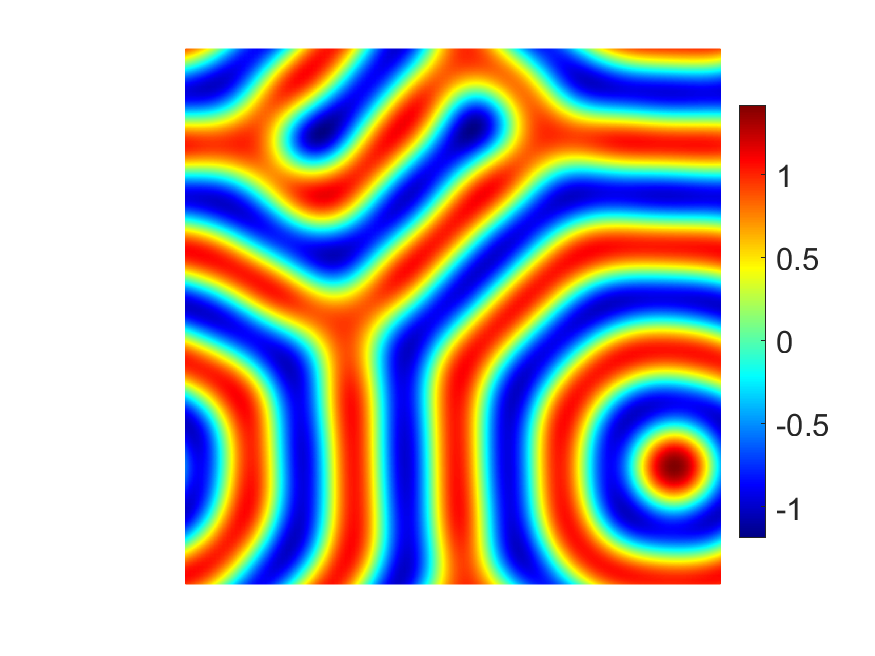}
	}
	\subfigure[$Energy$]{
		\includegraphics[width=3.4cm,height=2.8cm]{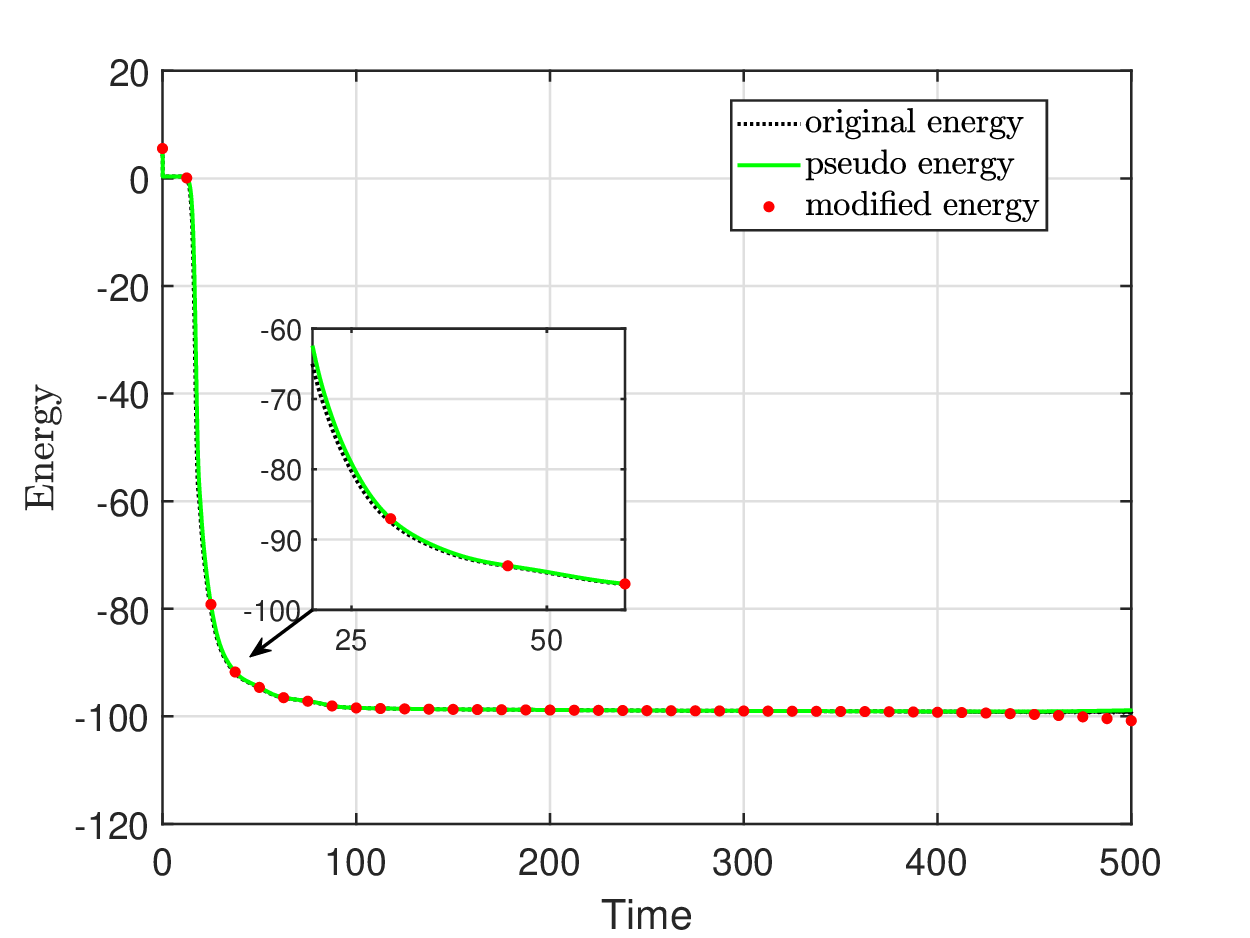}
	}
	\subfigure[$Mass$]{	
		\includegraphics[width=3.4cm,height=2.8cm]{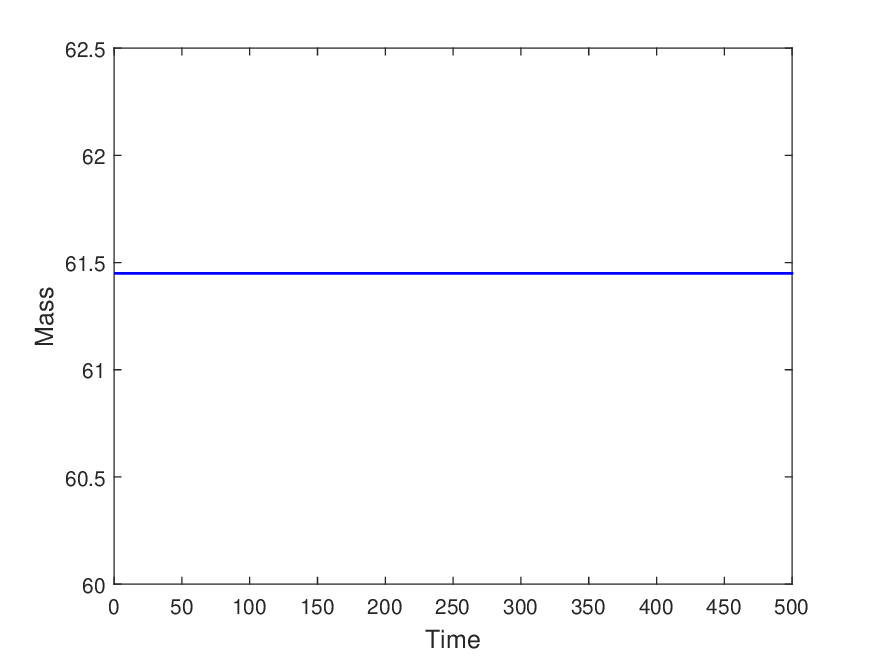}
	}
	\subfigure[$t=10$]{
		\includegraphics[width=3.4cm,height=2.8cm]{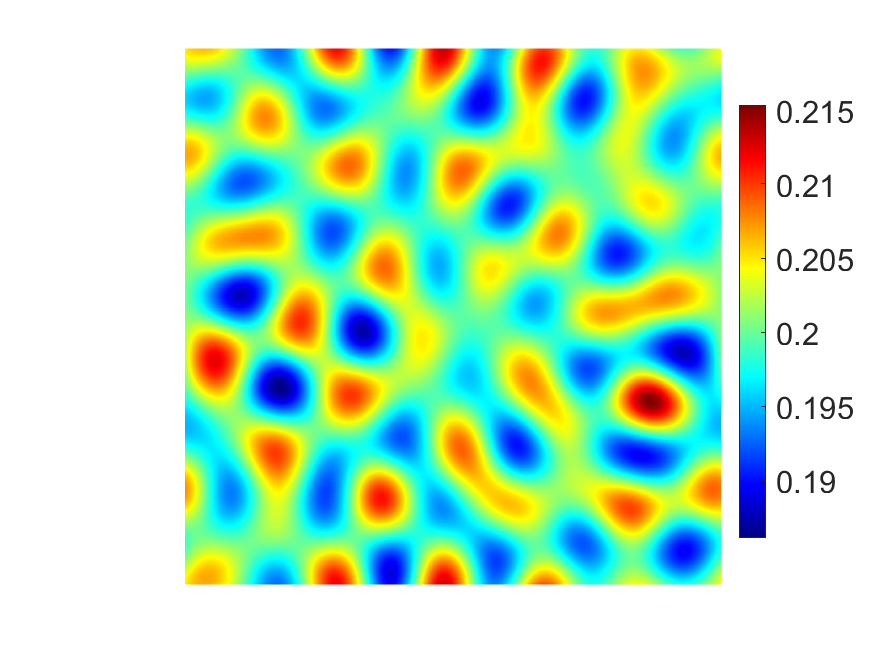}
	}
	\subfigure[$t=500$]{
		\includegraphics[width=3.4cm,height=2.8cm]{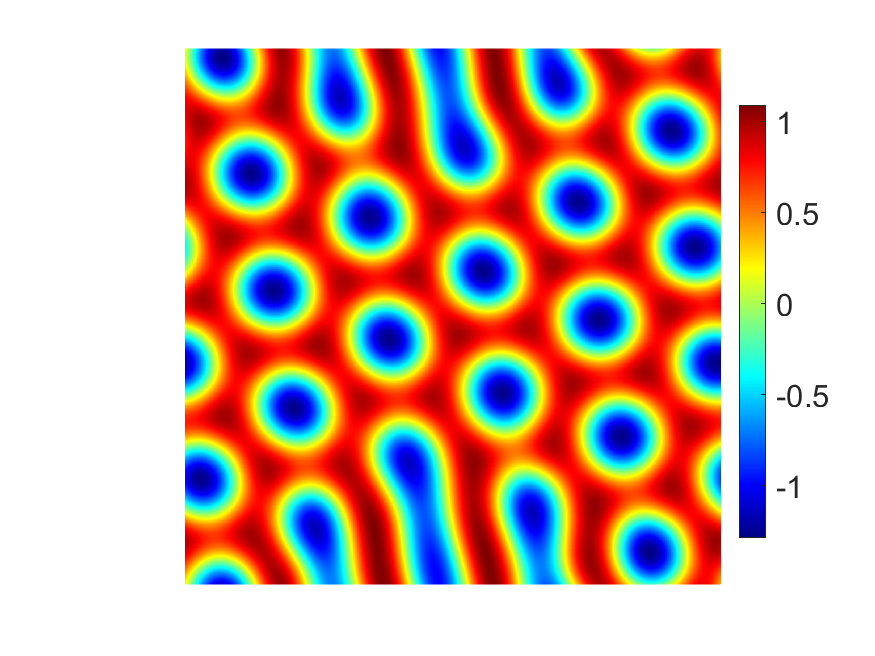}
	}
	\subfigure[$Energy$]{
		\includegraphics[width=3.4cm,height=2.8cm]{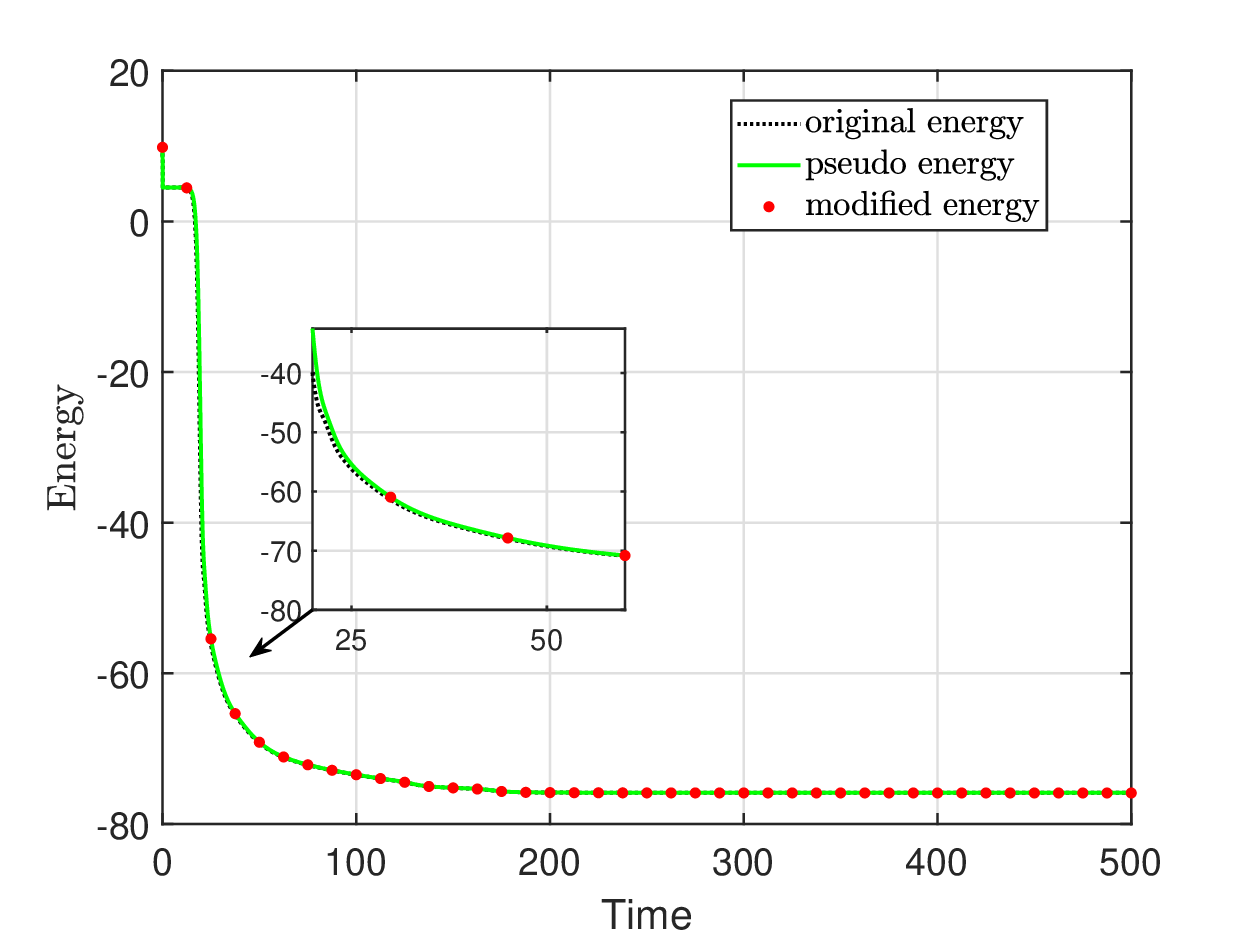}
	}
	\subfigure[$Mass$]{	
		\includegraphics[width=3.4cm,height=2.8cm]{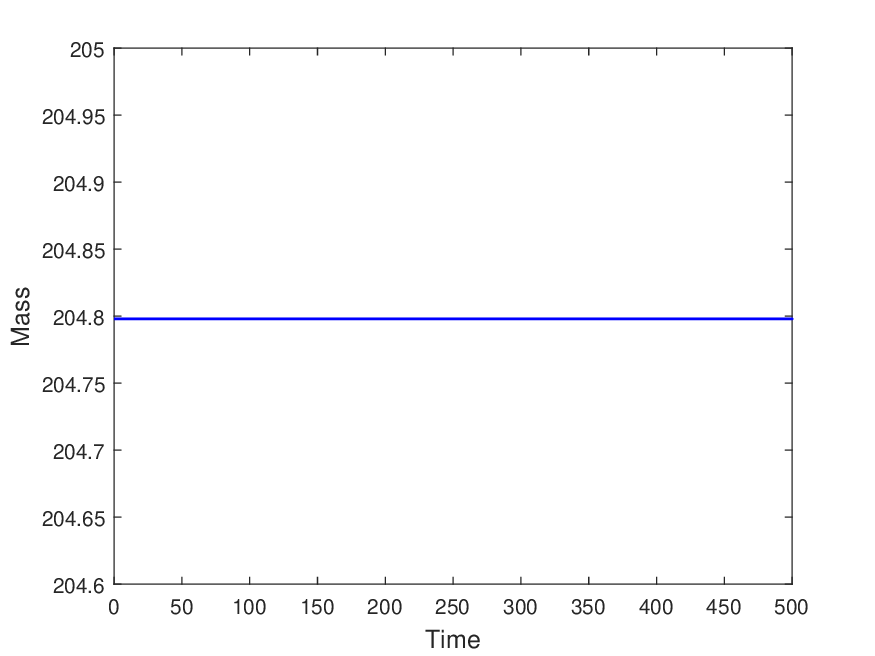}
	}
	\subfigure[$t=10$]{	
		\includegraphics[width=3.4cm,height=2.8cm]{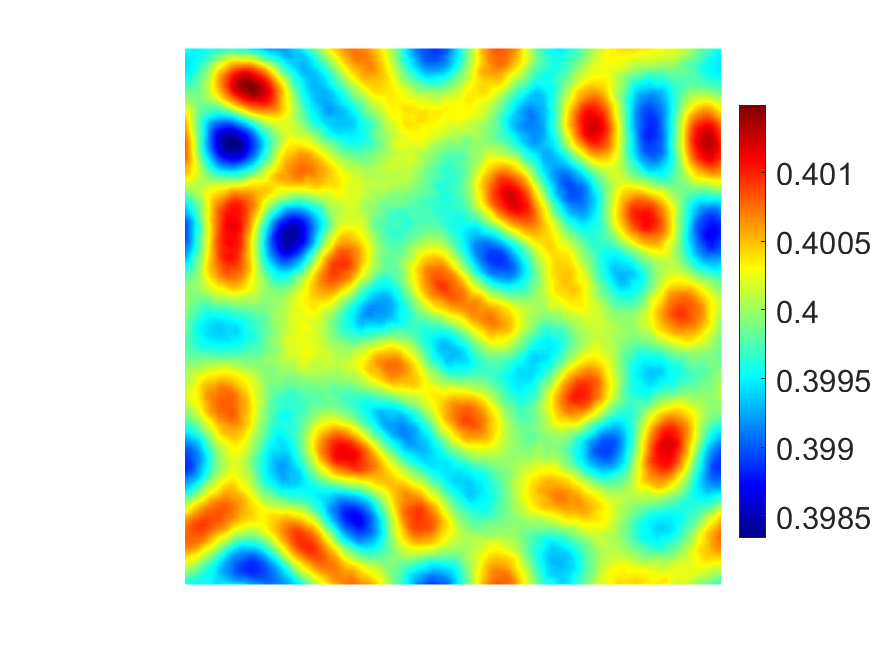}
	}
	\subfigure[$t=500$]{
		\includegraphics[width=3.4cm,height=2.8cm]{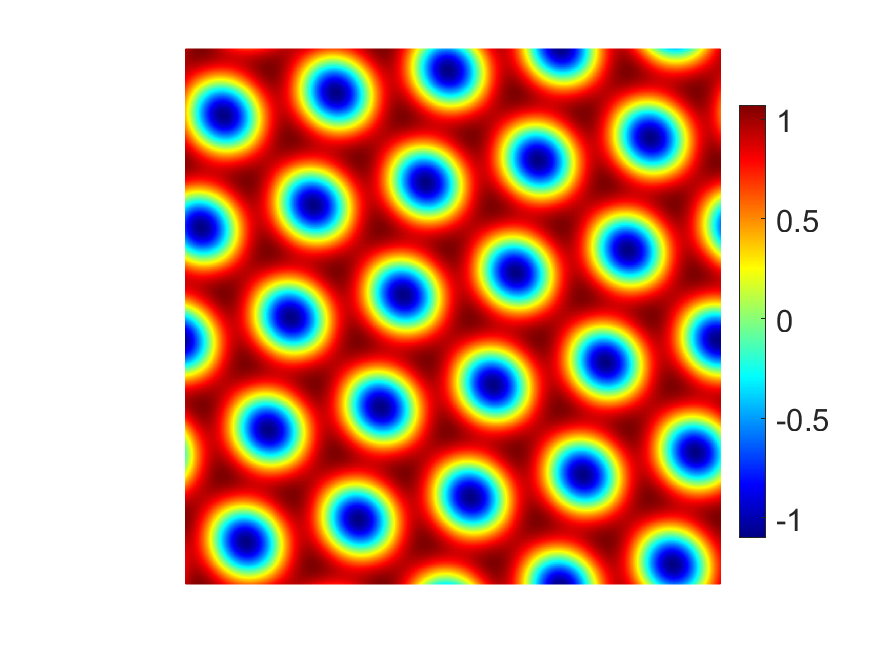}
	}
	\subfigure[$Energy$]{
		\includegraphics[width=3.4cm,height=2.8cm]{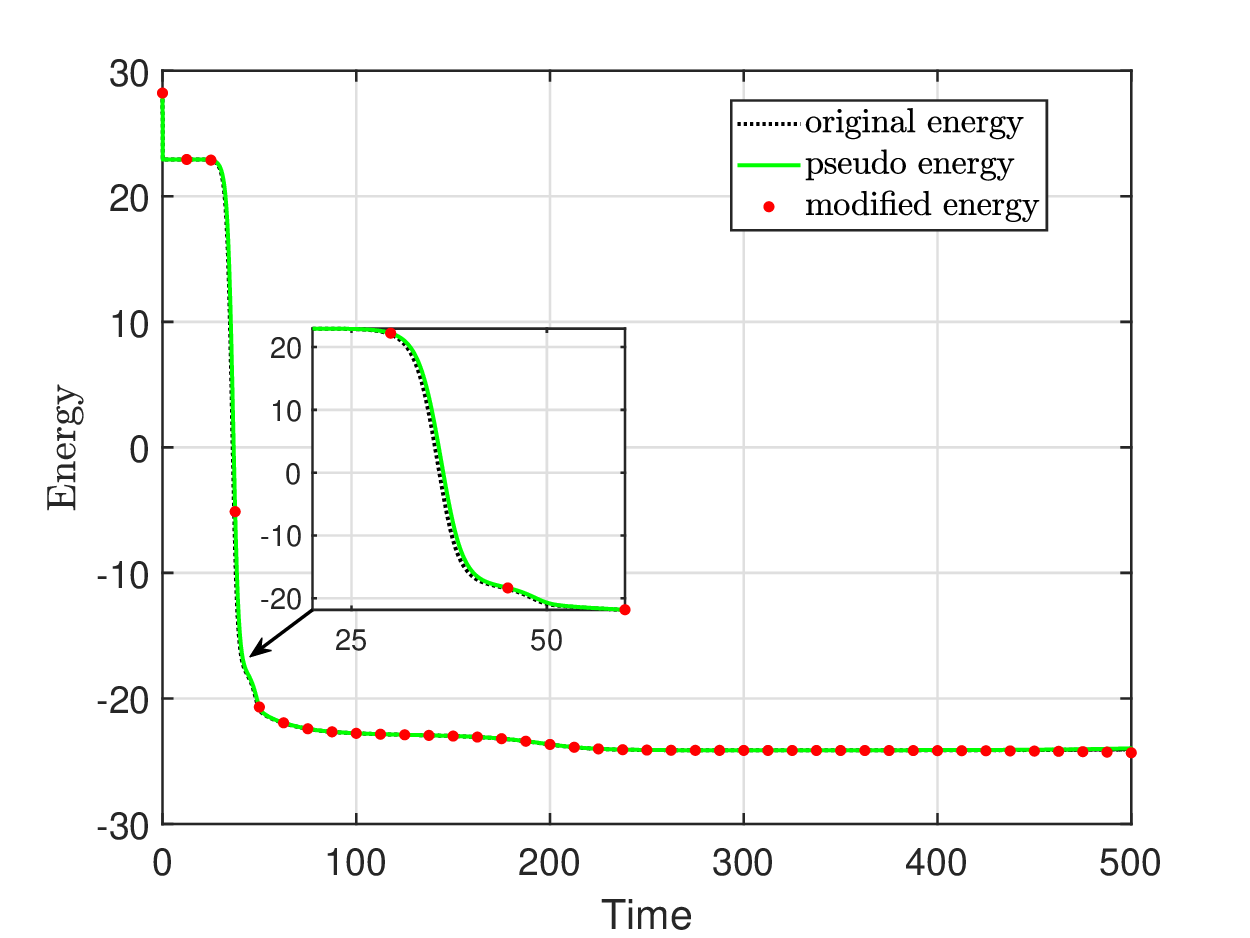}
	}
	\subfigure[$Mass$]{	
		\includegraphics[width=3.4cm,height=2.8cm]{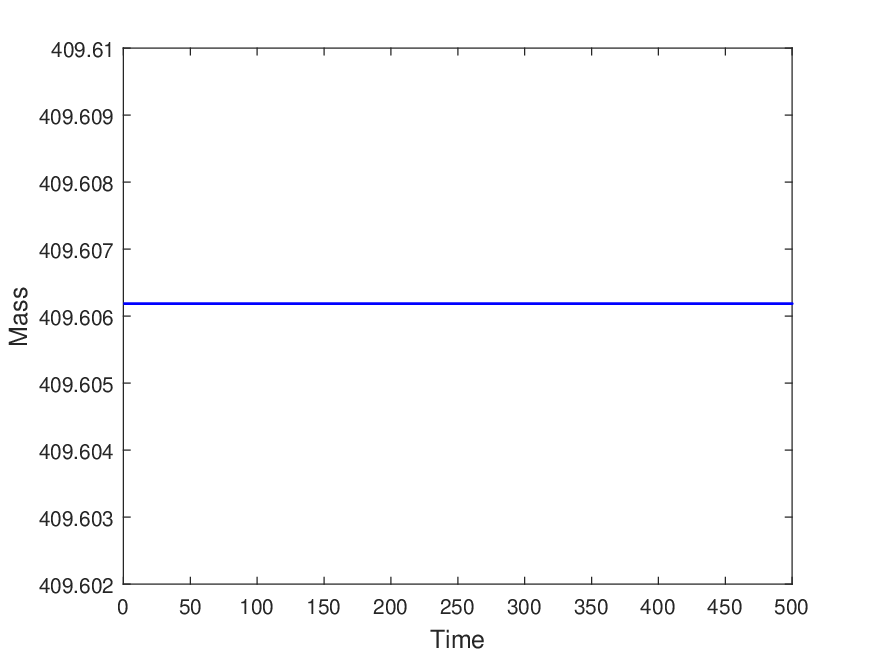}
	}
	\caption{The 2D dynamical behaviors of the phase transition without vacancy potential with $\bar{\phi} = 0.06(Row 1), 0.2(Row 2), $ and $0.4(Row 3)$. Snapshots of the phase variable $\phi$ correspond to $t = 10$ and $500$.}\label{fig:4.5}
\end{figure}
\subsection{Energy stability test}
We now conduct the energy stability tests for our proposed schemes using numerical experiments similar to those in \cite{31zhang2019efficient,32pei2022efficient,33zhang2023highly,16li2019efficient}. The fixed parameters are set to $\varepsilon=0.9,\alpha=0.01,\beta=1, M=1, T=400, L=128$. We use $128^2$ Fourier modes for spatial discretization to demonstrate the energy stability of the proposed schemes with the following initial conditions:
\begin{equation}\label{4.2.1}
	\phi(\textbf{x},0)=0.06+0.001rand(\textbf{x}), \quad  \psi(\textbf{x},0)=0,
\end{equation}
where $rand(\textbf{x})$ is uniformly distributed random number between -1 and 1.
First, we set the penalization parameter $h_{vac} = 5000$ and the stabilization parameter $S = 100$. Fig.~\ref{fig:4.4}(a) shows the temporal evolution of the pseudo energy \eqref{2.9}, computed by the S-SAV-CN scheme, for different time steps $\Delta t = 2, 1, 0.5, 0.1, 0.05,$ and $0.01$. The unconditional stability of the S-SAV-CN scheme is demonstrated numerically: the pseudo energy \eqref{2.9} decays monotonically for all tested time steps, including cases with relatively large $\Delta t$.
For comparison, identical tests are conducted for both S-GPAV-CN and S-ESAV-CN schemes, with their corresponding pseudo energy evolution curves shown in Fig.~\ref{fig:4.4}(b) and \ref{fig:4.4}(c), respectively.
All energy profiles in Fig.~\ref{fig:4.4}(a)--(c) exhibit monotonic decay, indicating that all three schemes are energy stable.
Due to the similarity of the energy results observed in Fig.~\ref{fig:4.4}(a), (b), and (c) using all schemes, we use the S-SAV-CN scheme to study the effects of the stabilization term.
Fig.~\ref{fig:4.4}(d) shows the energy evolution without the stabilization term ($S = 0$), revealing distinct behavior: pseudo energy growth occurs at larger time steps ($\Delta t > 0.003$), while monotonic decay is only achieved when $\Delta t \leq 0.0015$. These numerical results highlight two key findings: (1) the stabilization term plays a crucial role in maintaining energy stability, and (2) the developed second-order schemes guarantee unconditional energy stability.
\subsection{Phase transition simulations in 2D and 3D}
This subsection presents the phase transition dynamics of the VMPFC model. The developed benchmark experiment are equally applicable to classical PFC \cite{8hu2009stable,9shin2016first,10yang2017linearly} and MPFC \cite{16li2019efficient,24baskaran2013energy,28li2022efficient,30dehghan2016numerical} models. The initial conditions are defined as $\phi(\textbf{x}, 0)=\bar{\phi}+ 0.001rand(\textbf{x}), \psi(\textbf{x},0)=0,$
where $rand(\textbf{x})$ represents uniformly distributed random perturbations between -1 and 1 at grid points, with $\bar{\phi}$ representing the constant mean density.
\begin{figure}[t]
	\centering
	\subfigure[$t=50$]{
		\includegraphics[width=3.4cm,height=2.8cm]{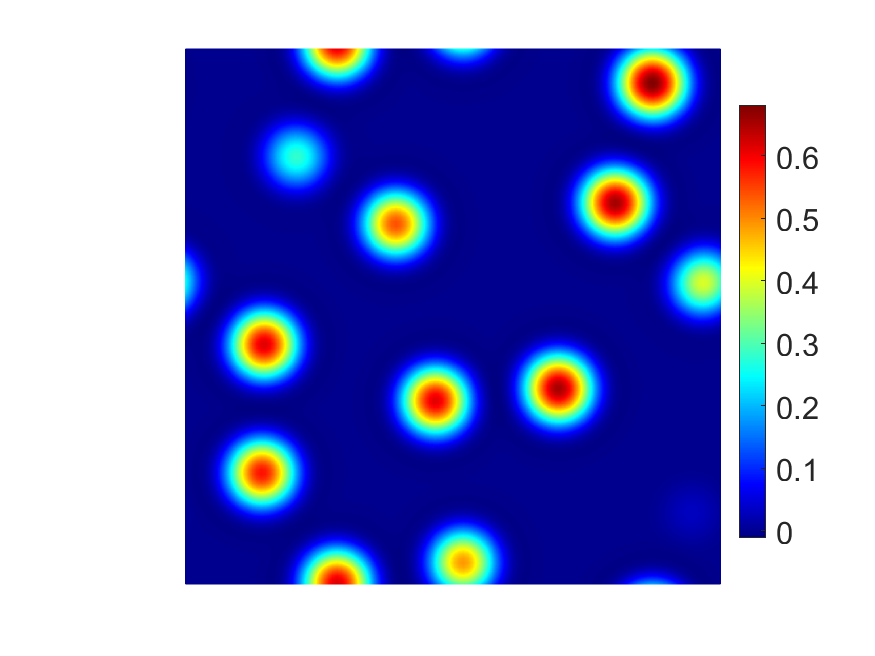}
	}
	\subfigure[$t=500$]{
		\includegraphics[width=3.4cm,height=2.8cm]{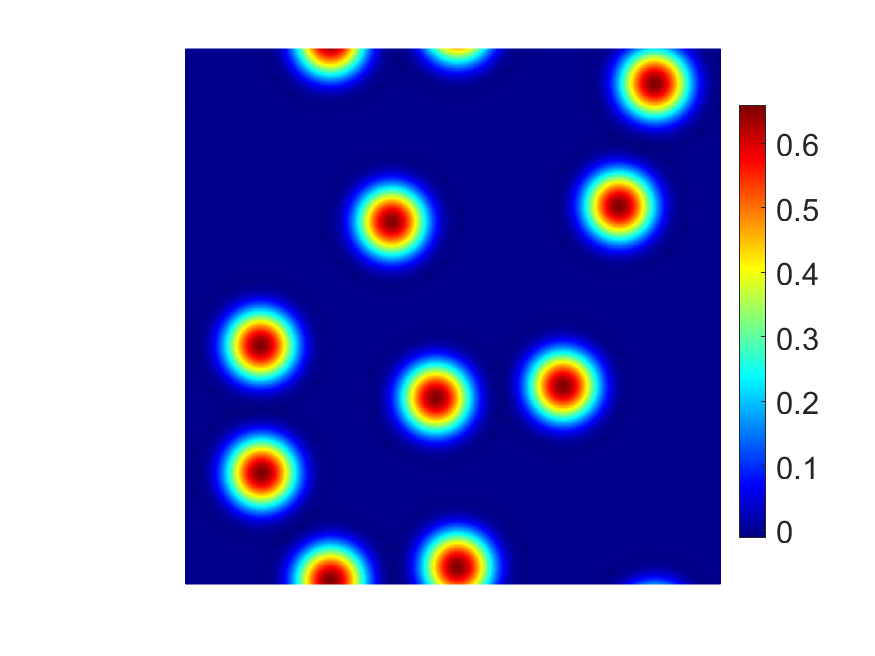}
	}
	\subfigure[$Energy$]{
		\includegraphics[width=3.4cm,height=2.8cm]{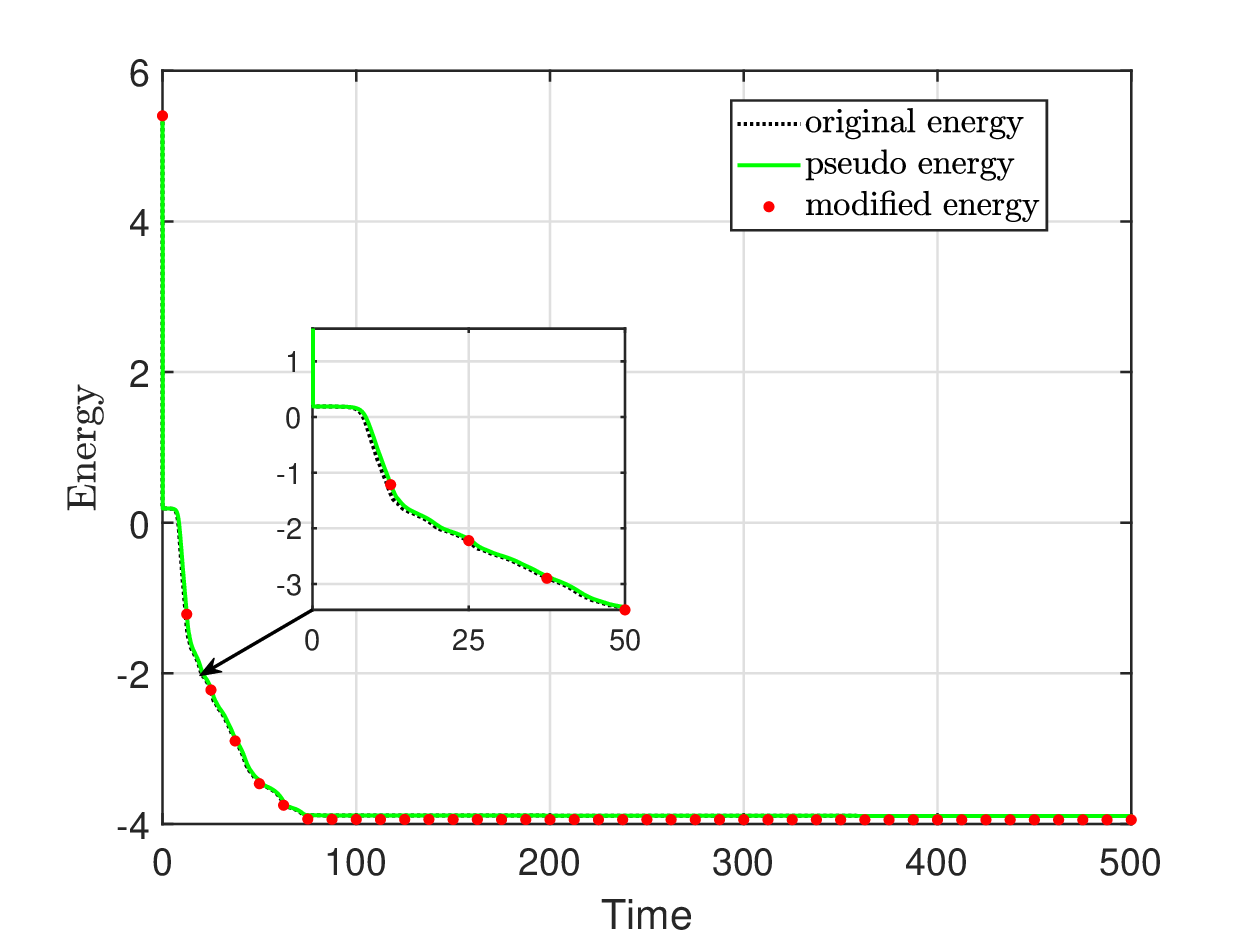}
	}
	\subfigure[$Mass$]{	
		\includegraphics[width=3.4cm,height=2.8cm]{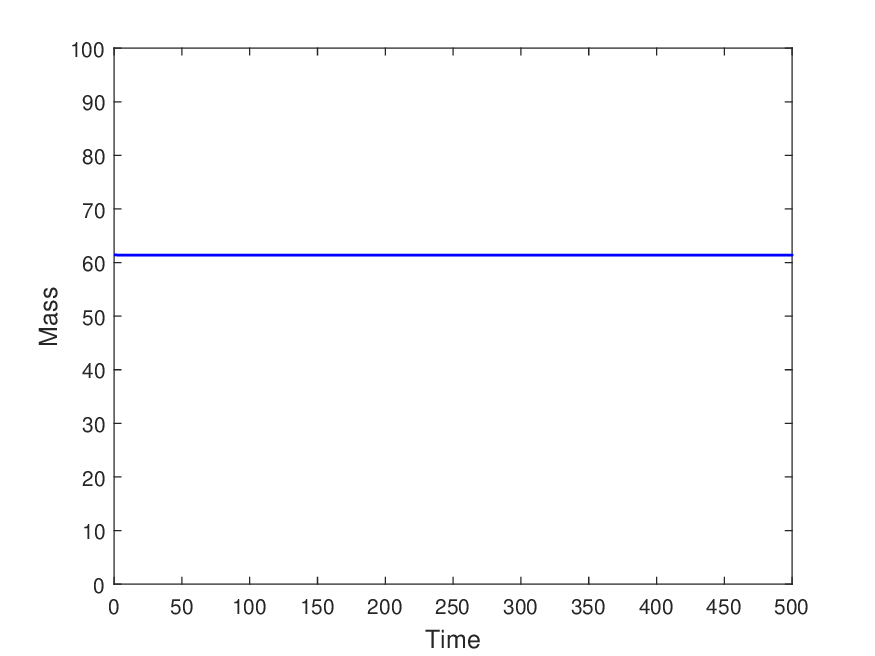}
	}
	\subfigure[$t=50$]{
		\includegraphics[width=3.4cm,height=2.8cm]{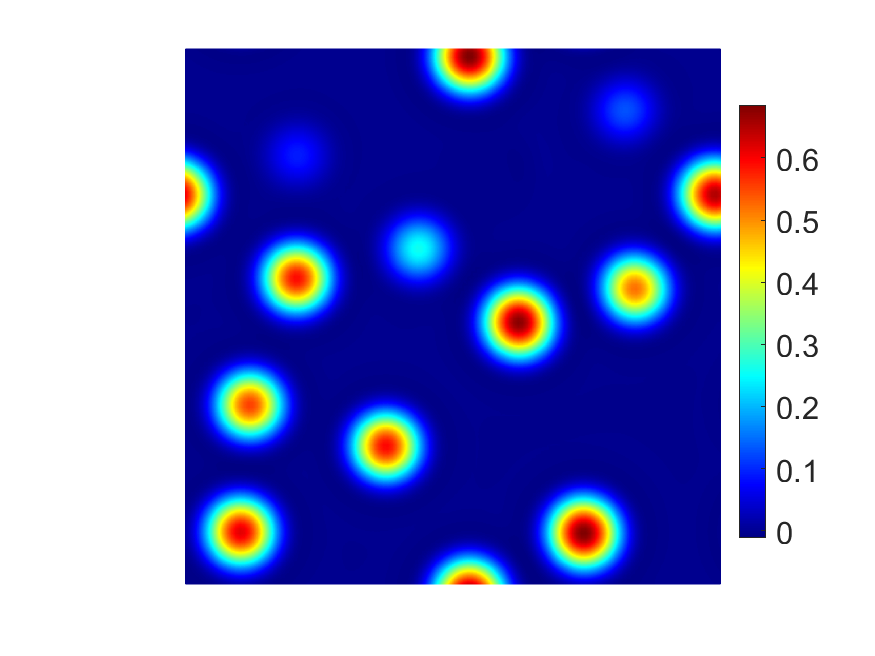}
	}
	\subfigure[$t=500$]{
		\includegraphics[width=3.4cm,height=2.8cm]{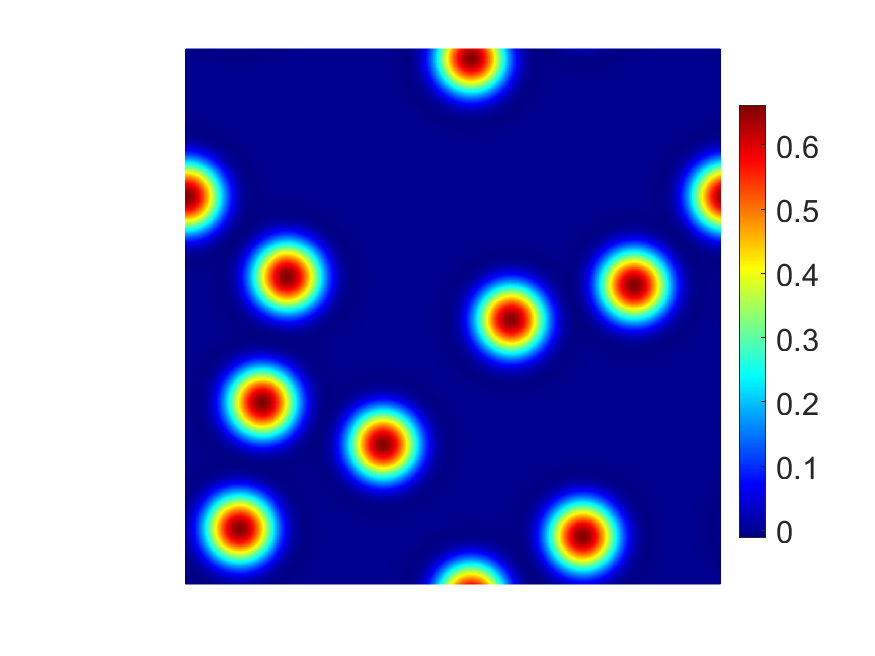}
	}
	\subfigure[$Energy$]{
		\includegraphics[width=3.4cm,height=2.8cm]{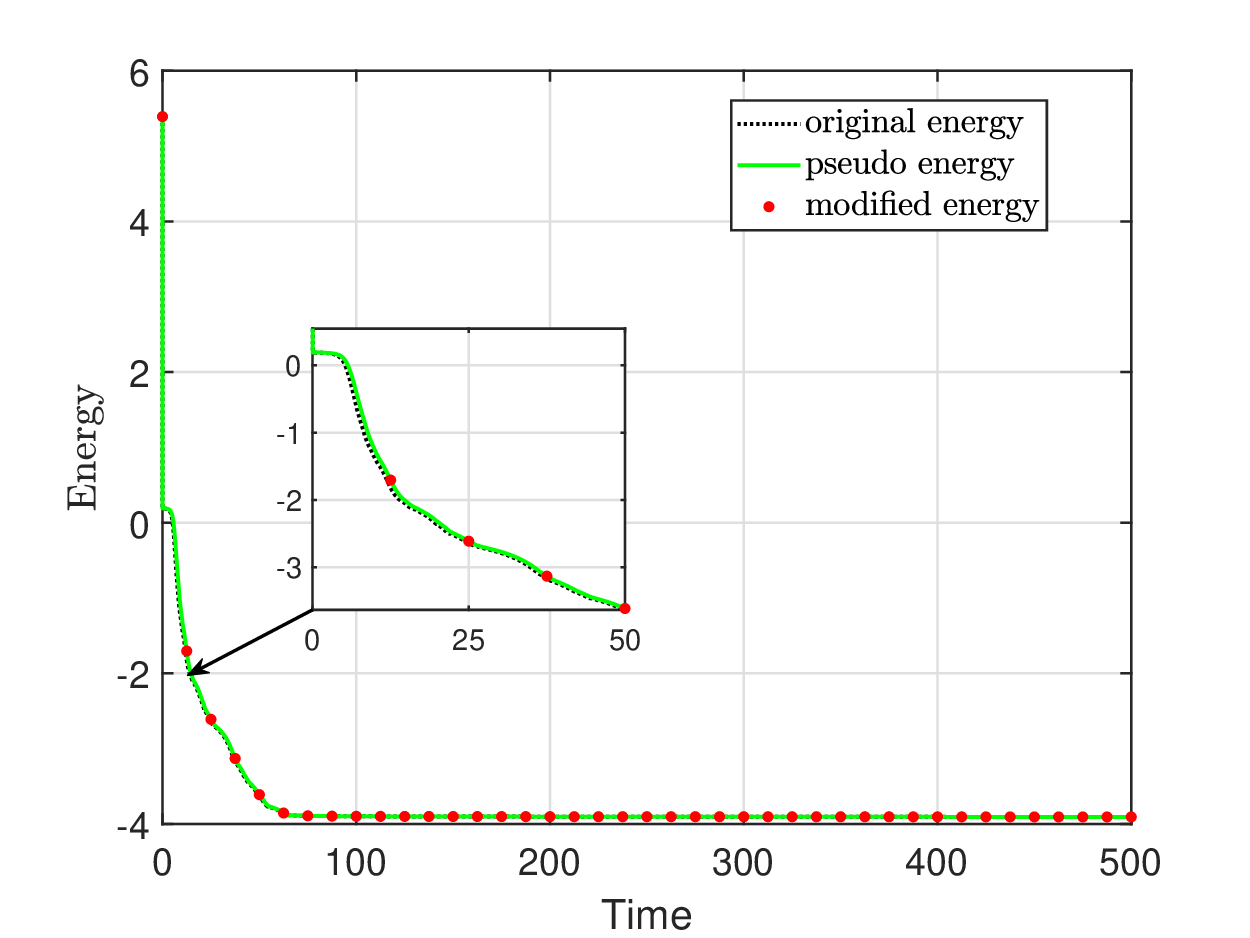}
	}
	\subfigure[$Mass$]{	
		\includegraphics[width=3.4cm,height=2.8cm]{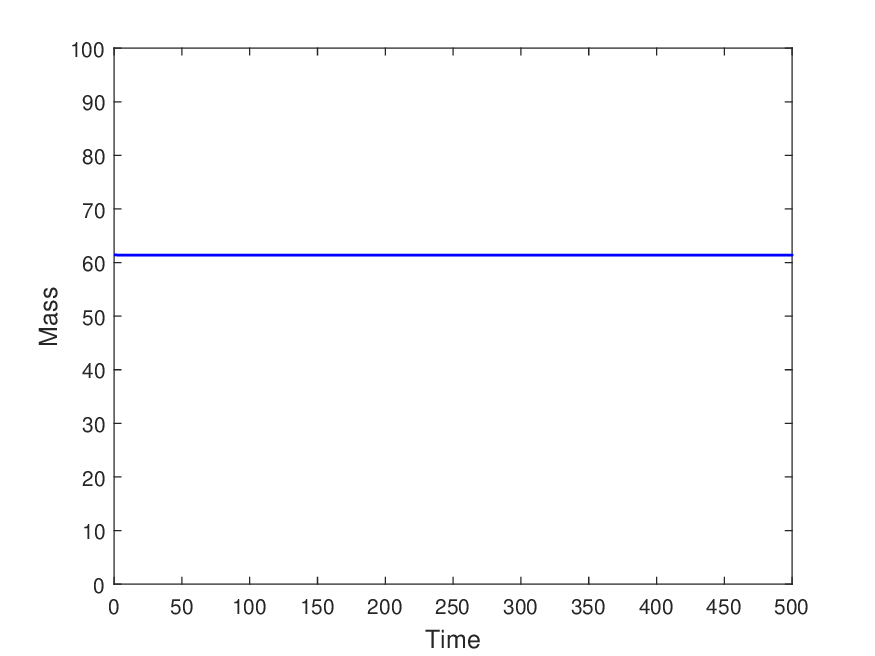}
	}
	\subfigure[$t=50$]{	
		\includegraphics[width=3.4cm,height=2.8cm]{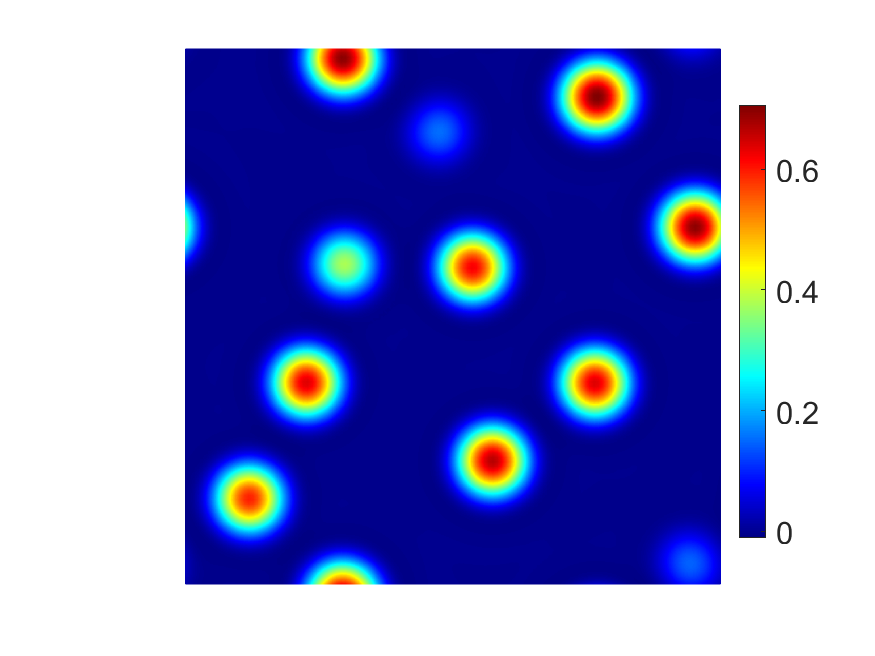}
	}
	\subfigure[$t=500$]{
		\includegraphics[width=3.4cm,height=2.8cm]{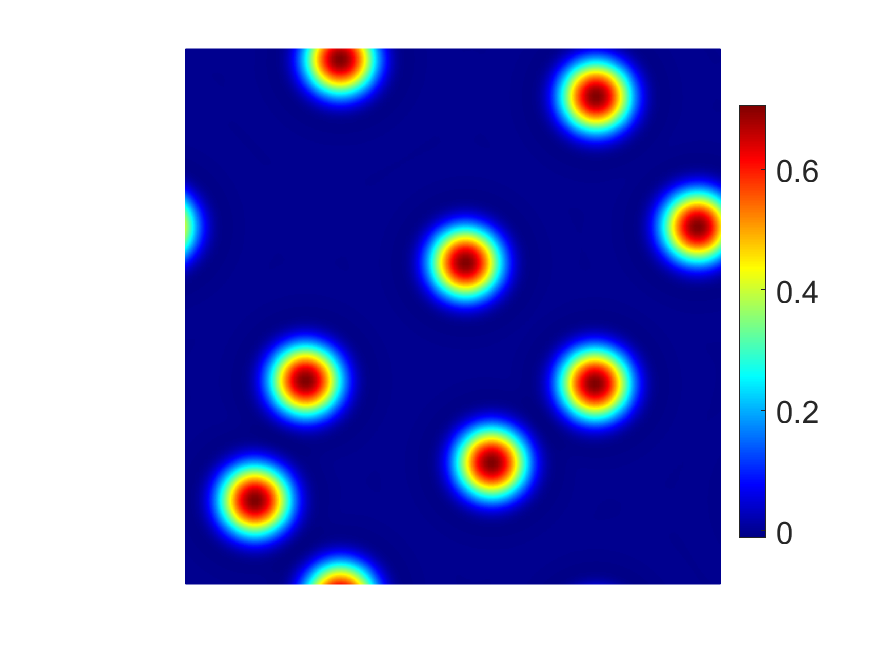}
	}
	\subfigure[$Energy$]{
		\includegraphics[width=3.4cm,height=2.8cm]{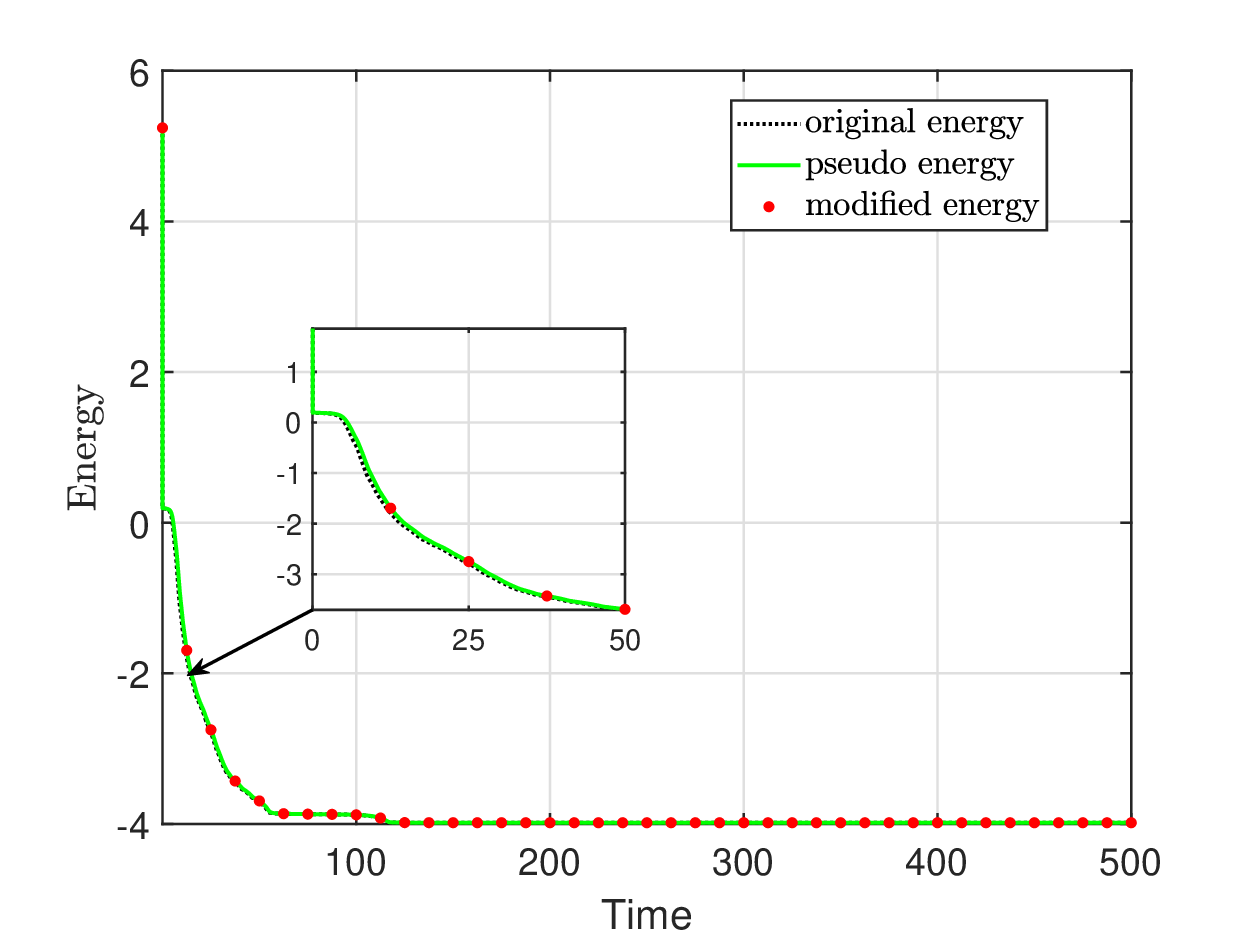}
	}
	\subfigure[$Mass$]{	
		\includegraphics[width=3.4cm,height=2.8cm]{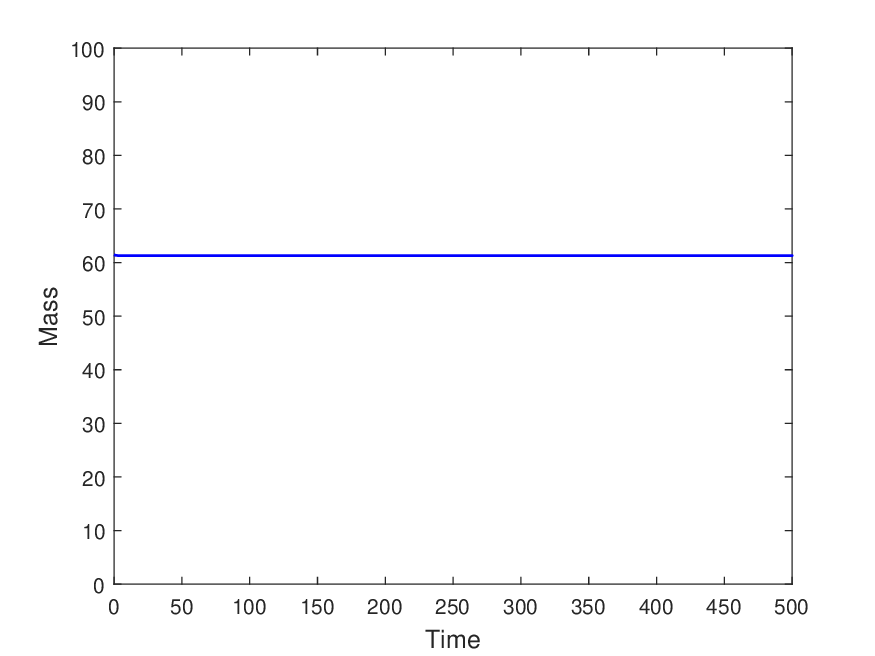}
	}
	\caption{The 2D dynamical behaviors of the phase transition with vacancy potential with $\bar{\phi} = 0.06$ using S-SAV-CN(Row 1), S-GPAV-CN(Row 2), and S-ESAV-CN(Row 3) schemes. Snapshots of the phase variable $\phi$ correspond to $t = 10$ and $500.$}\label{fig:4.6}
\end{figure}
\begin{figure}[htb]
	\centering
	\subfigure[$\bar{\phi}=0.08$]{
		\includegraphics[width=3.2cm,height=2.8cm]{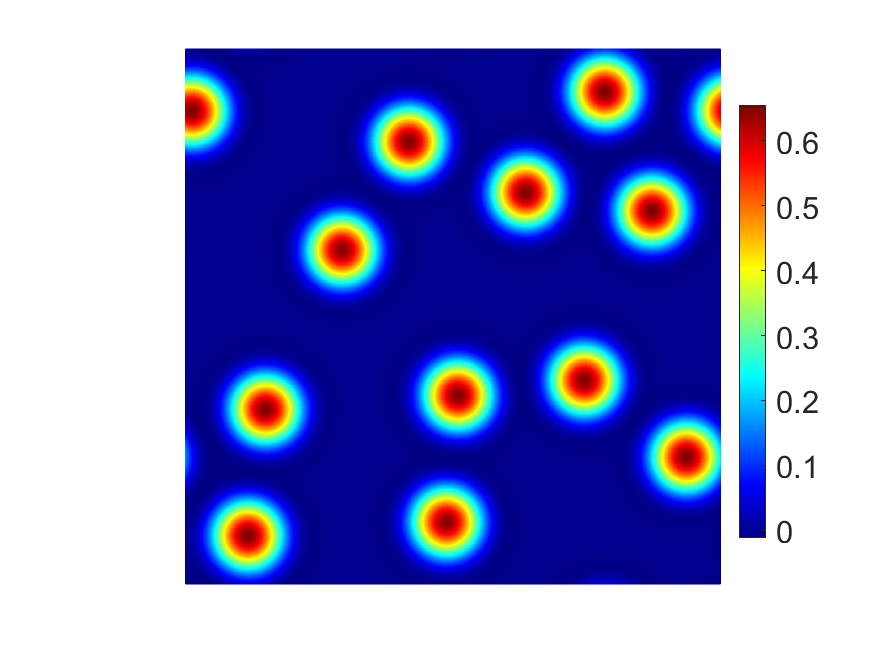}
	}
	\subfigure[$\bar{\phi}=0.10$]{
		\includegraphics[width=3.2cm,height=2.8cm]{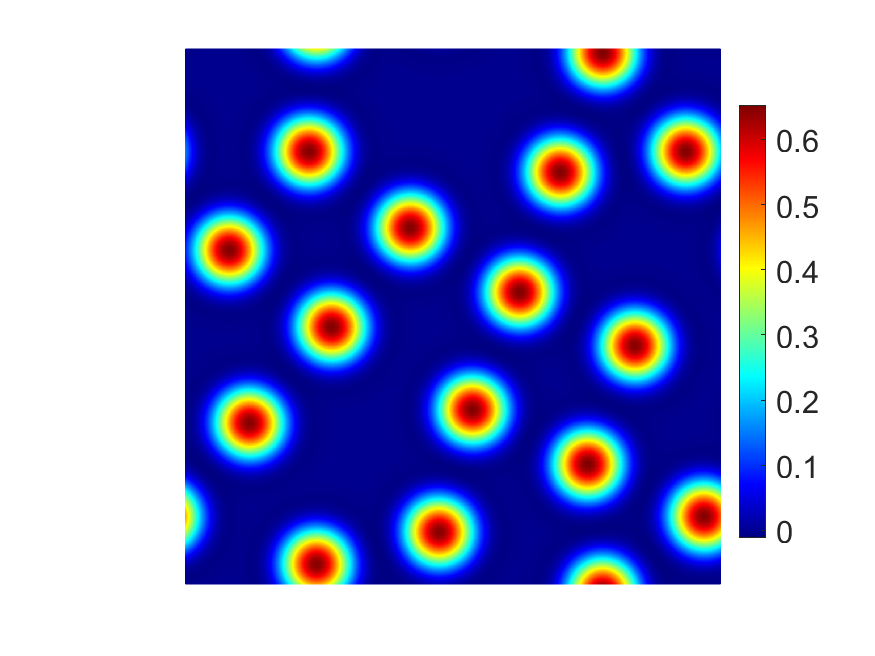}
	}
	\subfigure[$\bar{\phi}=0.12$]{
		\includegraphics[width=3.2cm,height=2.8cm]{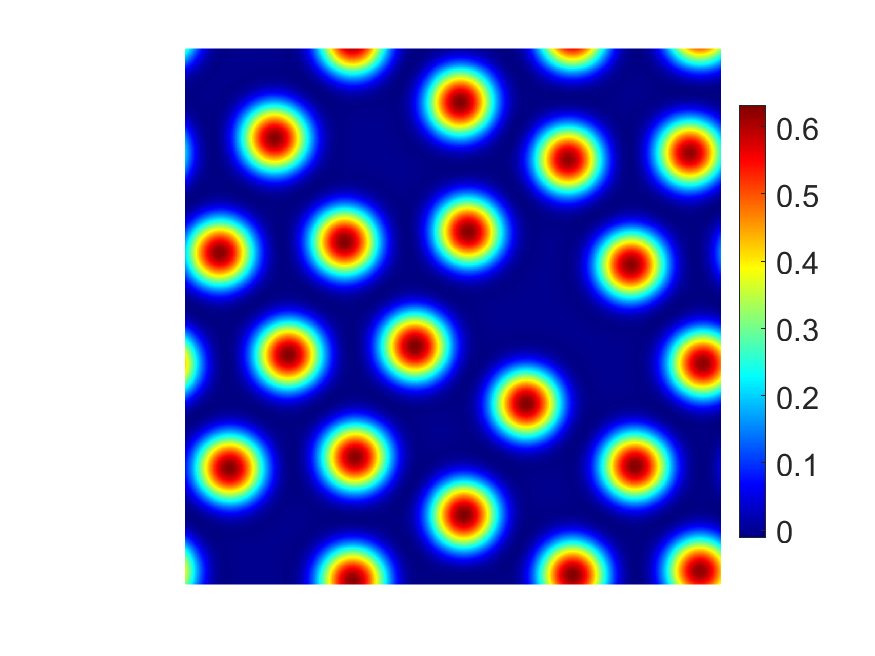}
	}
	\subfigure[$\bar{\phi}=0.18$]{
		\includegraphics[width=3.2cm,height=2.8cm]{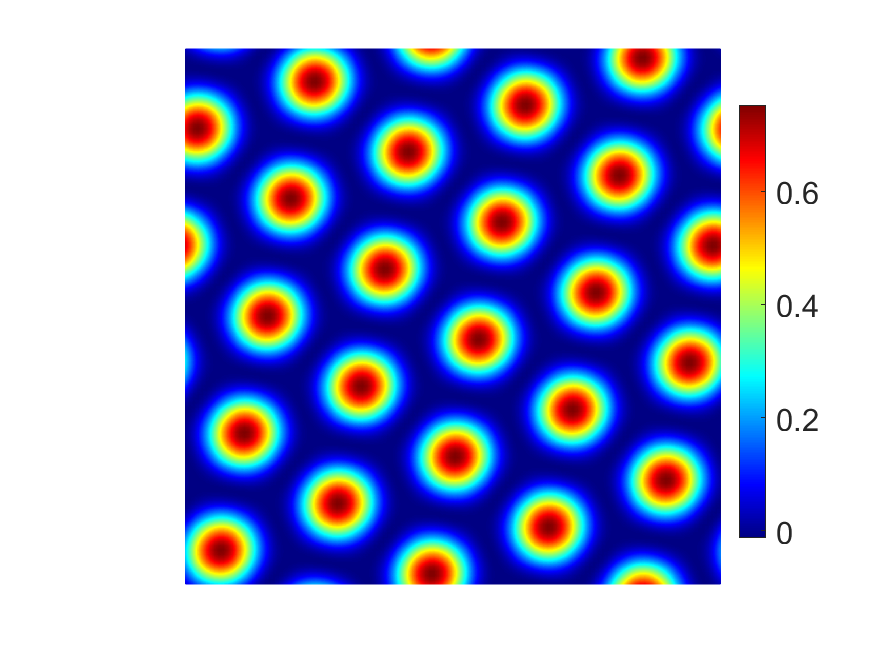}
	}
	\caption{The 2D dynamical behaviors of the phase transition with vacancy potential. Snapshots of the numerical approximation of the phase variable
		$\phi$ are taken at $t=500$ with $\bar{\phi}$ = 0.08, 0.10, 0.12, 0.18.}\label{fig:4.7}
\end{figure}

\subsubsection{Without vacancy potential}
Firstly, we simulate the phase transition dynamics without vacancy potential ($h_{vac}= 0, S=0$) in two-dimensional space. Without loss of generality, we take the SAV-CN numerical scheme \eqref{3.11}-\eqref{3.14} as example.
We set $\delta t = 0.05$ for better accuracy and use $128^{2}$ Fourier modes for spatial discretization in the 2D computational domain $\Omega = [0, 32]^{2}$. The simulation parameters are set as final time $T = 500$,  mobility coefficient $M = 1,$ and model parameters $\varepsilon = 0.9, \alpha= 1,\beta= 0.8$.
Fig.~\ref{fig:4.5} presents the snapshots of the phase variable $\phi(x, t)$ at different discrete time levels: $t = 10$ (Column 1) and $t = 500$ (Column 2), with varying mean densities $\bar{\phi}= 0.06$ (Row 1), $\bar{\phi}= 0.2$ (Row 2), and $\bar{\phi}= 0.4$ (Row 3). 

The simulations generate distinct pattern formations, including striped in Fig.~\ref{fig:4.5}(b), hybrid (stripes + triangles) in Fig.~\ref{fig:4.5}(f), and triangular in Fig.~\ref{fig:4.5}(j) configurations, in agreement with previous numerical results for MPFC model \cite{24baskaran2013energy,27qian2025error,28li2022efficient,16li2019efficient}.
Column 3 of Fig.~\ref{fig:4.5} plots the temporal evolution of three energy functionals: the original energy \eqref{2.1}, the pseudo-energy \eqref{2.9}, and the modified energy \eqref{3.3}. All of these exhibit monotonic decay, which numerically confirms the unconditional energy stability of the SAV-CN scheme. 	 	 
Column 4 of Fig.~\ref{fig:4.5} shows the temporal evolution of mass, numerically verifying the mass conservation property.
\subsubsection{With vacancy potential}
To investigate the effect of the vacancy potential, we set the parameters as $h_{vac}= 3000, S = 60,M = 1,\varepsilon= 0.9,\alpha= 1,$ and $\beta= 0.8$ to simulate the phase transition behaviors in 2D. Related numerical experiments are documented in References \cite{5chan2009molecular,31zhang2019efficient,32pei2022efficient,33zhang2023highly}. We use $128^{2}$ Fourier modes to discretize the computational domain $\Omega= [0, 32]^{2}$.
As shown in the first two columns of Fig.~\ref{fig:4.6}, snapshots at $t = 10$ and $500$ (with $\bar{\phi}= 0.06$) are presented, computed using the S-SAV-CN, S-GPAV-CN, and S-ESAV-CN schemes respectively.
Notably, these results exhibit marked differences from the $h_{vac}=0$ case (Fig.~ \ref{fig:4.5}). In this case, no periodic states are present throughout the domain, and a large number of vacancies appear.
To verify the unconditional energy stability of the three schemes, the temporal evolution curves of the original energy \eqref{2.1}, the pseudo energy \eqref{2.9}, and the modified energies (\eqref{3.3}, \eqref{3.40}, and \eqref{3.63}) are presented in the third column of Fig.~\ref{fig:4.6}. The results indicate that these three energy functions agree well, and the energy curves decay monotonically as expected.
The fourth column of Fig.~\ref{fig:4.6} displays the mass-time evolution curves for the three schemes,  numerically verifying mass conservation. Fig.~\ref{fig:4.7} presents the phase field $\phi$ distributions at $t=500$ with varying mean densities $\bar{\phi} = 0.08, 0.10, 0.12,$ and $0.18$, computed with the S-SAV-CN scheme. These results demonstrate a positive correlation between atomic density and $\bar{\phi}$. The number of atoms increases with increasing $\bar{\phi}$.
\begin{figure}[t]
	\centering
	\subfigure[$\bar{\phi}=0.04$]{
		\includegraphics[width=3.4cm,height=2.8cm]{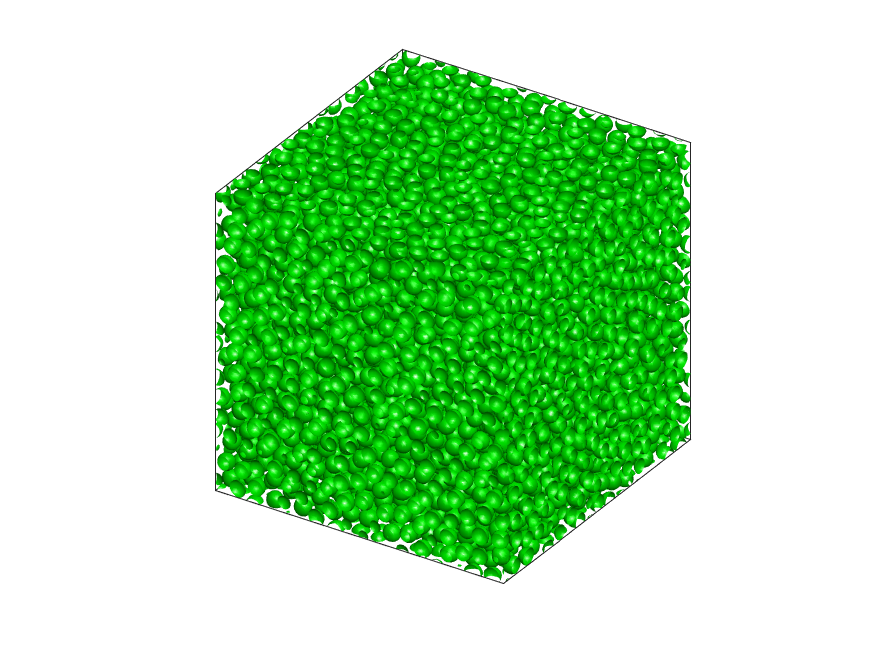}
		\includegraphics[width=3.4cm,height=2.8cm]{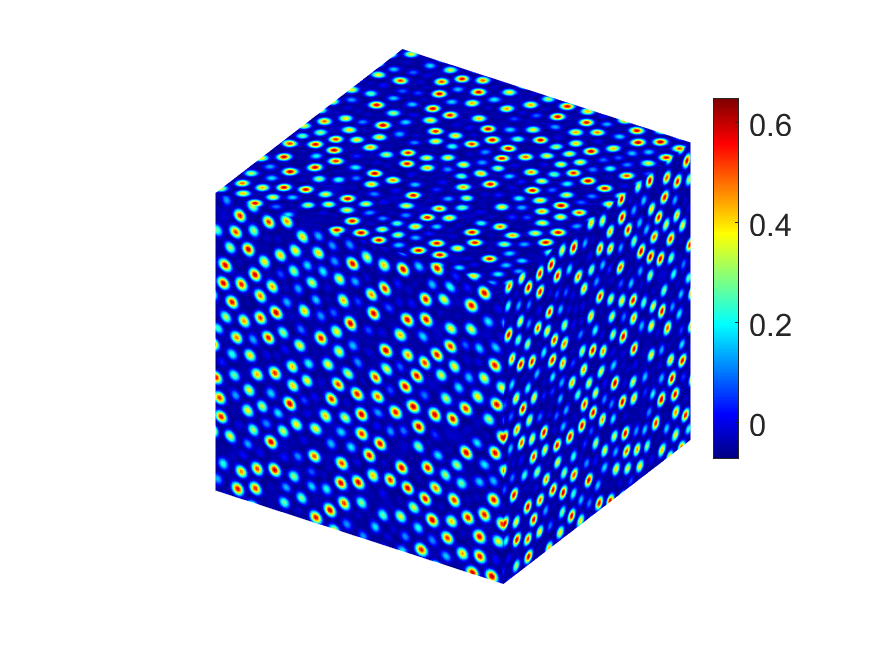}
		\includegraphics[width=3.4cm,height=2.8cm]{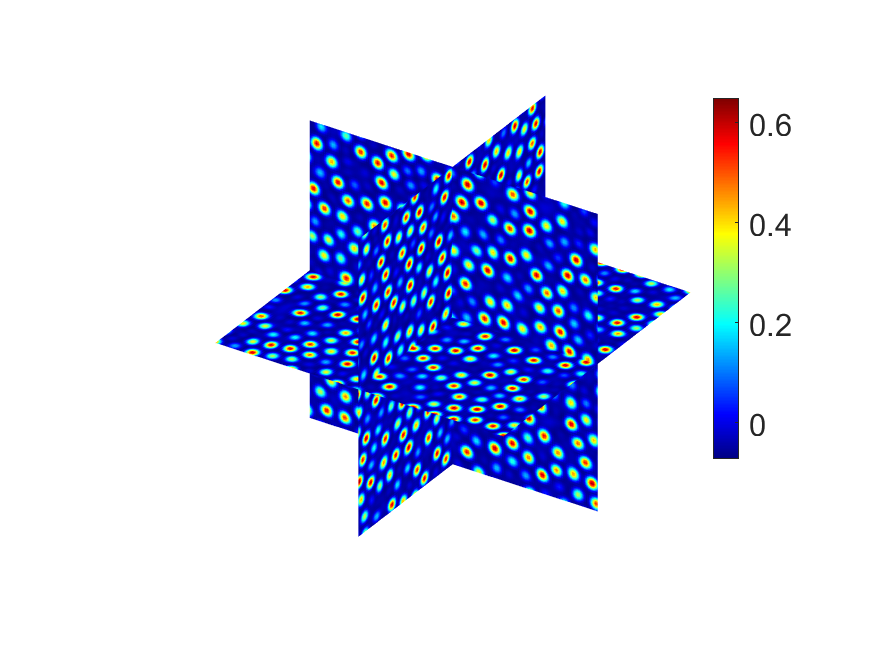}
	}
	\subfigure[$\bar{\phi}=0.06$]{
		\includegraphics[width=3.4cm,height=2.8cm]{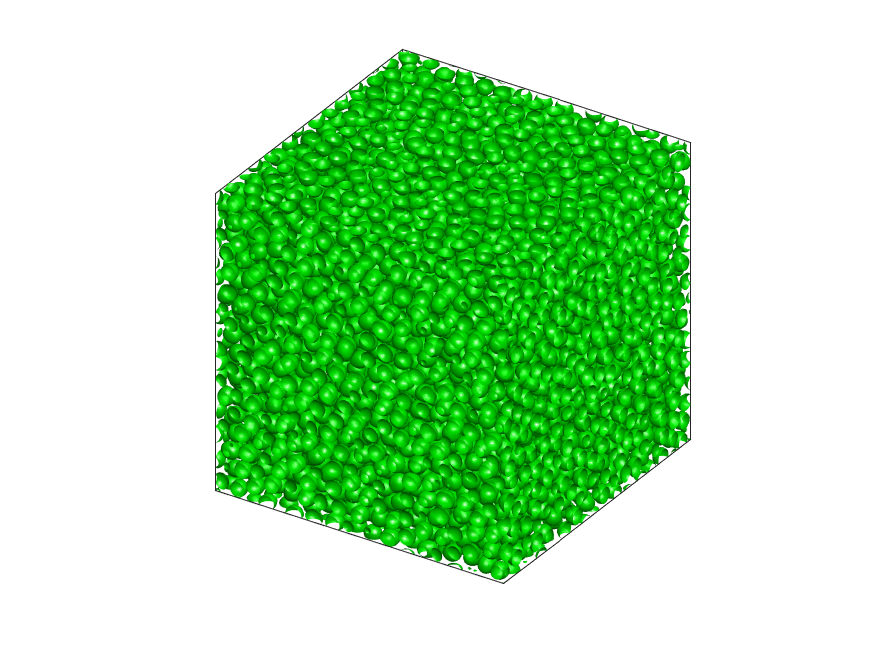}
		\includegraphics[width=3.4cm,height=2.8cm]{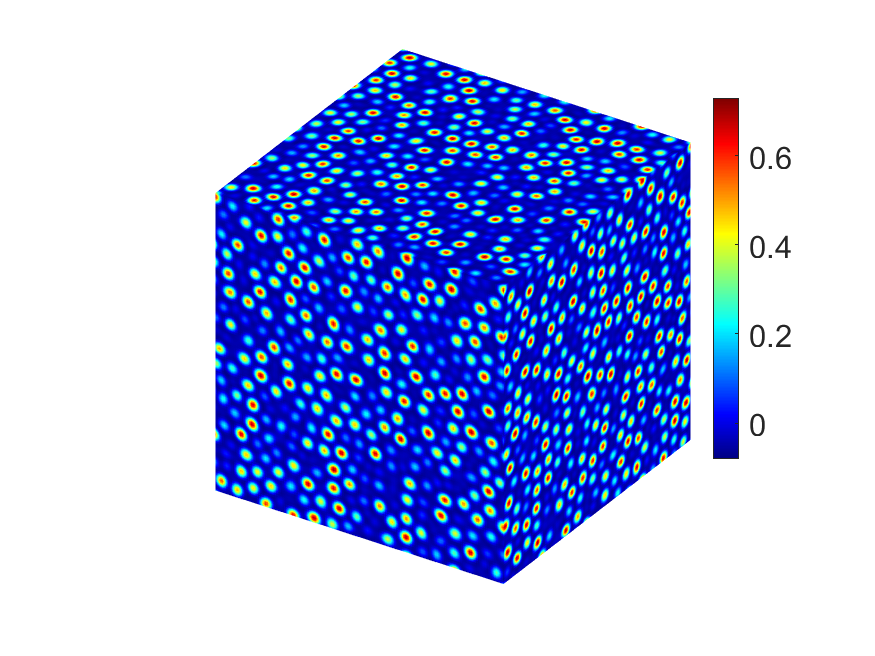}
		\includegraphics[width=3.4cm,height=2.8cm]{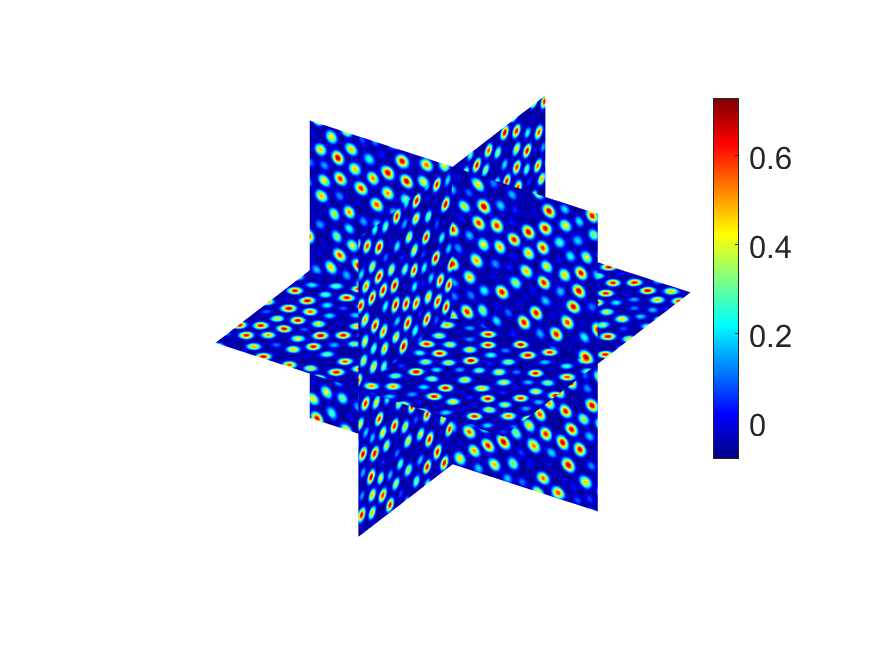}
	}
	\subfigure[$\bar{\phi}=0.10$]{
		\includegraphics[width=3.4cm,height=2.8cm]{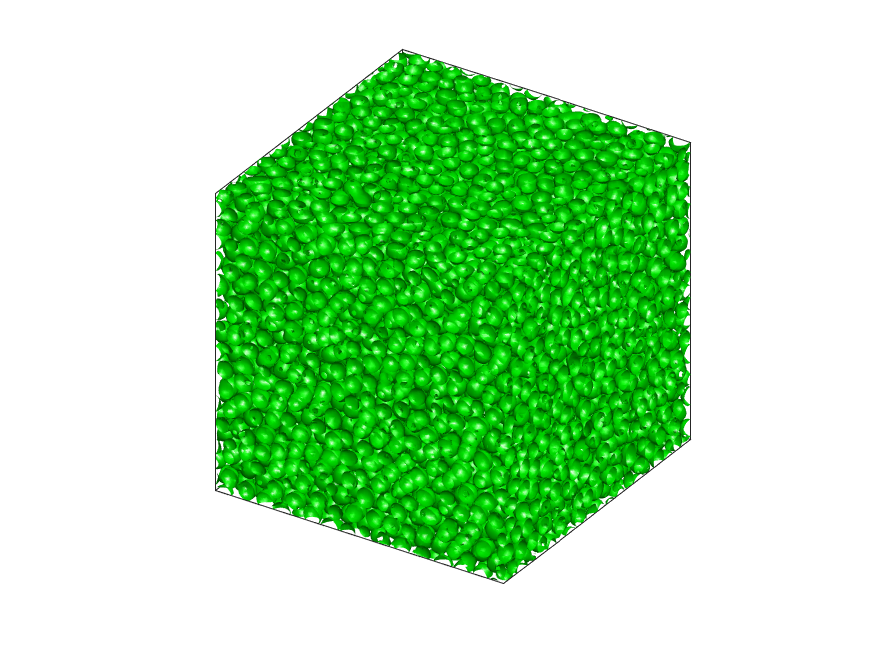}
		\includegraphics[width=3.4cm,height=2.8cm]{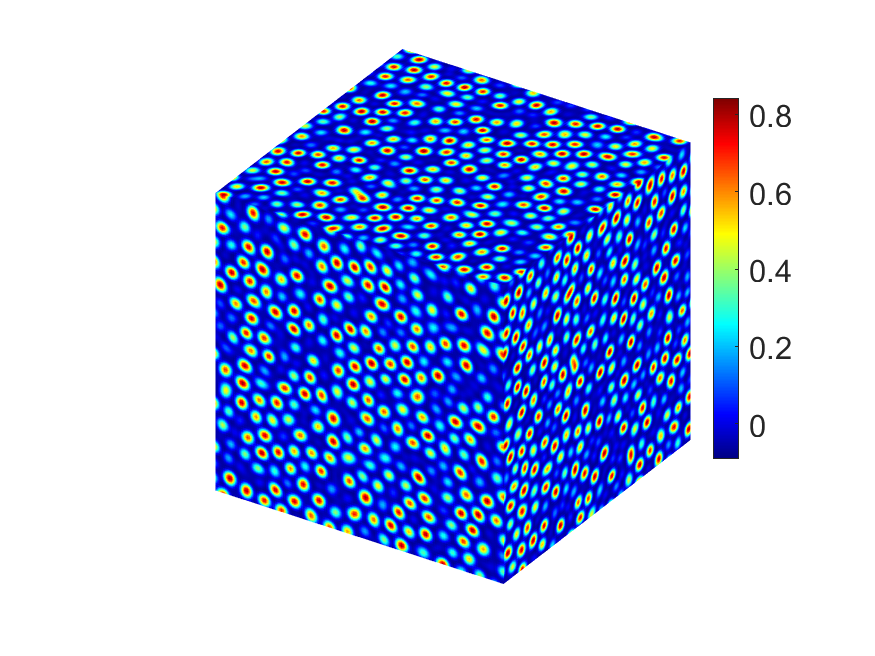}
		\includegraphics[width=3.4cm,height=2.8cm]{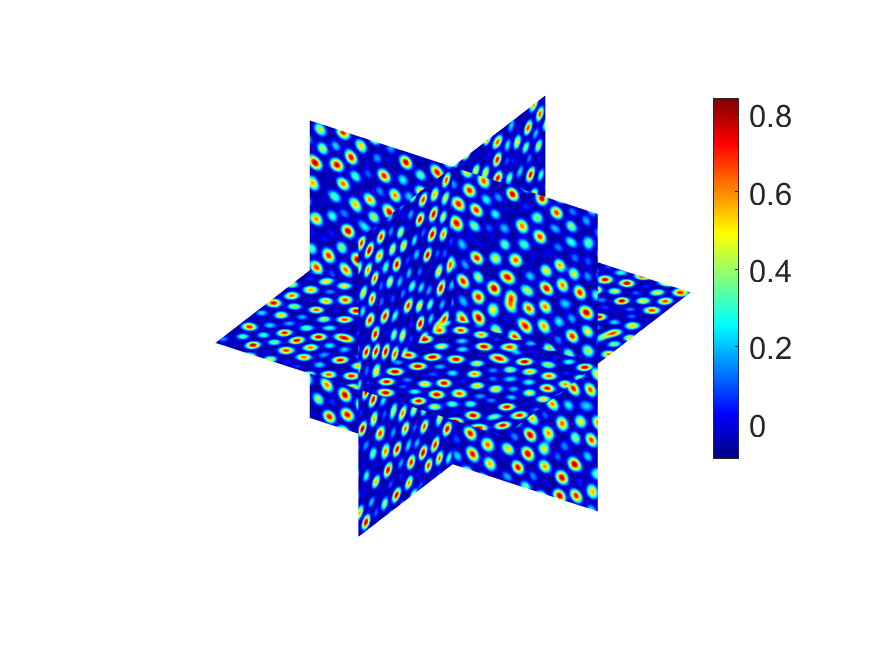}
	}	
	\caption{The steady-state microstructure of the phase transition behavior of the phase transition behavior with vacancy potential in 3D: isosurface plots of $\{\textbf{x}|\phi(\textbf{x})=0\}$ (column 1), snapshots of the phase variable $\phi$ (column 2) and three cross-section of the density field $\phi$ (column 3).}\label{fig:4.8}
\end{figure}
\begin{figure}[htb]
	\centering
	\subfigure[$\bar{\phi}=0.04$]{
		\includegraphics[width=3.4cm,height=3cm]{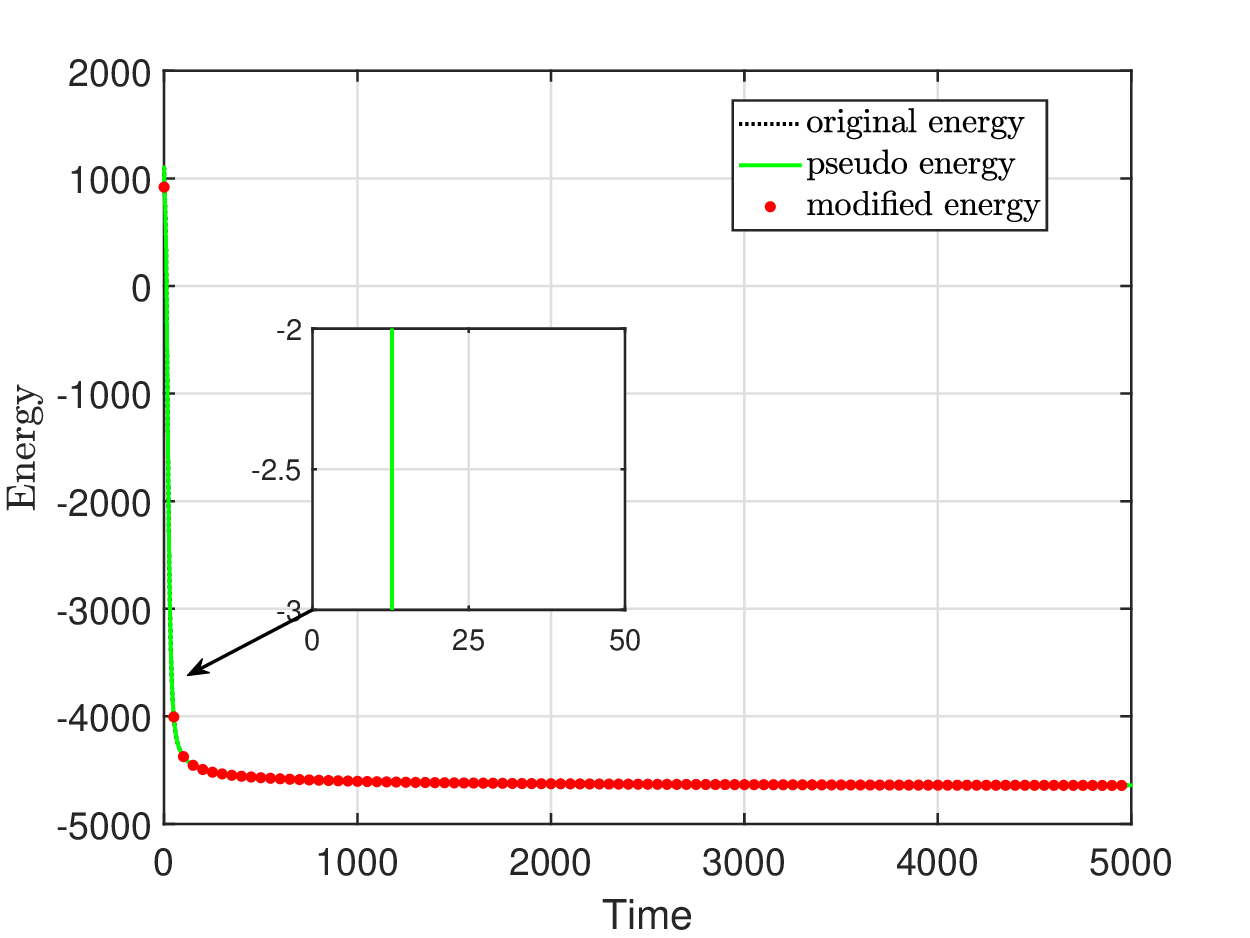}
	}
	\subfigure[$\bar{\phi}=0.06$]{
		\includegraphics[width=3.4cm,height=3cm]{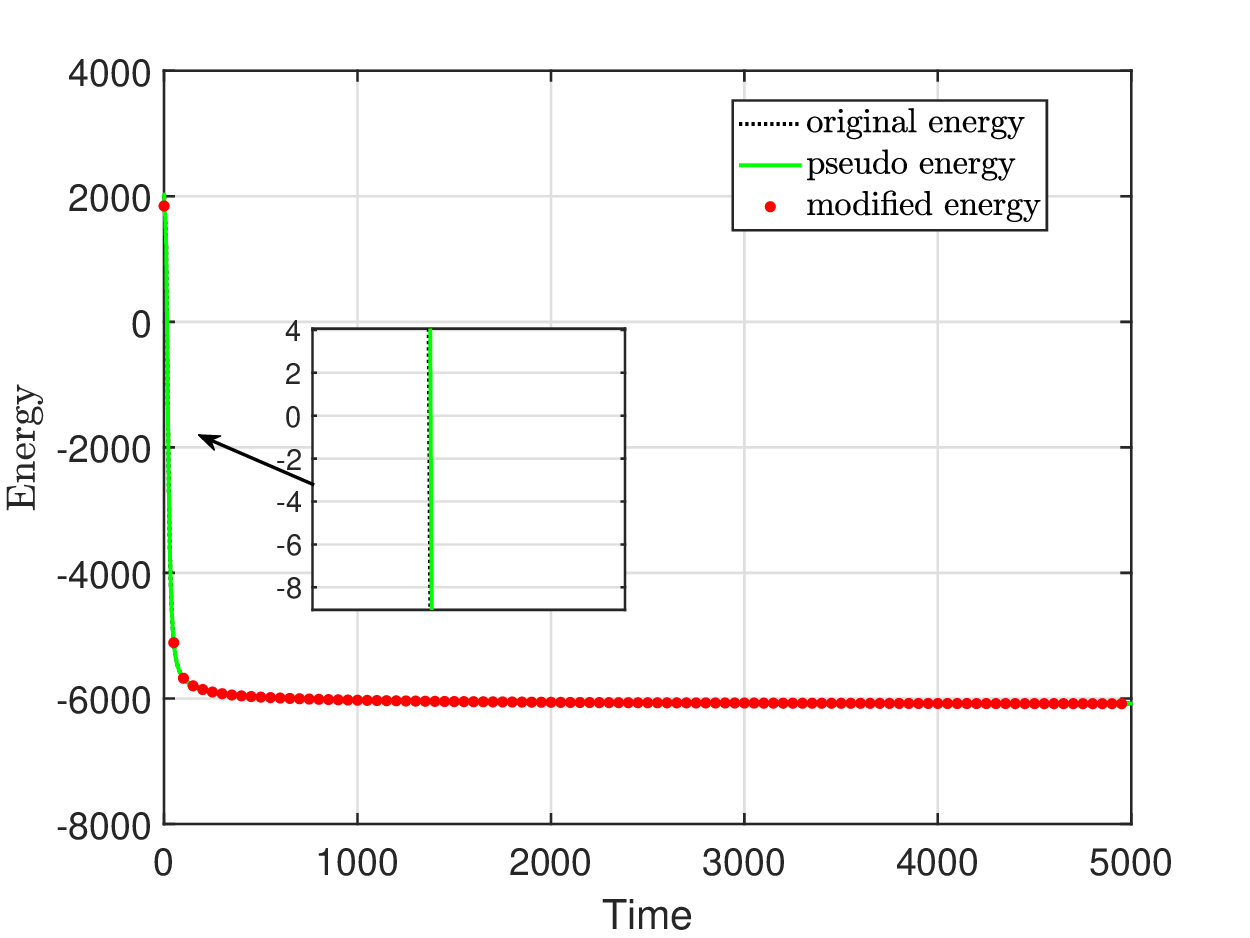}
	}
	\subfigure[$\bar{\phi}=0.10$]{
		\includegraphics[width=3.4cm,height=3cm]{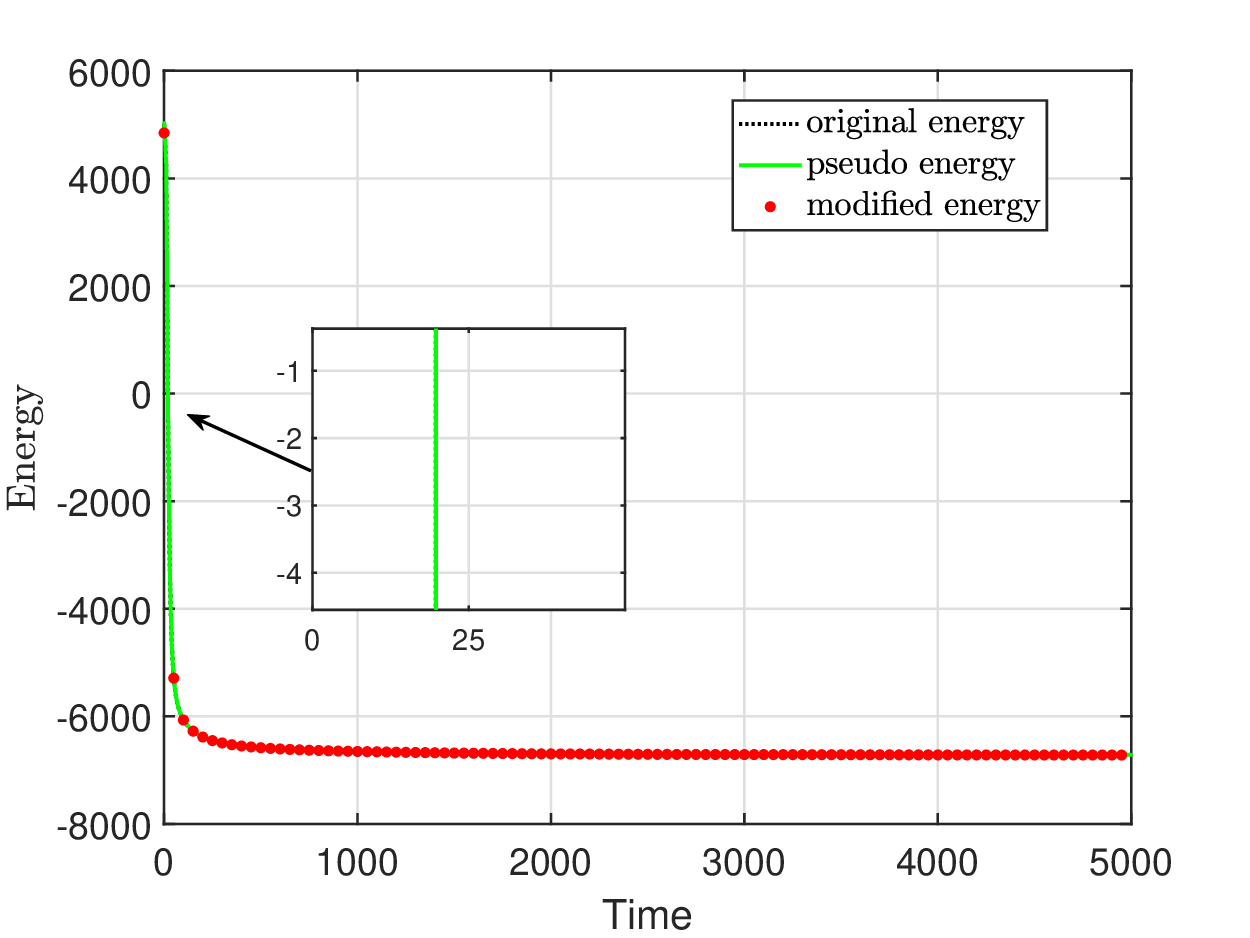}
	}
	\caption{The temporal evolution of energies of the phase transition behavior with vacancy potential in 3D with $\bar{\phi}=0.04$, $0.06$, and $0.10$, respectively.}\label{fig:4.9}
\end{figure}

Finally, we examine the phase transition simulations with vacancy potential in 3D. 
The simulations are performed in the computational domain $\Omega =[0,128]^3$ discretized using $128^3$ Fourier modes, with fixed parameters $T=5000$, $\Delta t=0.1$, $M=1$, $\alpha=0.1$, $\beta=1$, and $\epsilon=0.56$. 
Fig.~\ref{fig:4.8}(a) displays the steady-state 3D microstructure with mean density $\bar{\phi}=0.04$, showing isosurfaces $\{\textbf{x}|\phi(\textbf{x})=0\}$ in the first column, density field $\phi$ snapshots in the second column, and three cross-sectional views of $\phi$ in the third column. The results clearly demonstrate significant vacancy formation in the atomic arrangement.
When increasing the mean density to $\bar{\phi}=0.06$ (Fig.~\ref{fig:4.8} (b) and $\bar{\phi}=0.10$ (Fig.~\ref{fig:4.8}(c)), the simulations reveal density-dependent behavior of atomic number analogous to the 2D case, demonstrating substantially enhanced atomic number at higher density values $\bar{\phi}$.
Fig.~\ref{fig:4.9} further confirms the unconditional energy stability of the S-SAV-CN scheme through the monotonic temporal decay of three energy formulations: the original energy \eqref{2.1}, pseudo energy \eqref{2.9}, and modified energy \eqref{3.3}. These numerical results show complete consistency with the established phase diagram in References~\cite{5chan2009molecular,31zhang2019efficient,32pei2022efficient,33zhang2023highly,43lee2024linear}.

\begin{figure}[t]
	\centering
	\subfigure[$t=200$]{
		\includegraphics[width=3.4cm,height=2.8cm]{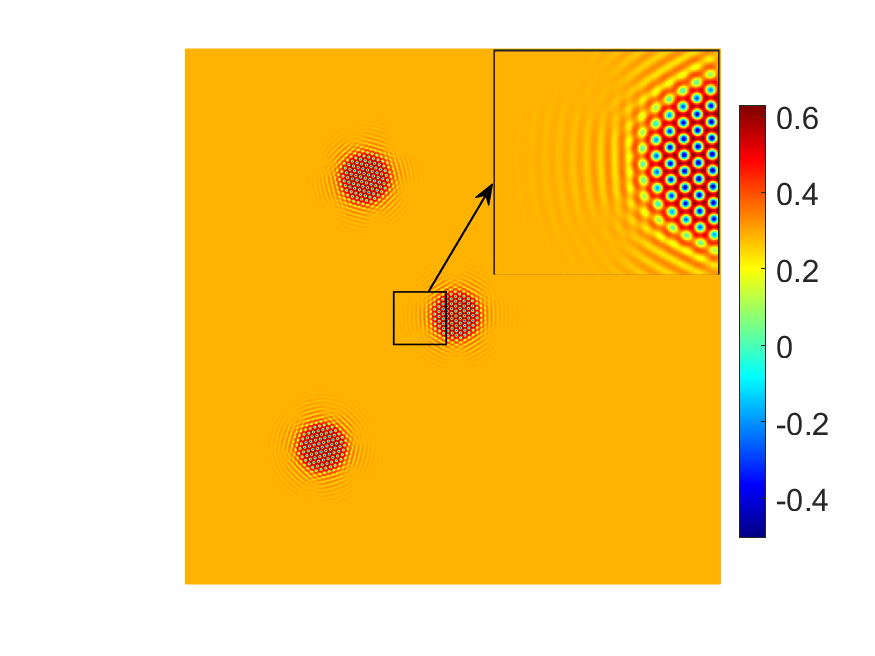}
	}
	\subfigure[$t=500$]{
		\includegraphics[width=3.4cm,height=2.8cm]{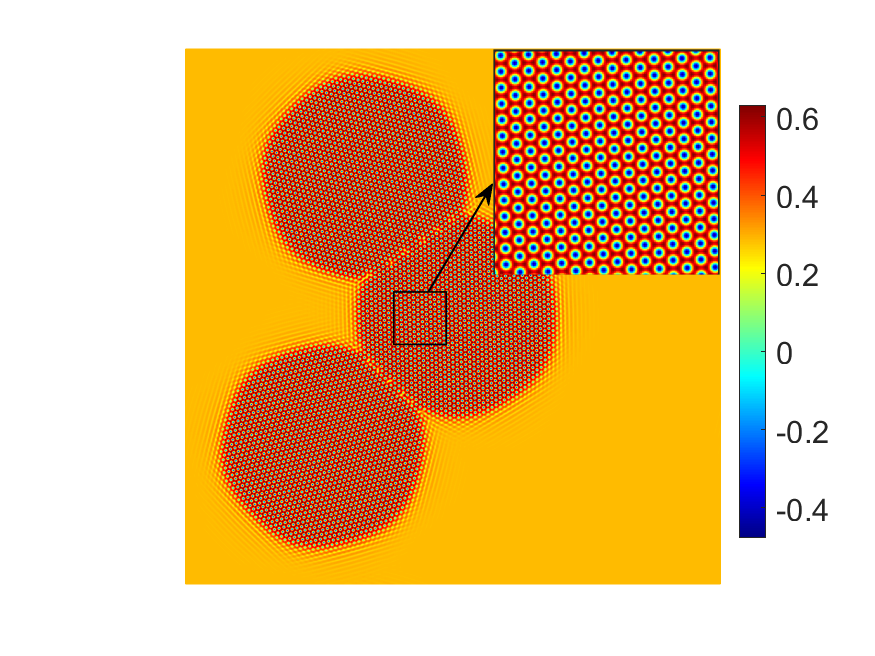}
	}
	\subfigure[$t=800$]{
		\includegraphics[width=3.4cm,height=2.8cm]{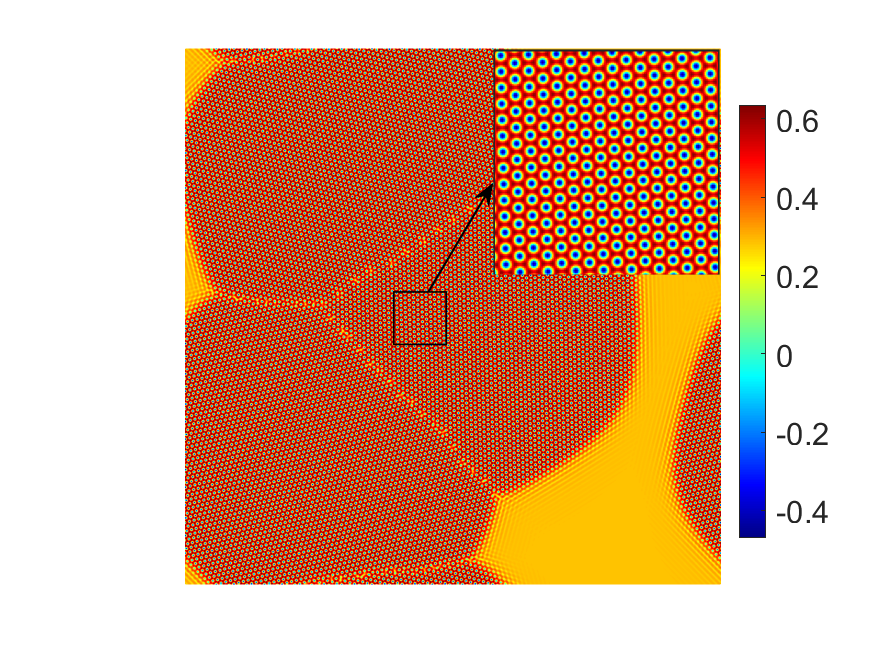}
	}
	\subfigure[$t=2000$]{
		\includegraphics[width=3.4cm,height=2.8cm]{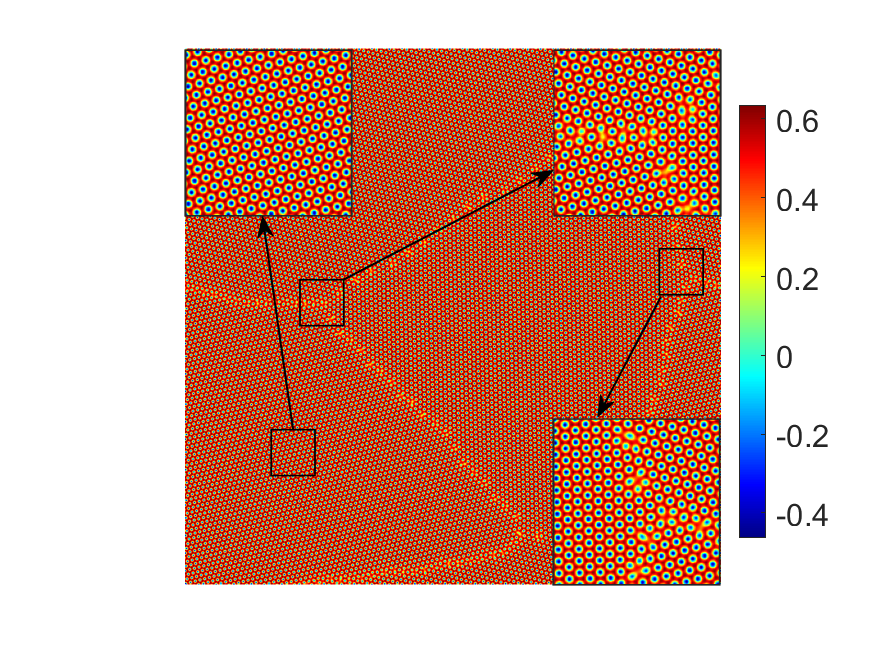}
	}
	\caption{The 2D dynamical behaviors of the crystal growth without vacancy potential. Snapshots of the numerical approximation of the phase variable
		$\phi$ are taken at $t = 200, 50, 800,$ and $2000$.}\label{fig:4.10}
\end{figure}
\begin{figure}[htb]
	\centering	
	\subfigure[$Energy$]{
		\includegraphics[width=4cm,height=3cm]{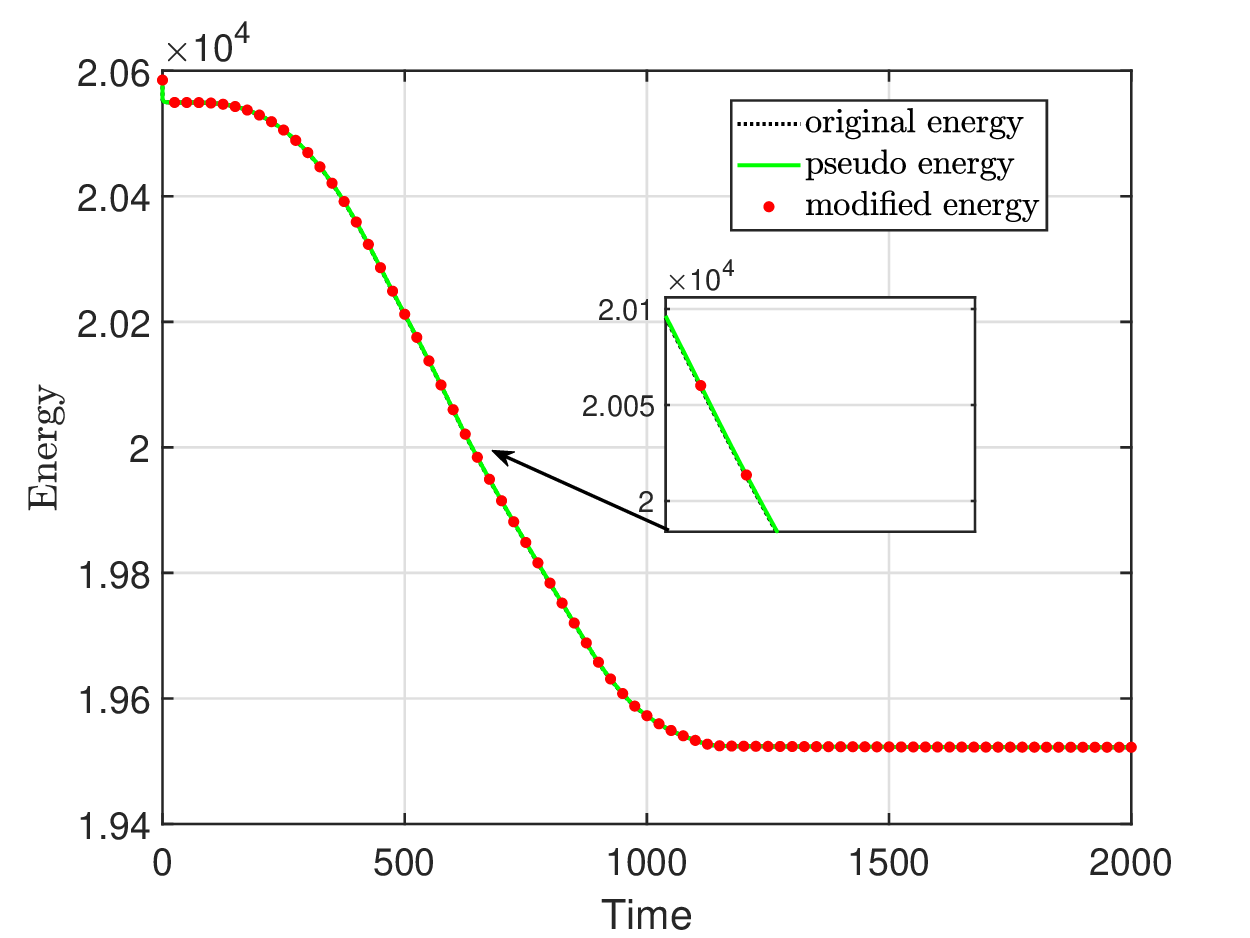}
	}	
	\subfigure[$Mass$]{
		\includegraphics[width=4cm,height=3cm]{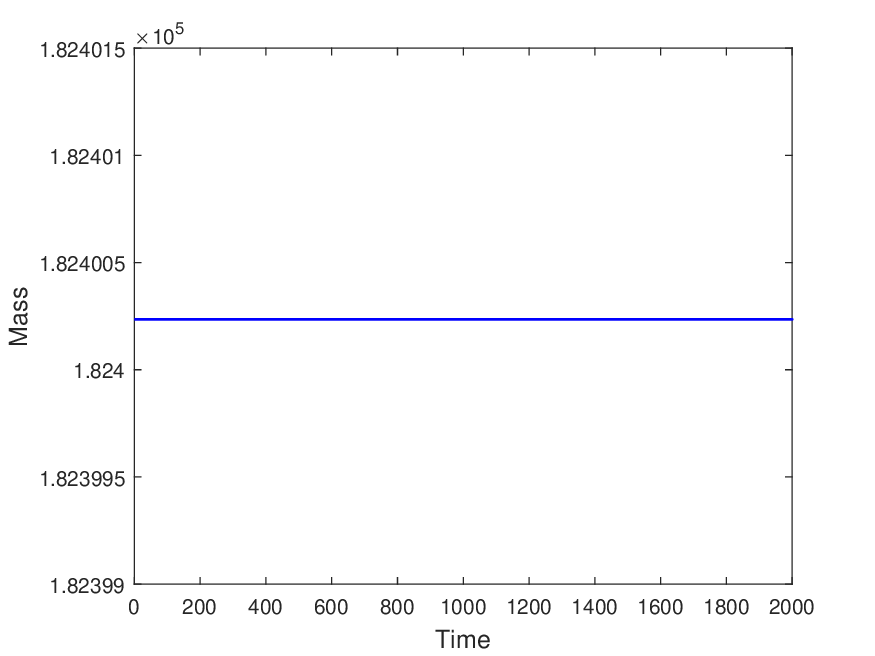}
	}
	\caption{The temporal evolution of energies and mass of the crystal growth without vacancy potential using S-GPAV-CN scheme.}\label{fig:4.11}
\end{figure}
\subsection{Crystal growth}
In this example, we simulate the poly-crystal growth in a supercooled liquid to study the dynamical process. The 2D computational domain is set to $[0, 800]^{2}$, discretized using  $1024^{2}$ Fourier modes.
We set the initial configuration $\phi(\textbf{x}, 0) =\overline{\phi}$ is a constant and then we change the values of $\overline{\phi}$ by setting three crystallites in small square patches.
We define each small crystallite as follows:
$$\phi(x_{l},y_{l})=\overline{\phi}+C(cos(\frac{q}{\sqrt{3}}y_{l})cos(qx_{l})-0.5cos(\frac{2q}{\sqrt{3}}y_{l})),\quad l=1,2,3,  $$
where $(x_{l}$, $y_{l})$ define a local system of Cartesian coordinates using an affine transformation of the global coordinates $(x, y)$ by $x_{l}(x,y)=xsin\theta+ycos\theta, y_{l}(x,y)=-xcos\theta+ysin\theta.$
Here, we choose $\theta= 0, -\pi/4$ and $\pi/4$ to generate three crystallite lattices with different orientations.
\begin{figure}[t]
	\centering
	\subfigure[$S-SAV-CN$]{
		\includegraphics[width=4cm,height=3cm]{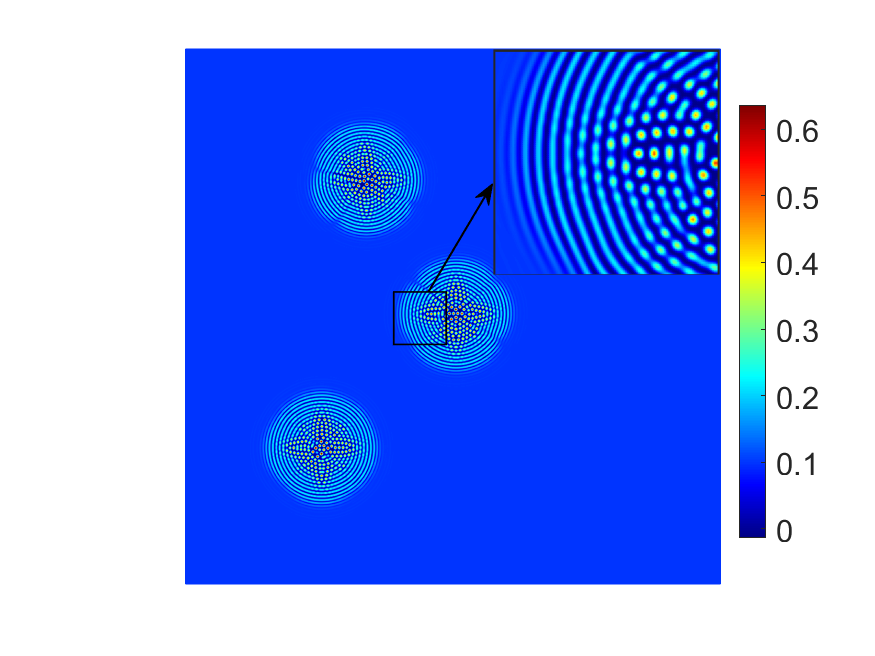}
		\includegraphics[width=4cm,height=3cm]{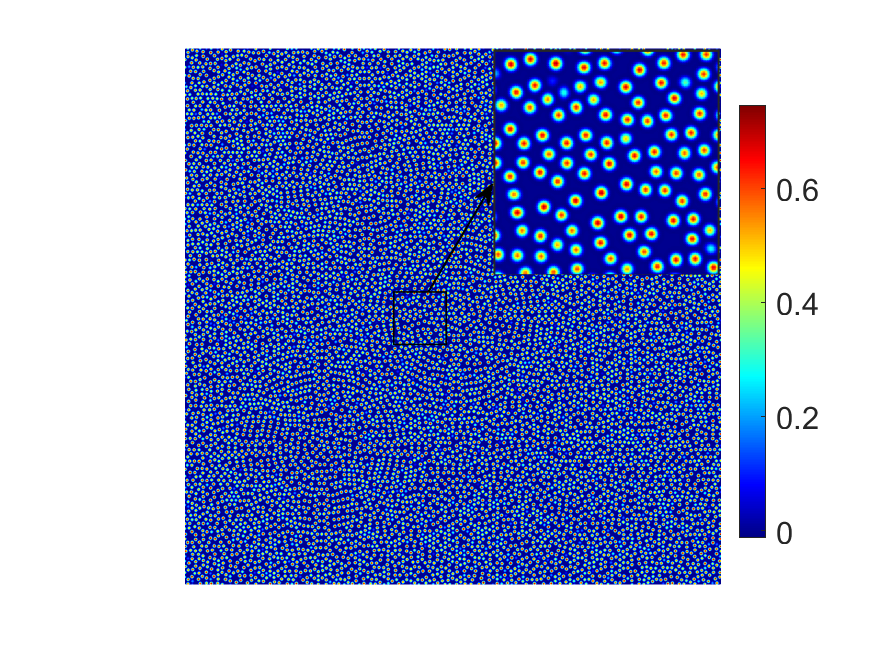}
		\includegraphics[width=4cm,height=3cm]{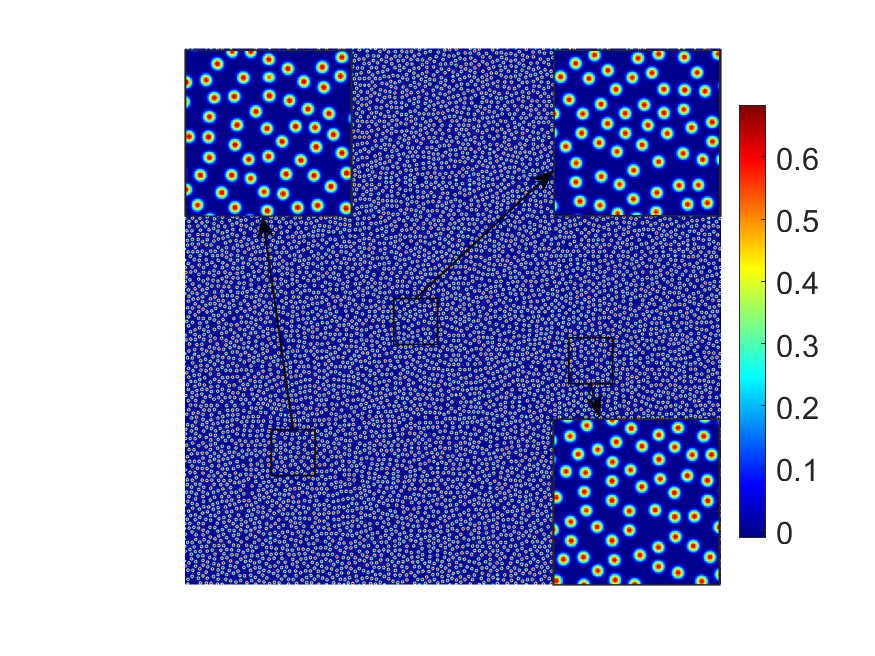}
	}
	\subfigure[$S-GPAV-CN$]{
		\includegraphics[width=4cm,height=3cm]{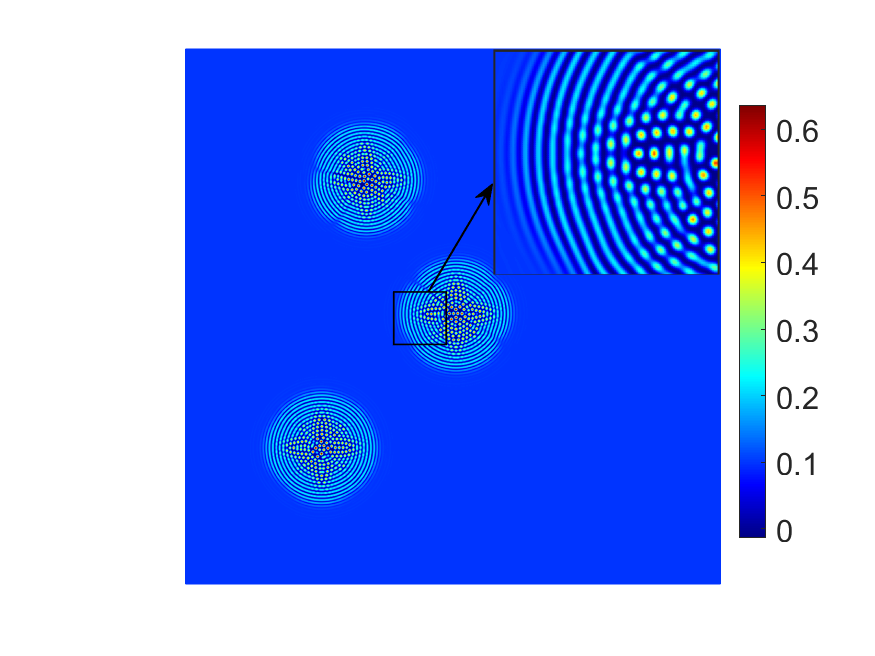}
		\includegraphics[width=4cm,height=3cm]{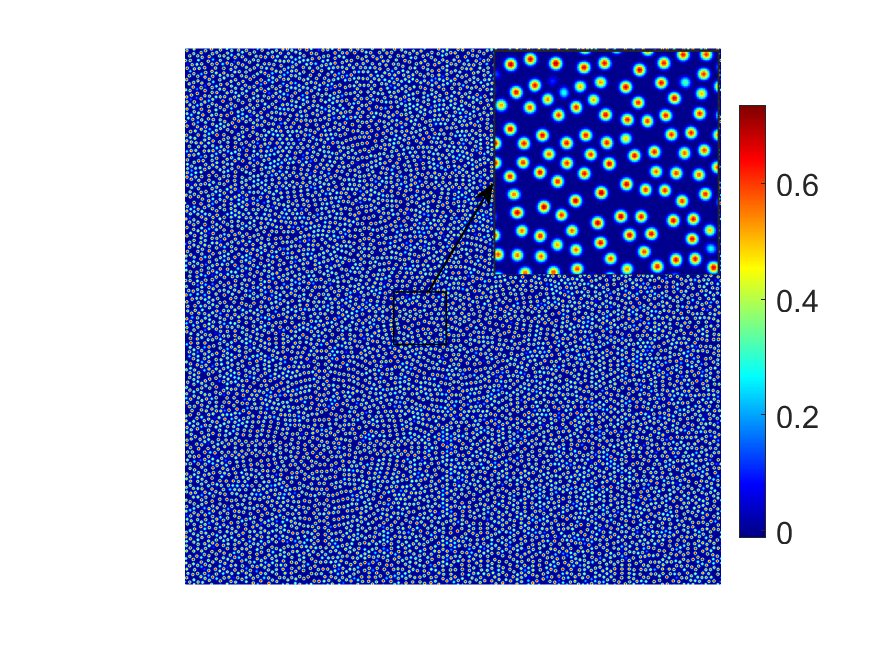}
		\includegraphics[width=4cm,height=3cm]{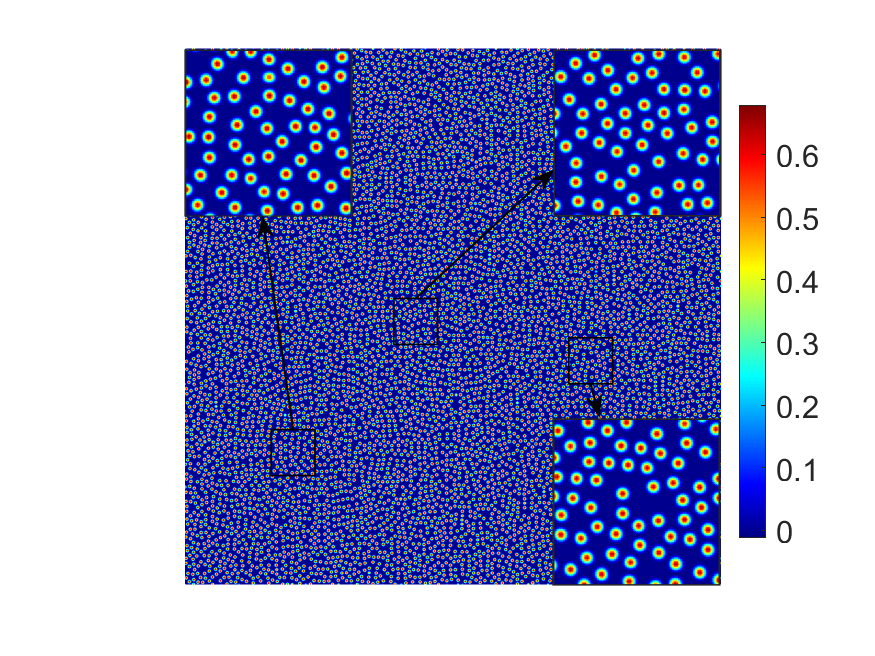}
	}	
	\subfigure[$S-ESAV-CN$]{
		\includegraphics[width=4cm,height=3cm]{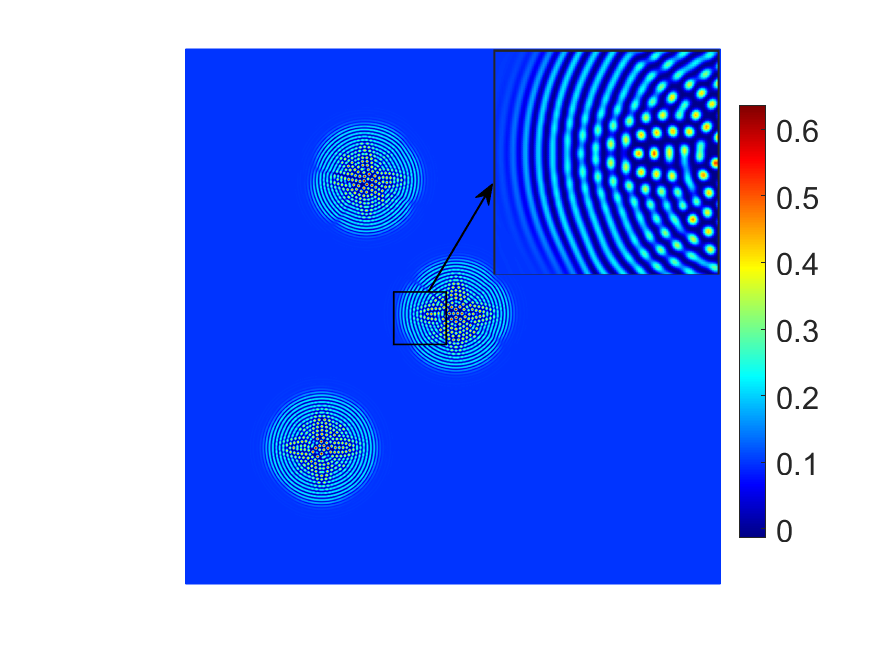}
		\includegraphics[width=4cm,height=3cm]{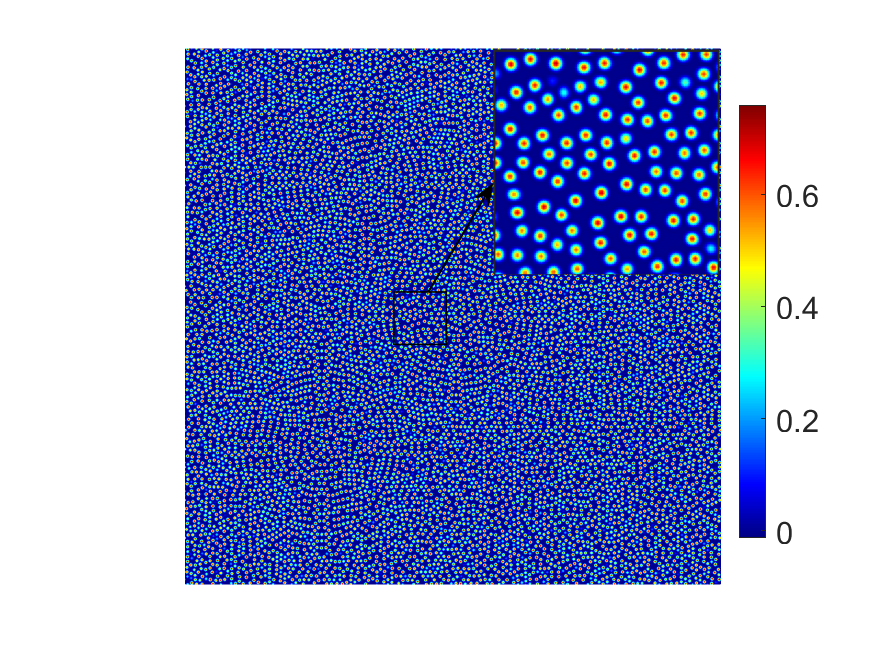}	
		\includegraphics[width=4cm,height=3cm]{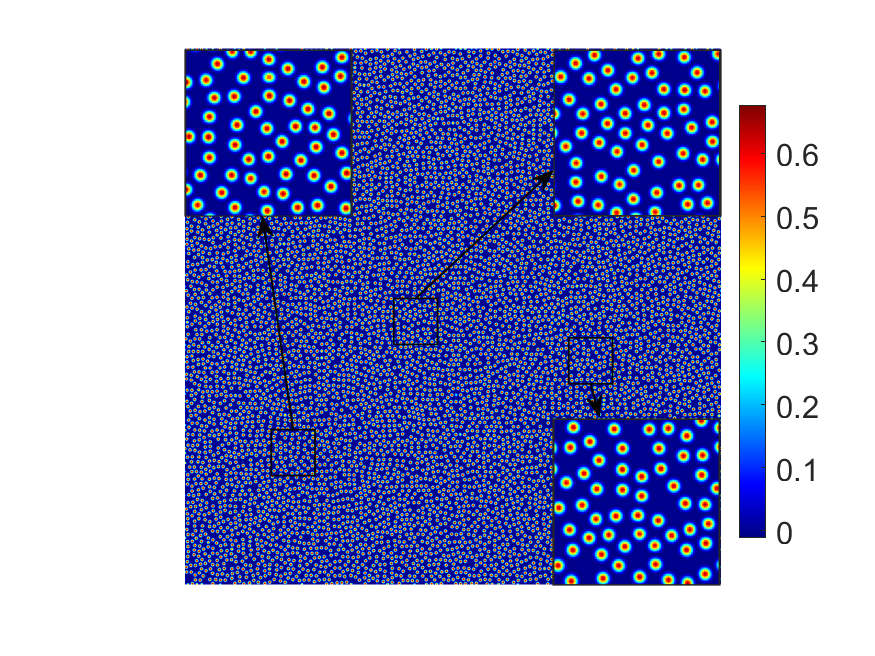}
	}
	\caption{The 2D dynamical behaviors of the crystal growth with vacancy potential. Snapshots of the numerical approximation of the phase variable
		$\phi$ are taken at $t = 30, 100,$ and $800$ using S-SAV-CN(Row 1), S-GPAV-CN(Row 2), and S-ESAV-CN(Row 3) schemes.}\label{fig:4.12}
\end{figure}
\subsubsection{Crystal growth without vacancy potential}
To investigate the crystallization dynamics leading to periodic lattice formation, we first disable vacancy effects by setting $h_{\text{vac}}=0$ and $S=0$, and other parameters are fixed as $\epsilon=0.25, \alpha=1, \beta=1, M=1, B=10^8, C=0.446, \overline{\phi}=0.285, q=0.66, T=2000, \Delta t=0.1.$ The simulations use the second-order GPAV-CN scheme.
Fig.~\ref{fig:4.10} illustrates the growth process of three crystallites, capturing the crystal-liquid interfacial dynamics under varying orientation alignments through phase-field solutions $\phi$ at discrete time levels $t=200, 500, 800,$ and $2000$. 
The grain boundaries are clearly visible, which is consistent with the results reported in \cite{24baskaran2013energy,27qian2025error,16li2019efficient}.
In Fig.~\ref{fig:4.11}(a), we plot the temporal evolution of the original energy \eqref{2.1}, the pseudo energy \eqref{2.9}, and the modified energy \eqref{3.40}. All three curves exhibit monotonic decay with excellent mutual agreement, where an inset provides magnified visualization. Fig.~\ref{fig:4.11}(b) verifies mass conservation throughout the simulation duration.
These numerical results demonstrate that the GPAV-CN scheme is unconditionally energy stable and mass conservative.
\subsubsection{Crystal growth with vacancy potential}
To systematically examine vacancy effects on crystal growth, we conduct simulations using the parameters $h_{vac} = 3000, S=60 ,\epsilon= 0.9,\alpha =1,\beta= 0.8, M=1, C = 0.000446,\overline{\phi} =0.1, q=0.66.$ In Fig.~\ref{fig:4.12}, we plot the snapshots of the phase field variable $\phi$ at different discrete time levels $t=30, 100$, and $800$, using S-SAV-CN, S-GPAV-CN and S-ESAV-CN schemes.
The simulation results demonstrate atomic formation accompanied by complete disappearance of periodic ordering and ubiquitous vacancy distribution, exhibiting fundamentally distinct morphology compared to the vacancy-free system in Fig.~\ref{fig:4.10}. All of the numerical results are consistent with the phase diagram in References~\cite{31zhang2019efficient,32pei2022efficient,33zhang2023highly}.
\begin{figure}[htb]
	\centering
	\subfigure[Energy evolution]{
		\includegraphics[width=4cm,height=3cm]{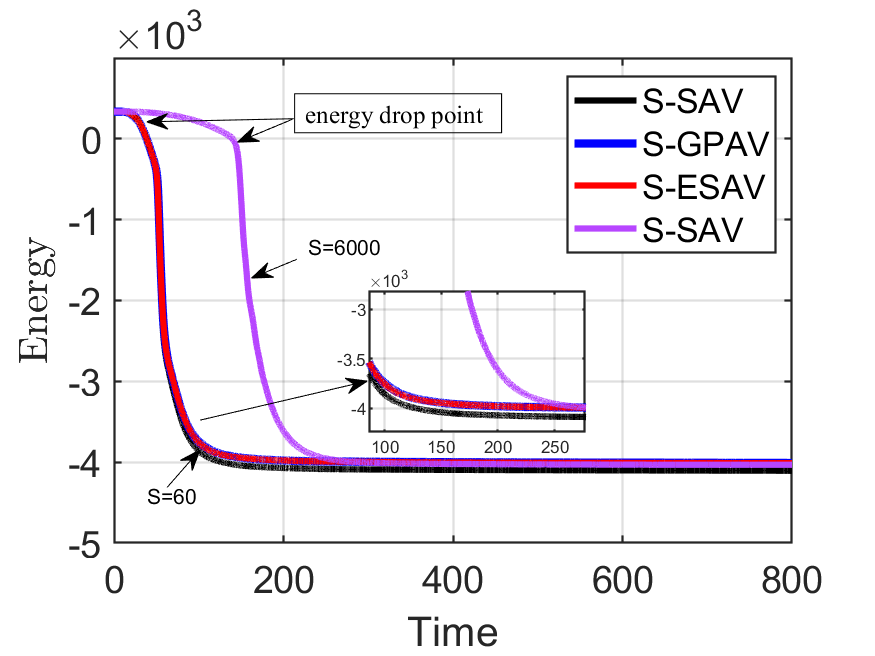}
	}
	\subfigure[Mass evolution]{
		\includegraphics[width=4cm,height=3cm]{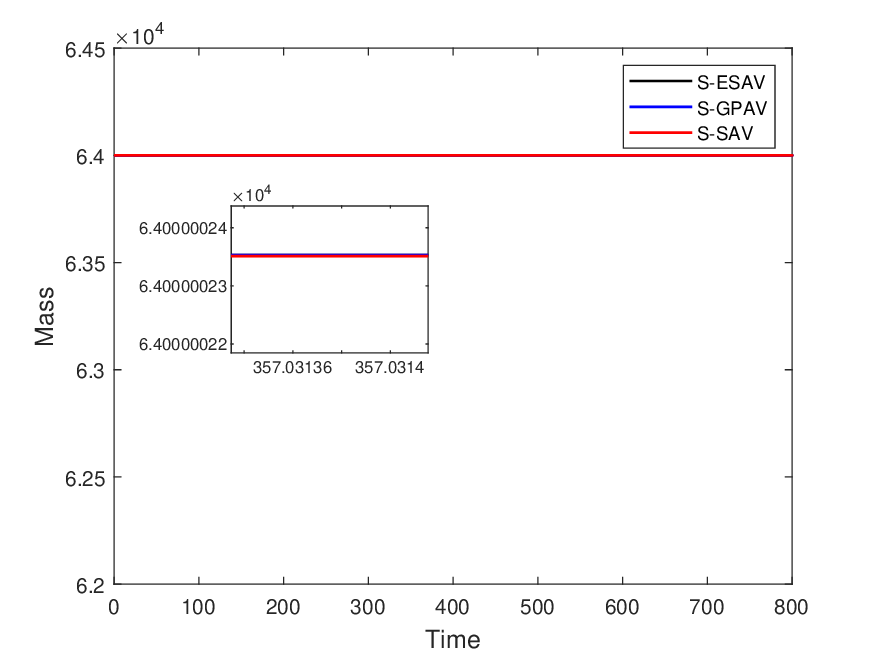}
	}
	\caption{The temporal evolution of the energy $\widetilde{E}$ and mass of the crystal growth with vacancy potential.}\label{fig:4.13}	
\end{figure}
\begin{figure}[htb]
	\centering	
	\subfigure[$\phi_{(S-GPAV)} - \phi_{(S-SAV)}$]{
		\includegraphics[width=3.4cm,height=3cm]{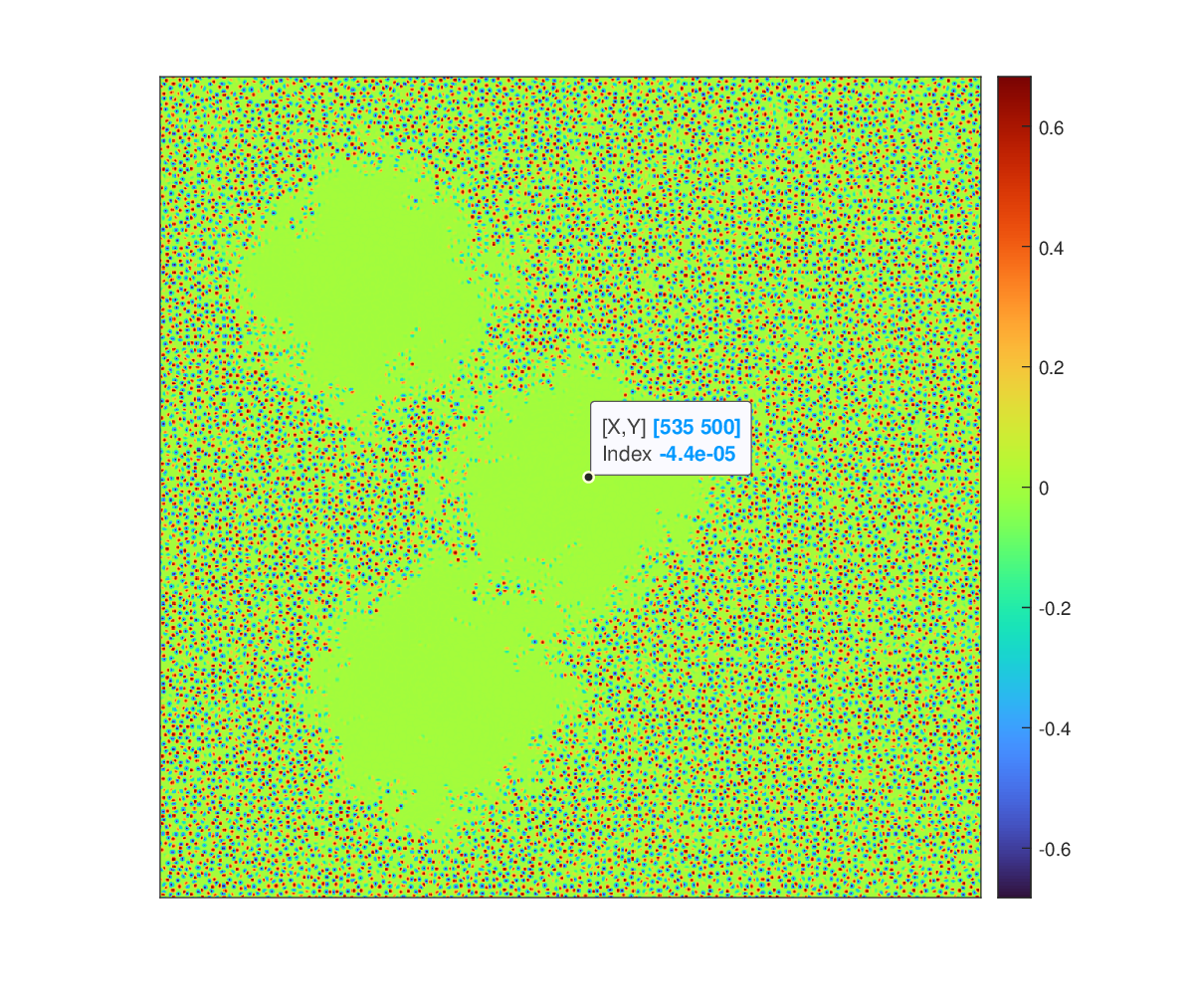}
	}	
	\subfigure[$\phi_{(S-ESAV)} - \phi_{(S-SAV)}$]{
		\includegraphics[width=3.4cm,height=3cm]{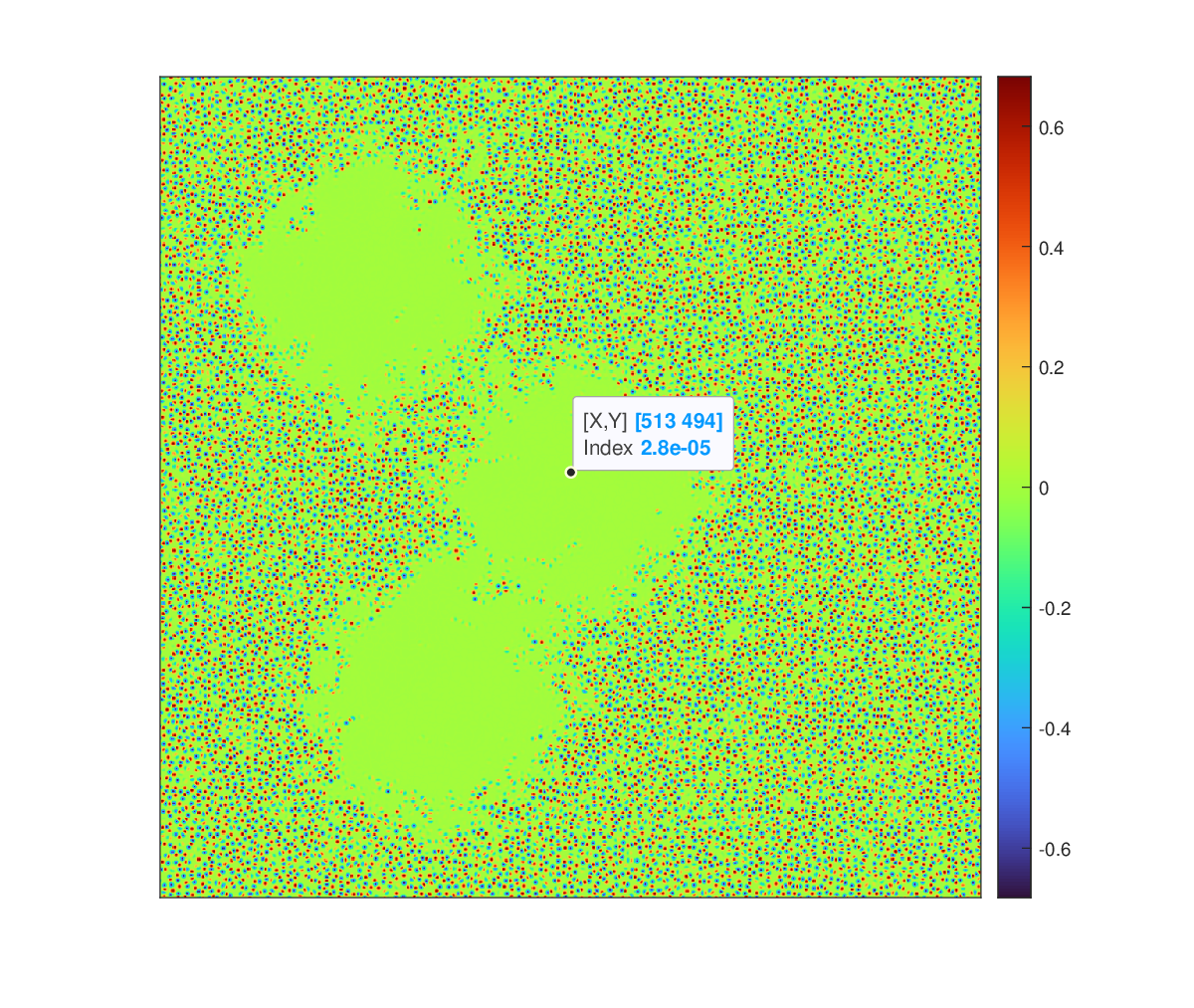}
	}
	\subfigure[Energy  Differences]{
		\includegraphics[width=3.4cm,height=3cm]{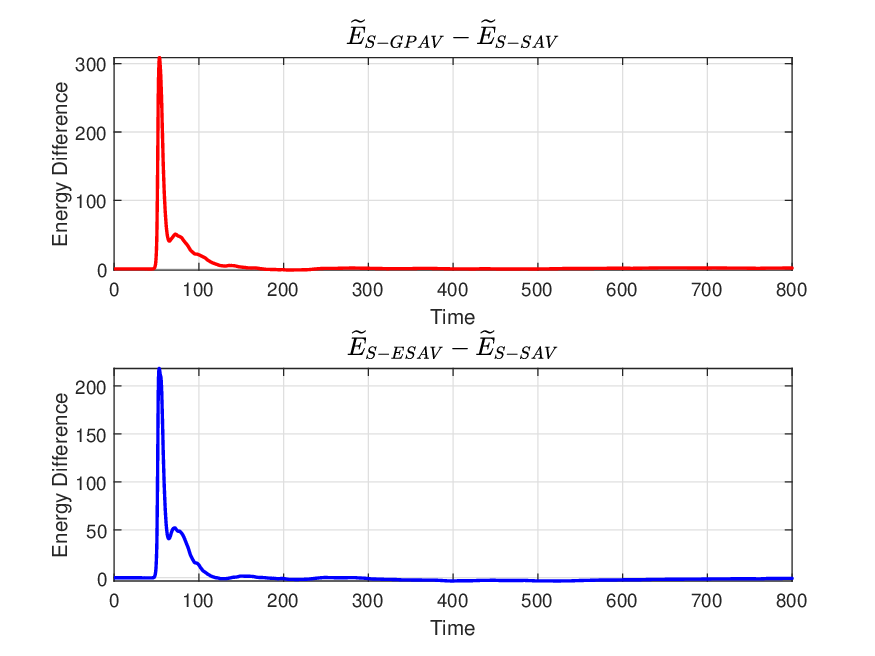}
	}	
	\subfigure[Mass  Differences]{
		\includegraphics[width=3.4cm,height=3cm]{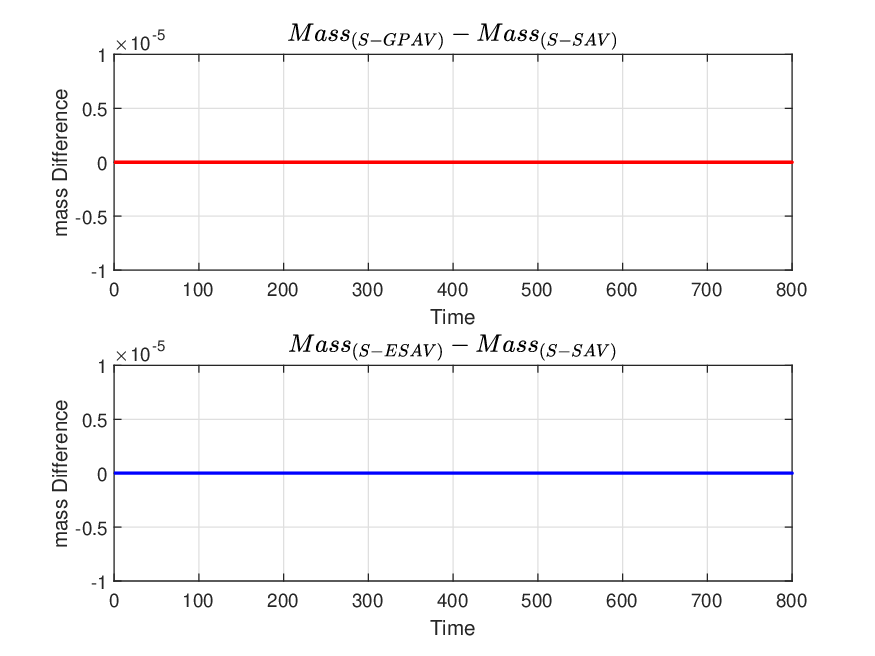}
	}
	\caption{Differences of the phase field variable $\phi$, energy $\widetilde{E}$, and mass }\label{fig:4.14}
\end{figure}
\begin{figure}[htb]
	\centering
	\subfigure[$\Delta t= 0.015,S=100$]{
		\includegraphics[width=3.4cm,height=2.8cm]{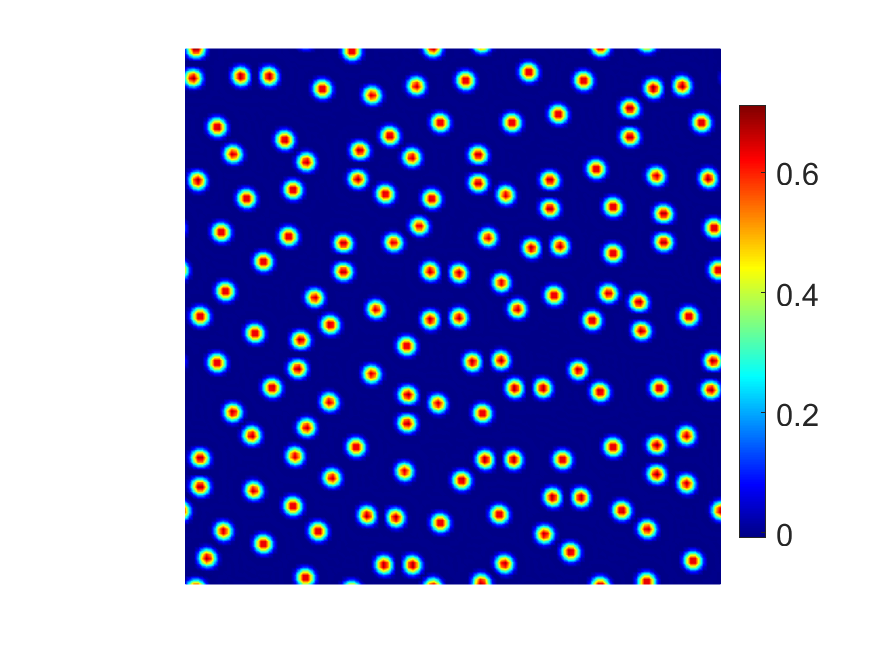}
	}
	\subfigure[$\Delta t= 0.001,S=0$]{
		\includegraphics[width=3.4cm,height=2.8cm]{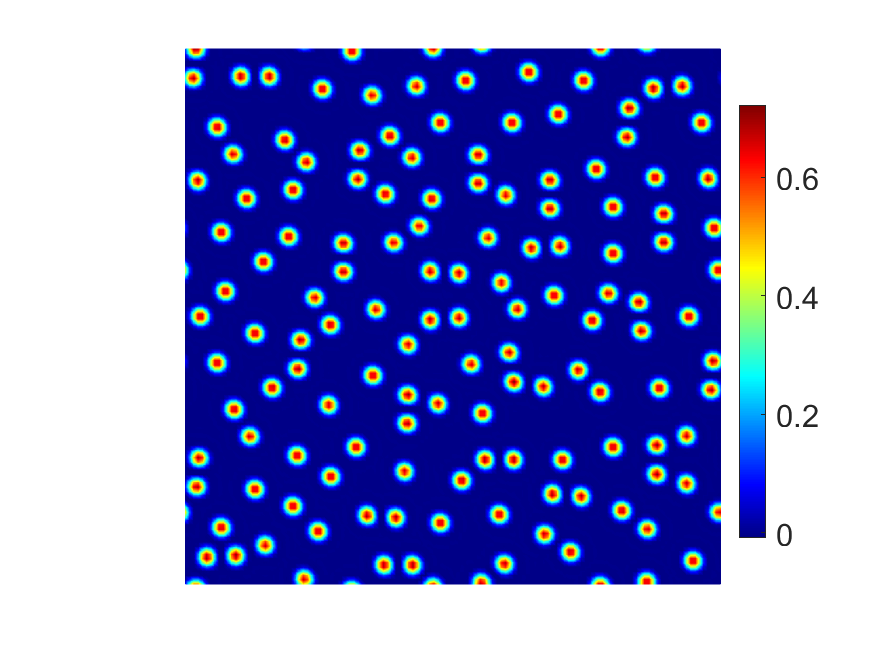}	
	}
	\subfigure[$\Delta t= 0.015,S=0$]{
		\includegraphics[width=3.4cm,height=2.8cm]{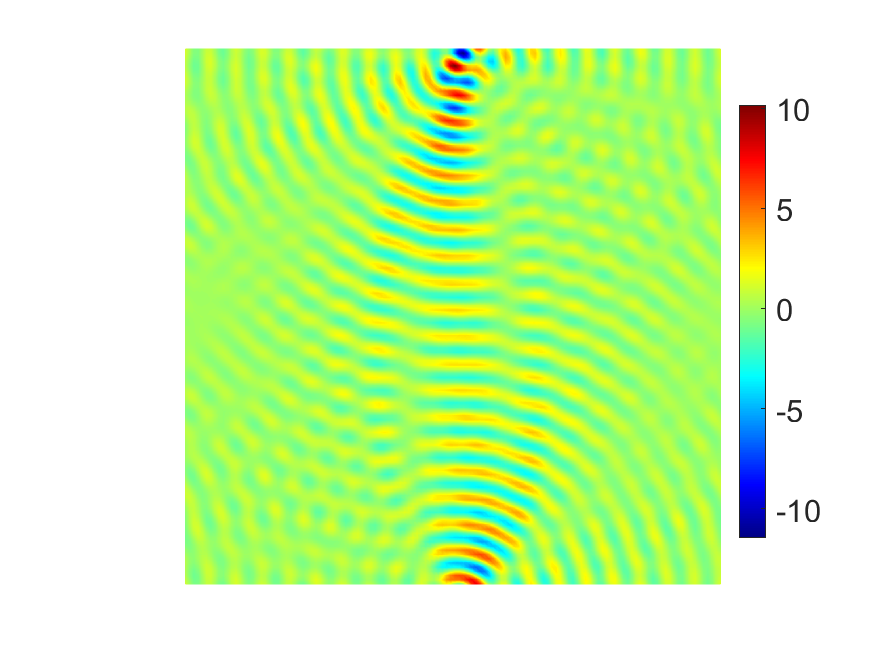}
	}
	\subfigure[$S=100$]{
		\includegraphics[width=3.4cm,height=2.8cm]{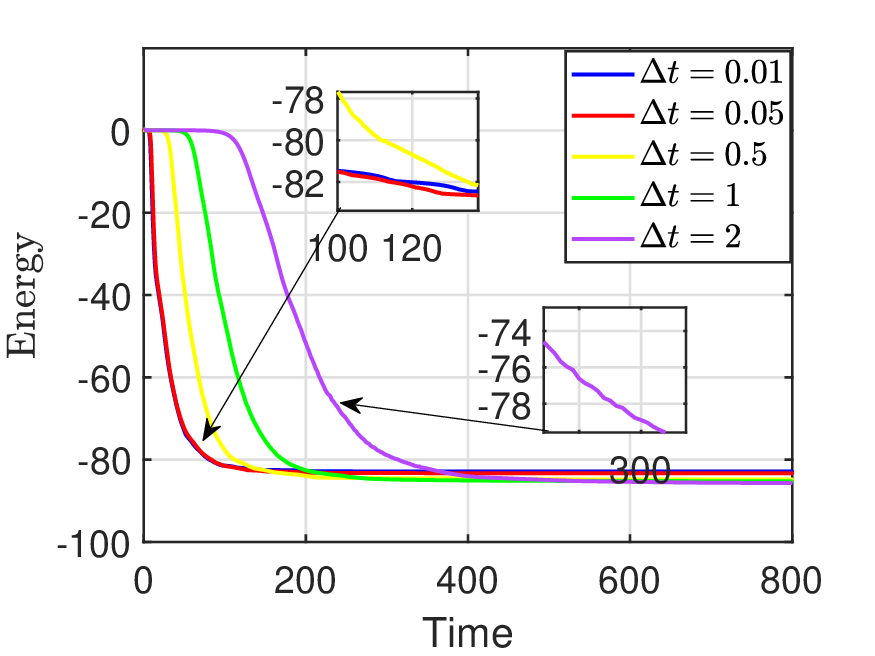}
	}
	\caption{The phase transition simulations with vacancy potential using fixed time steps.}\label{fig:4.15}
\end{figure}

\begin{figure}[htp]
	\centering
	\subfigure[$\Delta t_{\min}=0.0001, \Delta t_{\max}=2,\alpha_1=10^4$]{
		\includegraphics[width=4cm,height=3cm]{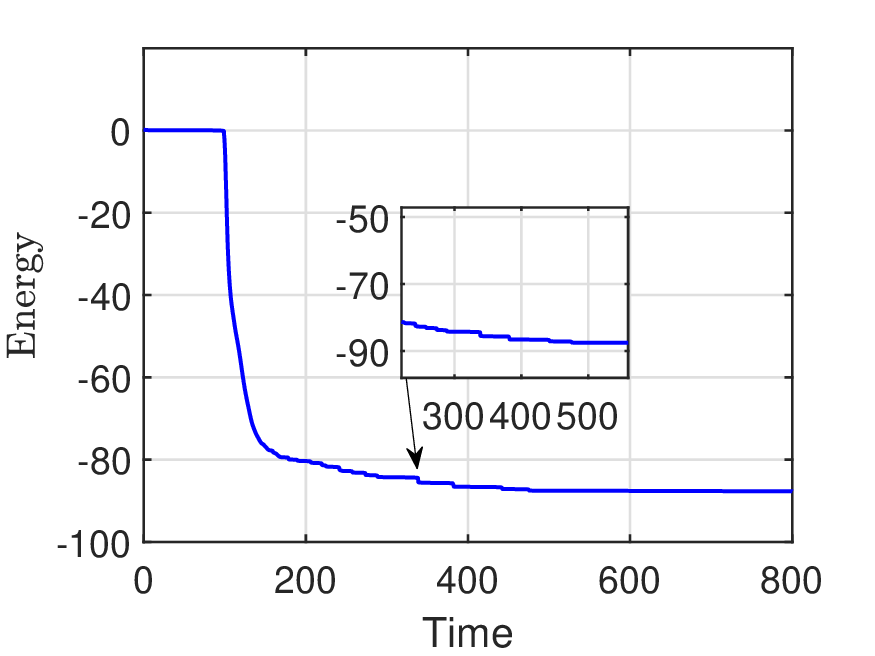}
		\includegraphics[width=4cm,height=3cm]{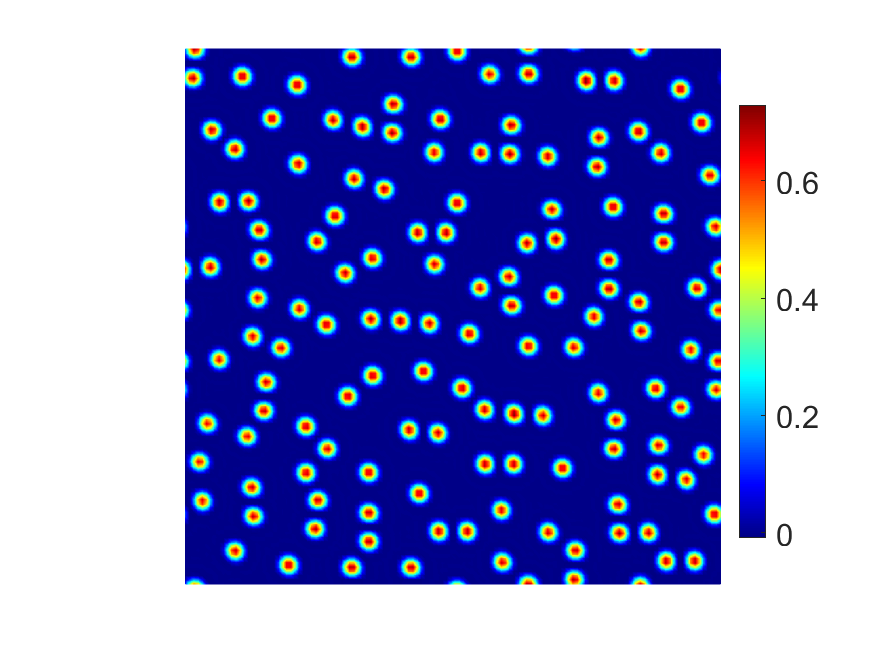}
		\includegraphics[width=4cm,height=3cm]{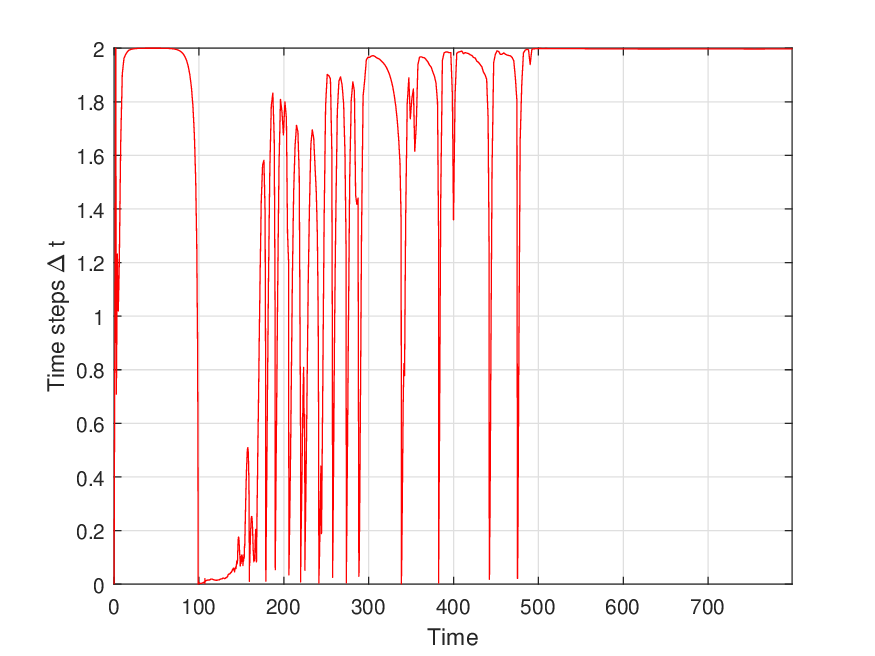}
	}
	\subfigure[$\Delta t_{\min}=0.002, \Delta t_{\max}=6,\alpha_1=10^5$]{
		\includegraphics[width=4cm,height=3cm]{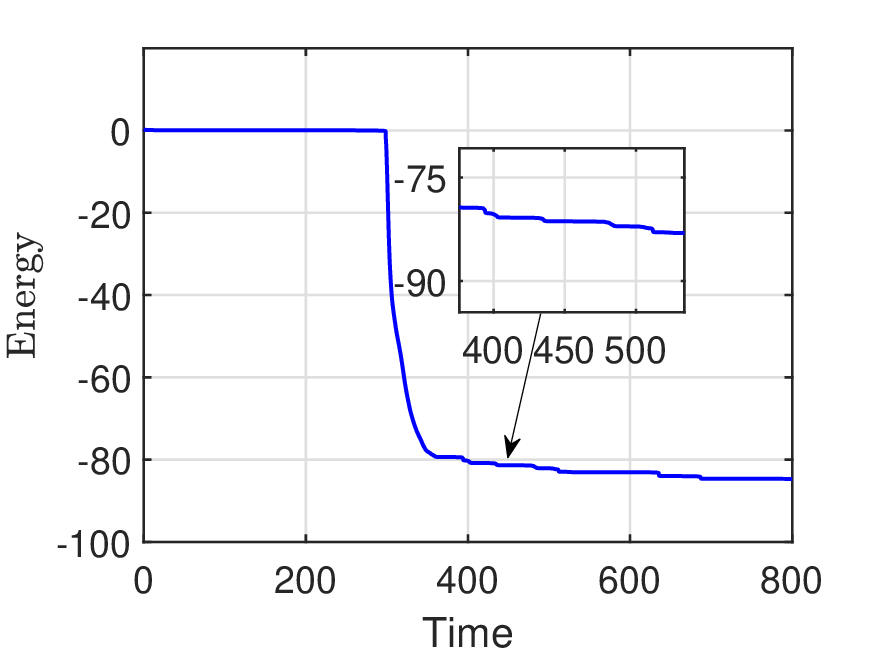}
		\includegraphics[width=4cm,height=3cm]{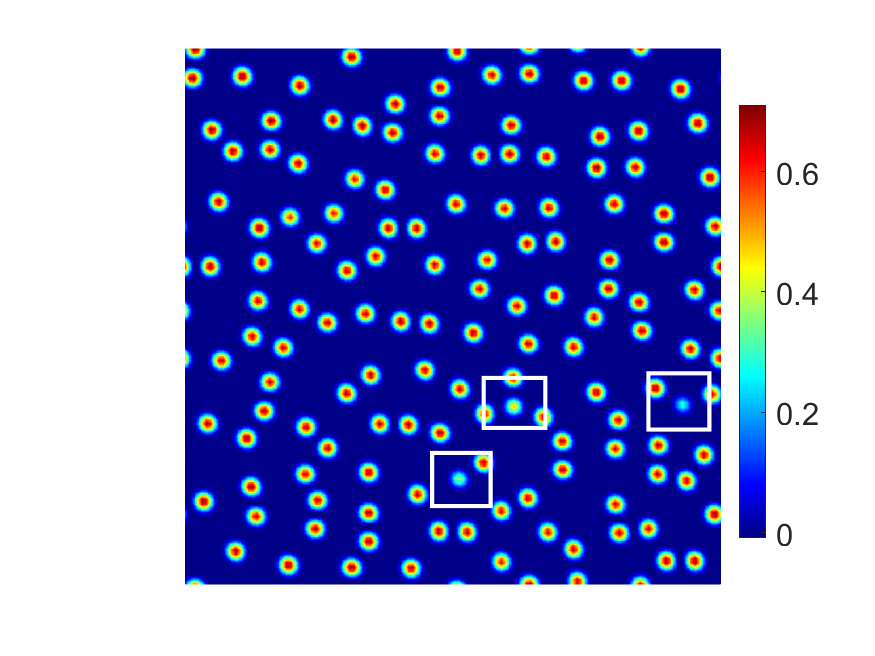}
		\includegraphics[width=4cm,height=3cm]{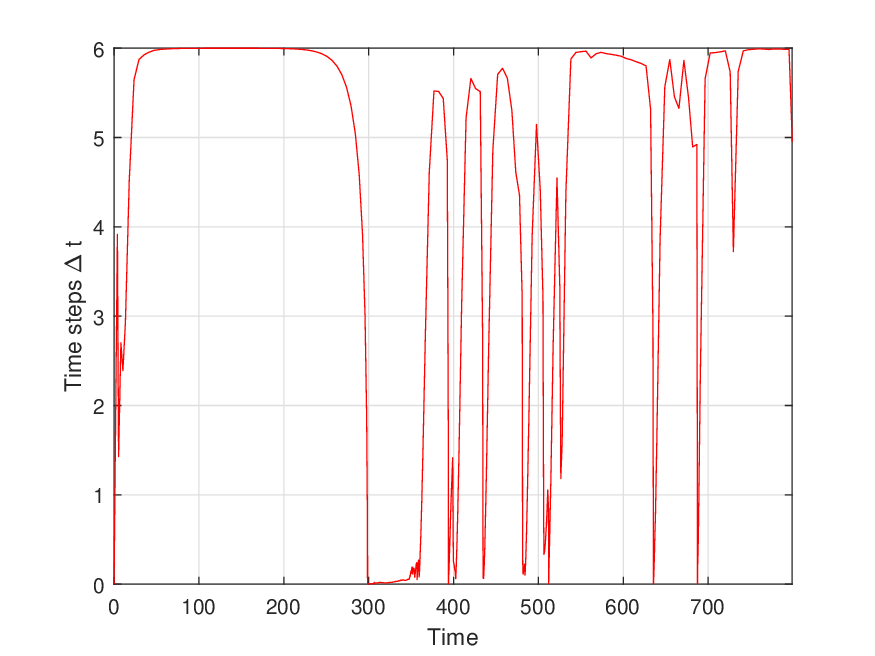}
	}
	\subfigure[$\Delta t_{\min}=0.005, \Delta t_{\max}=10,\alpha_1=10^6$]{
		\includegraphics[width=4cm,height=3cm]{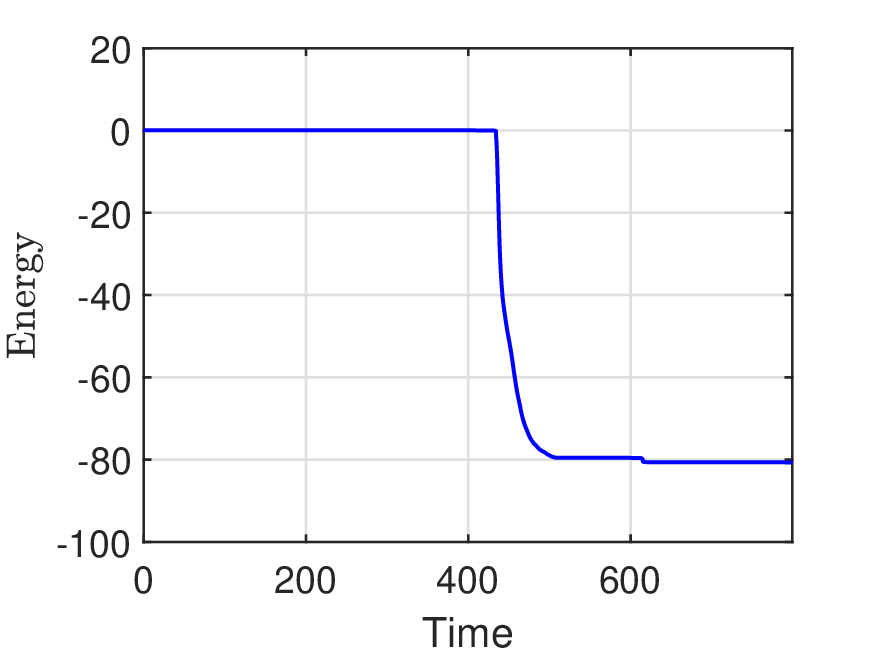}
		\includegraphics[width=4cm,height=3cm]{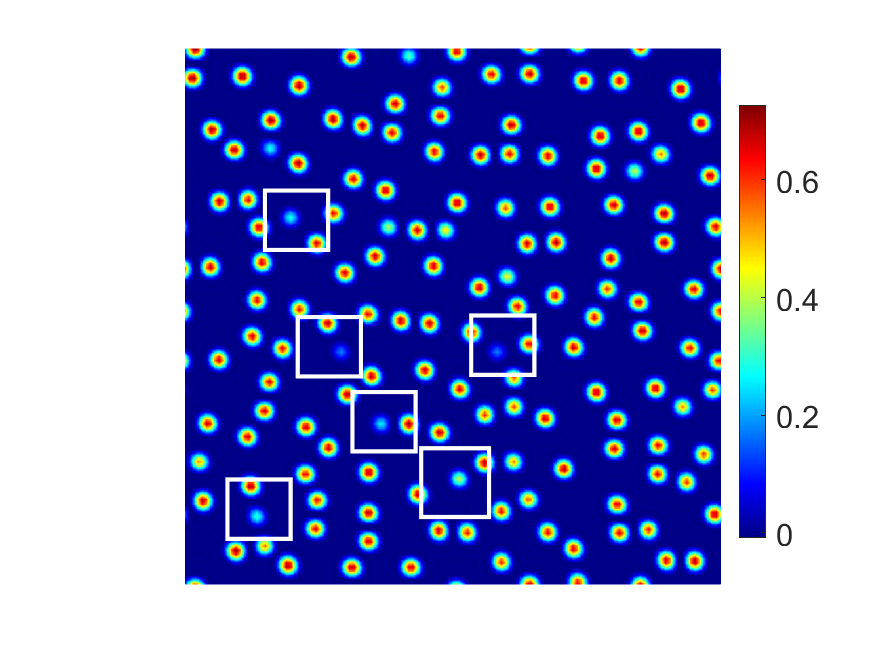}
		\includegraphics[width=4cm,height=3cm]{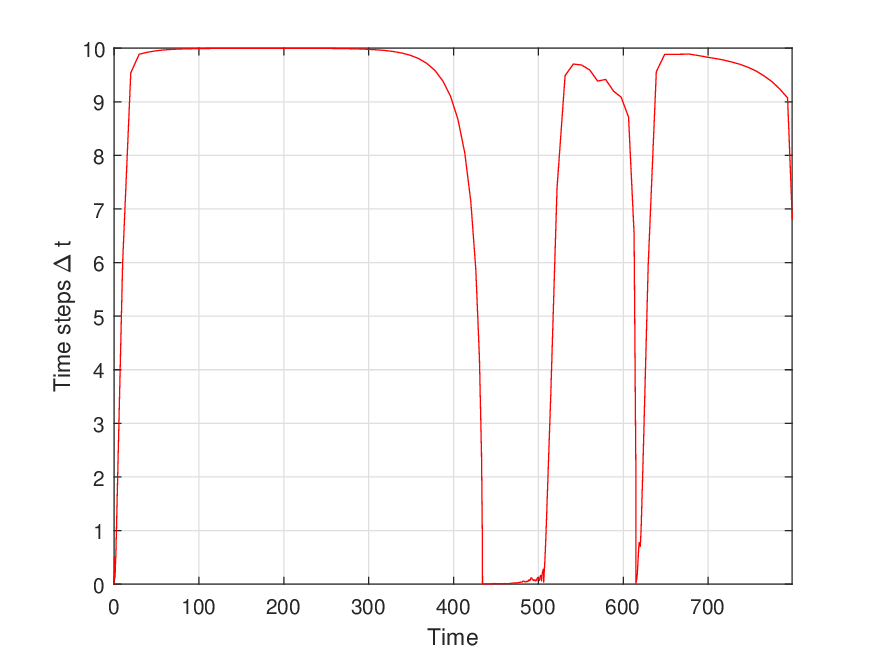}
	}
	\subfigure[$\Delta t_{\min}=0.0001, \Delta t_{\max}=7,\alpha_1=10^6$]{
		\includegraphics[width=4cm,height=3cm]{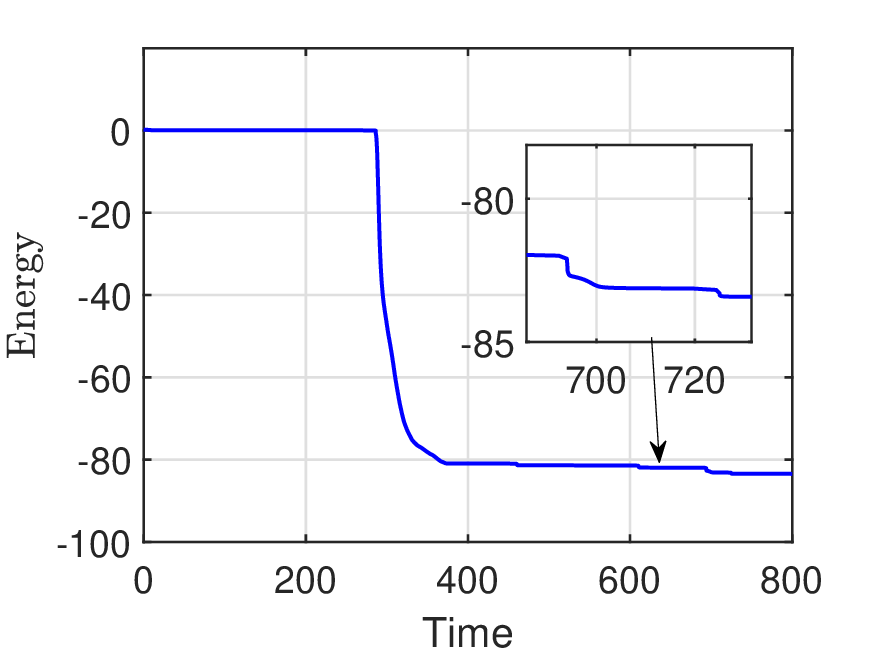}
		\includegraphics[width=4cm,height=3cm]{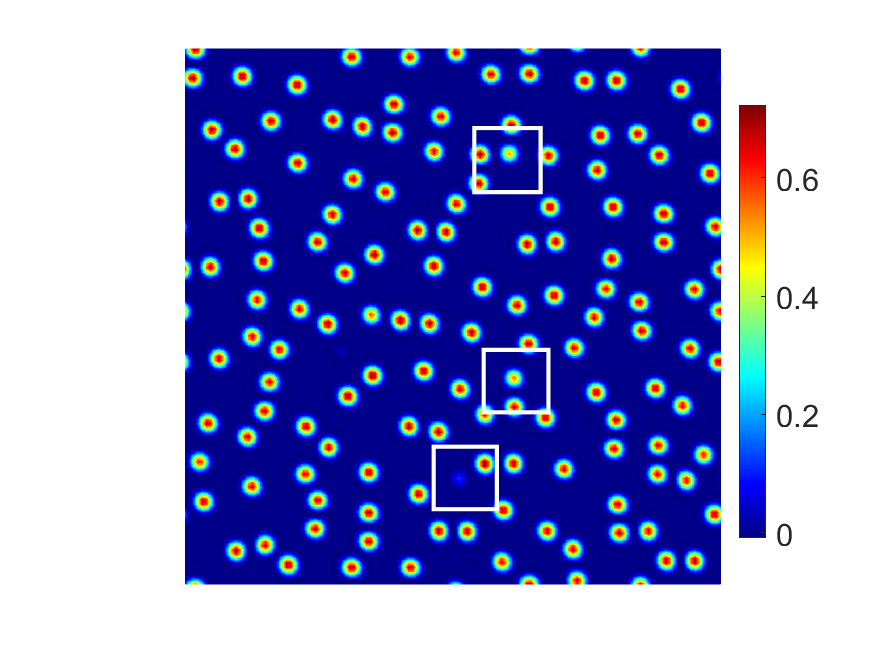}
		\includegraphics[width=4cm,height=3cm]{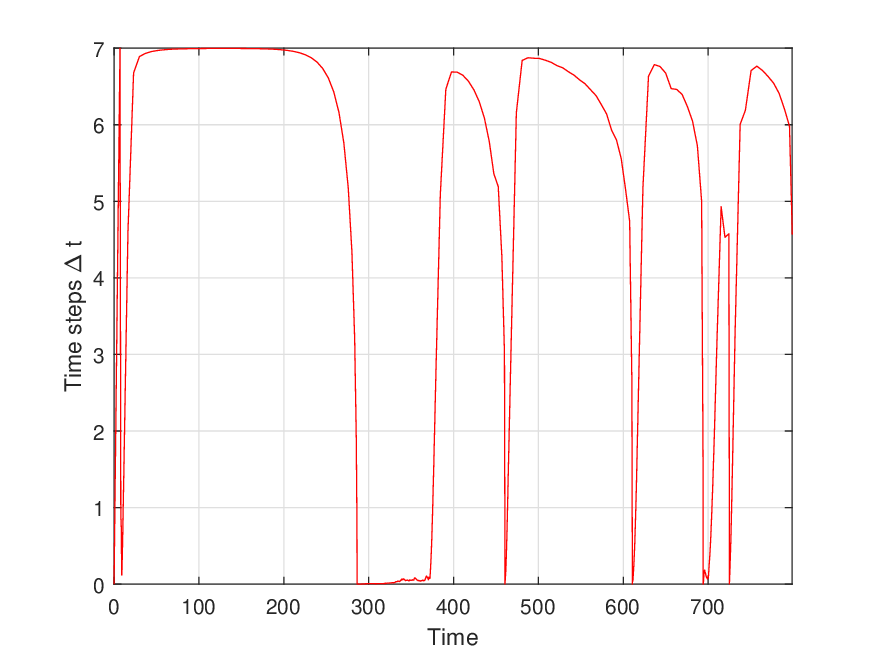}
	}
	\caption{The phase transition simulations with vacancy potential using adaptive time strategy in \cite{47qiao2011adaptive}. First column: discrete energy \eqref{3.21} evolution; Second column: Snapshot of $\phi$ at $t=800$; Third column: adaptive time step evolution.
	}\label{fig:4.16}
\end{figure}

\begin{figure}[htp]
	\centering
	\subfigure[$\mathit{w_{size}} = 7, \mathit{ratio_{max}}=1.5, \Delta t_{\min}=0.0001,  \Delta t_{\max}=2,\alpha_1=10^4$]{
		\includegraphics[width=4cm,height=3cm]{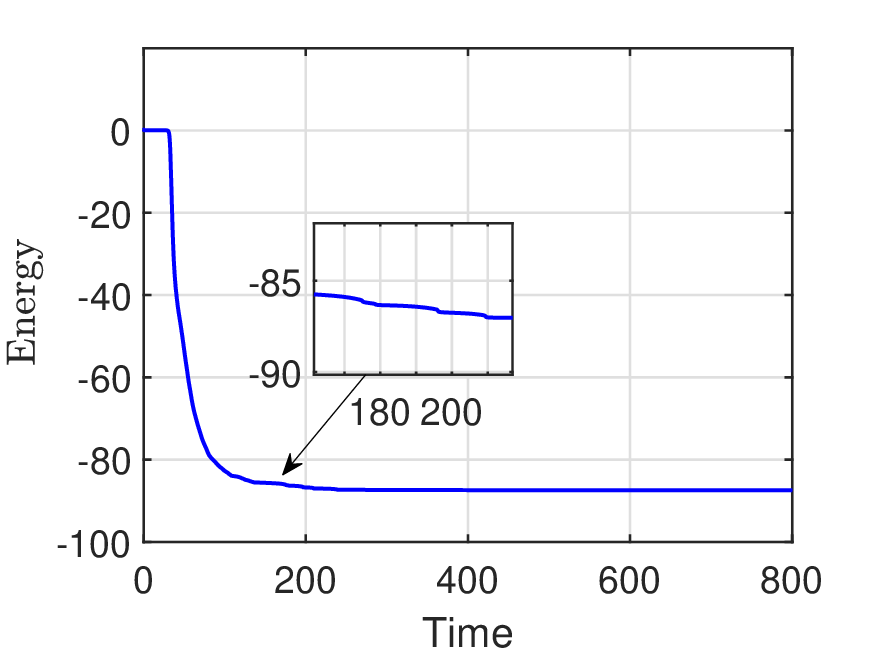}
		\includegraphics[width=4cm,height=3cm]{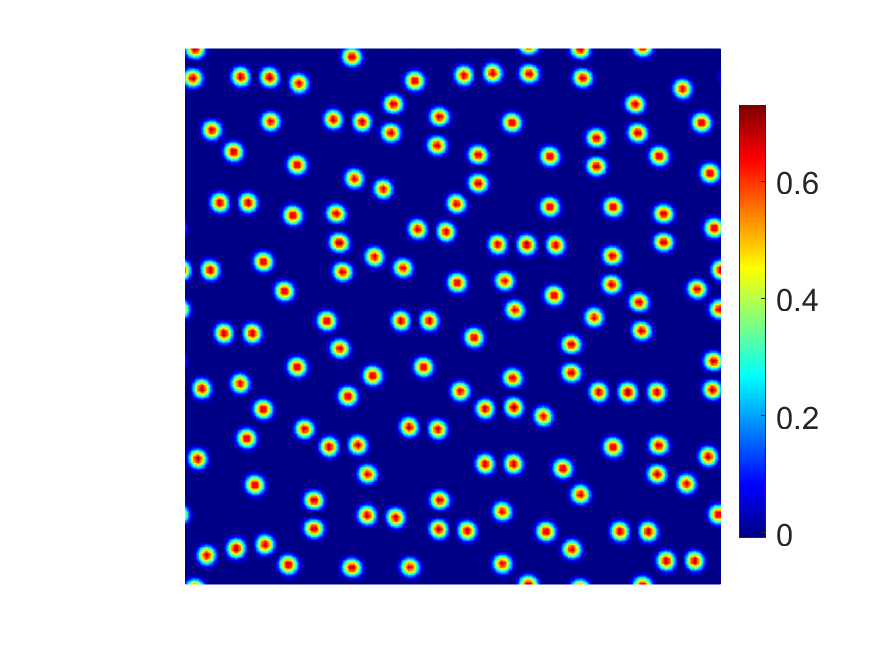}
		\includegraphics[width=4cm,height=3cm]{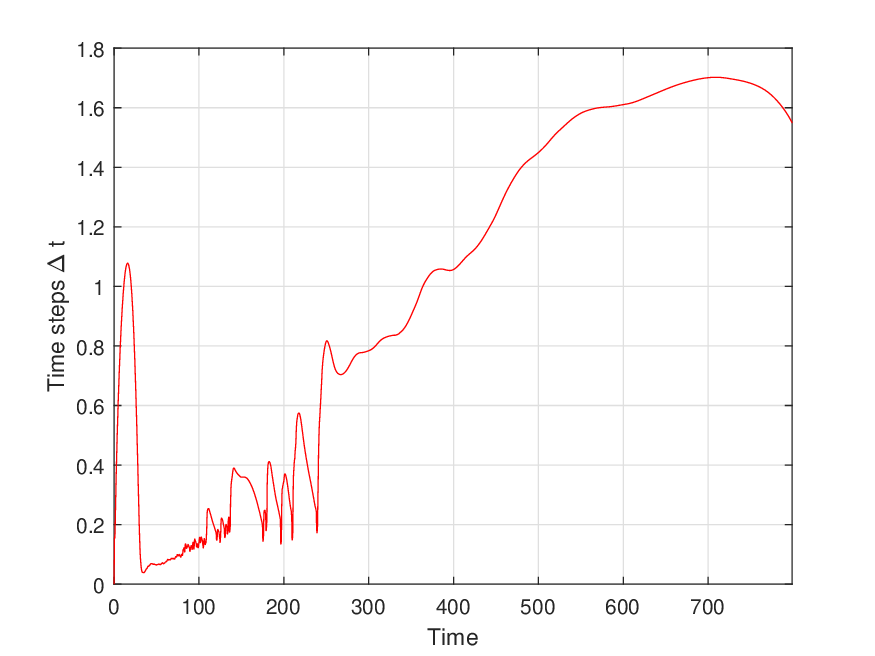}
	}
	\subfigure[$\mathit{w_{size}} = 7, \mathit{ratio_{max}}=1.5, \Delta t_{\min}=0.002,  \Delta t_{\max}=6,\alpha_1=10^5$]{
		\includegraphics[width=4cm,height=3cm]{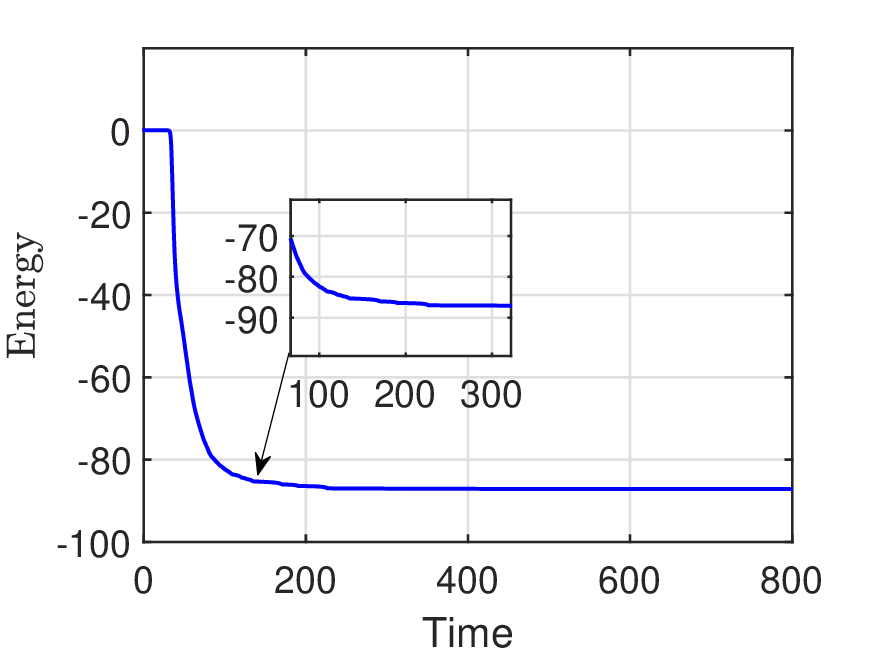}
		\includegraphics[width=4cm,height=3cm]{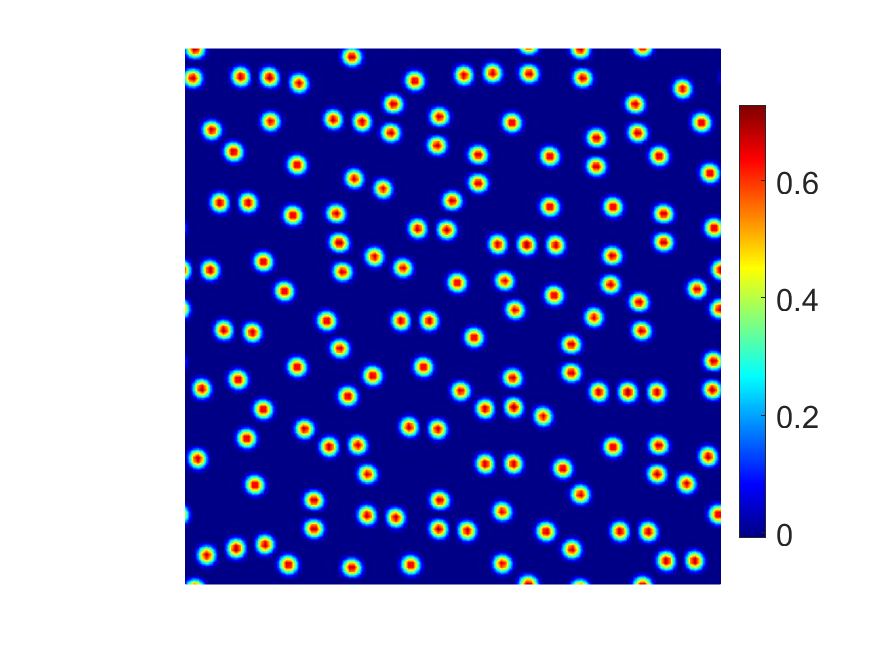}
		\includegraphics[width=4cm,height=3cm]{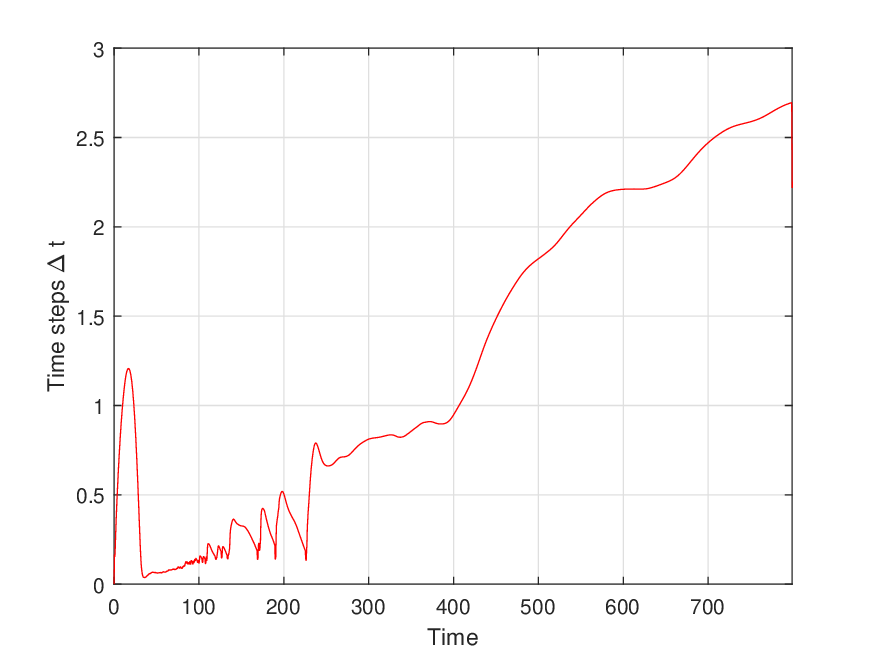}
	}
	\subfigure[$\mathit{w_{size}} = 7, \mathit{ratio_{max}}=1.5, \Delta t_{\min}=0.005,  \Delta t_{\max}=10,\alpha_1=10^6$]{
		\includegraphics[width=4cm,height=3cm]{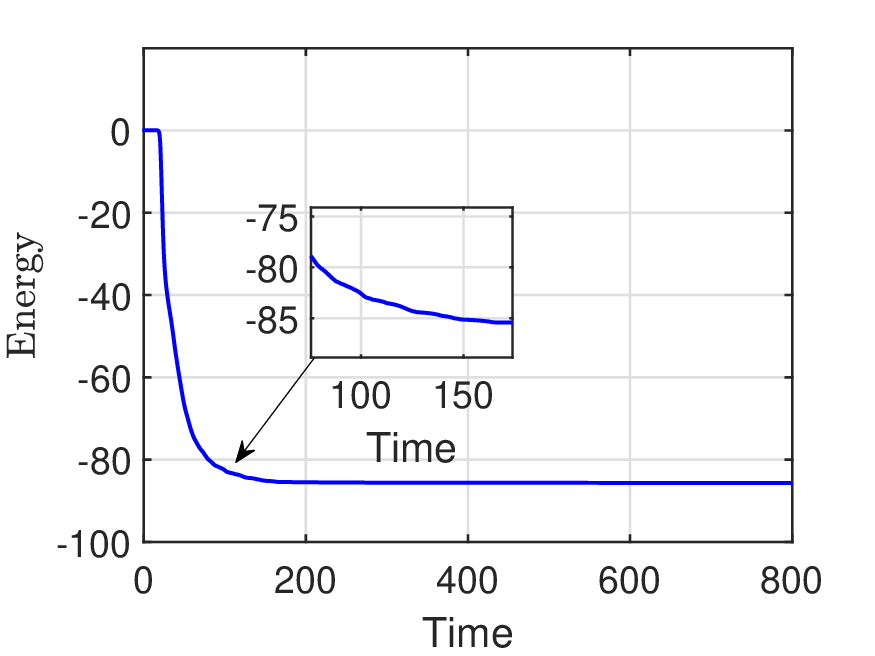}
		\includegraphics[width=4cm,height=3cm]{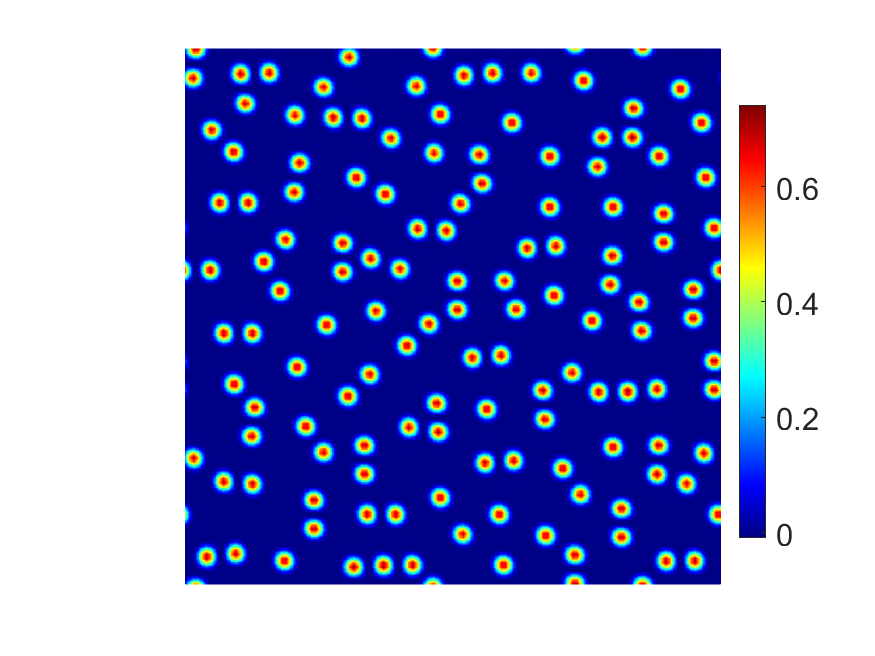}
		\includegraphics[width=4cm,height=3cm]{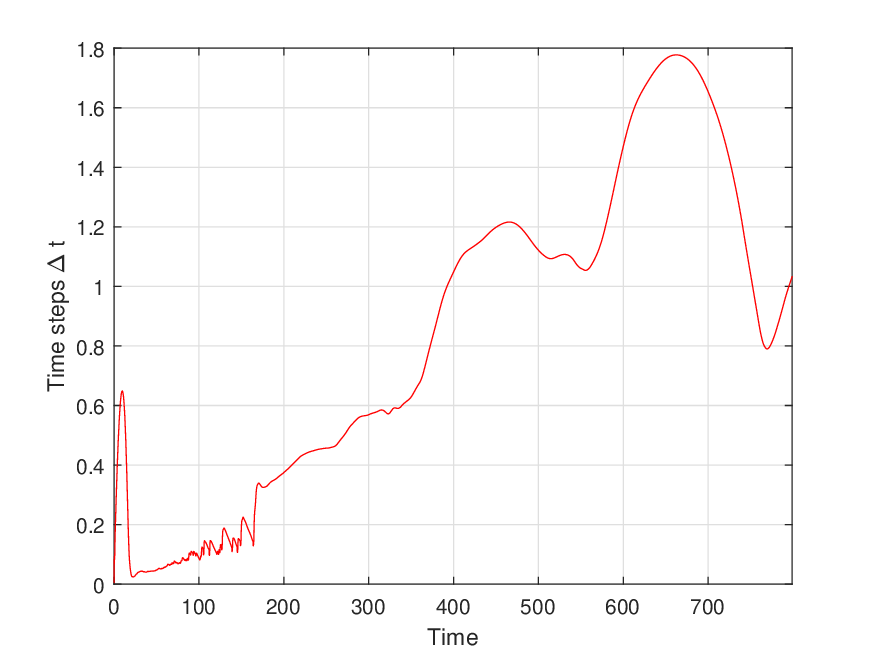}
	}
	\subfigure[$\mathit{w_{size}} = 7, \mathit{ratio_{max}}=1.5, \Delta t_{\min}=0.0001,  \Delta t_{\max}=7,\alpha_1=10^6$]{
		\includegraphics[width=4cm,height=3cm]{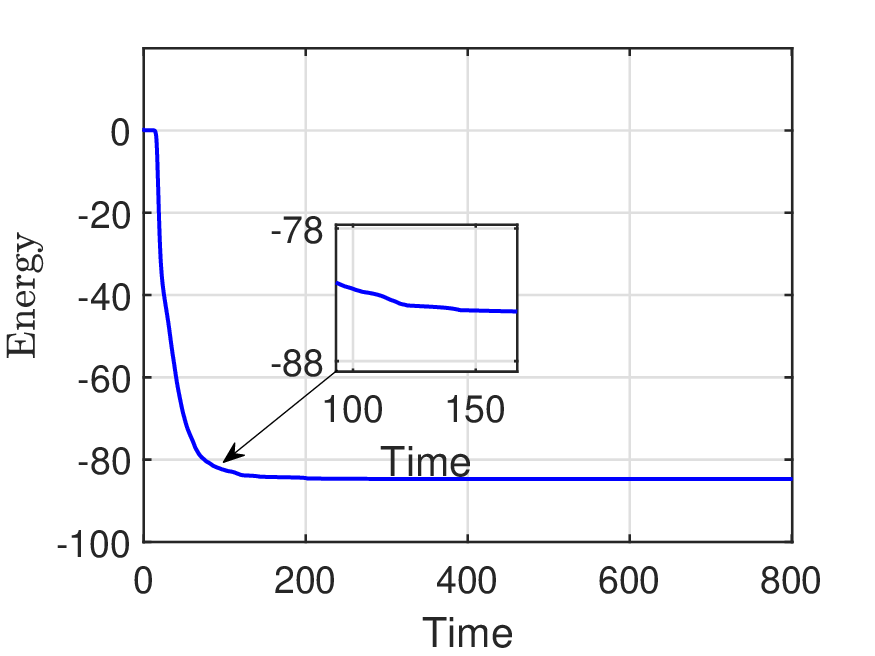}
		\includegraphics[width=4cm,height=3cm]{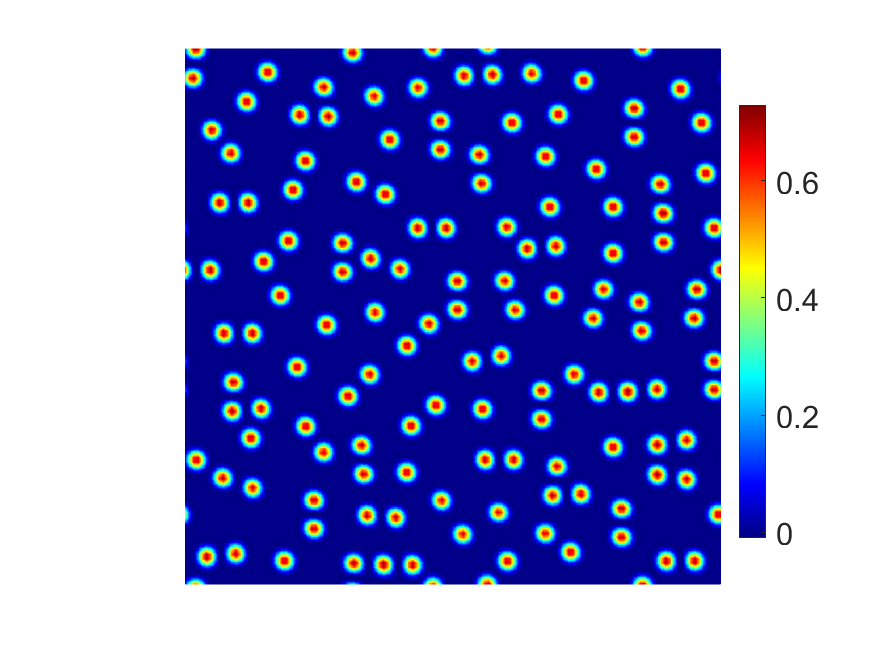}
		\includegraphics[width=4cm,height=3cm]{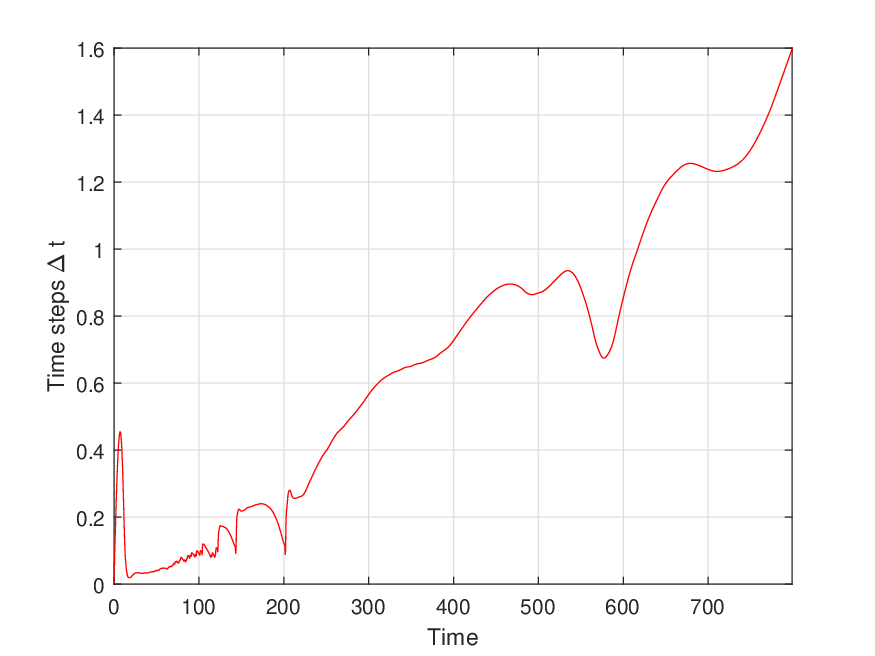}
	}		
	\caption{The phase transition simulations with vacancy potential using the EV-MA adaptive time strategy. First column: discrete energy \eqref{3.21} evolution; Second column:  Snapshot of $\phi$ at $t=800$; Third column: adaptive time step evolution.}\label{fig:4.17}
\end{figure}
To verify the unconditionally energy stability of these three schemes, Fig.~\ref{fig:4.13}(a) presents the energy dissipation curves for the S-SAV-CN, S-GPAV-CN, and S-ESAV-CN schemes with $S=60$ and $S=6000$(S-SAV-CN scheme). As shown in the results, for identical parameter values, the modified energies (\eqref{3.3}, \eqref{3.40}, and \eqref{3.63}) exhibit similar decay rates. Furthermore, the energy drop point occurs later with $S=6000$ compared to $S=60$.
These results collectively demonstrate the energy stability of the numerical schemes. 
Fig.~\ref{fig:4.13}(b) plots the mass evolution curves for all three schemes with $S=60$, numerically proving the conservation of mass.
To compare the numerical results of the three discretization schemes in detail, Fig.~\ref{fig:4.14} presents the spatial distribution of the differences in $\phi$, i.e., (a) $\phi_{(S-GPAV)} - \phi_{(S-SAV)}$ and (b) $\phi_{(S-ESAV)} - \phi_{(S-SAV)}$. 
These $\phi$ values are obtained at $t=800$ with $S=60$.
Under the same experimental parameters, all three schemes exhibit nearly identical crystal growth patterns around the three initial crystallites, with minor discrepancies ($ \sim 10^{-5}$). However, distinct crystallite distributions emerge in regions distant from the three initial crystallites.
As shown in Fig.\ref{fig:4.14}(c)-(d), the temporal evolution of the energy and mass differences (from Fig.\ref{fig:4.13}) reveals that: (i) the energy dissipation rates differ slightly during rapid energy decay stages, and (ii) strict mass conservation is maintained,  with zero mass differences throughout the simulation.
\subsection{Adaptive time stepping}
In this example, we use the phase transition experiments to verify the effectiveness of the new adaptive time-stepping strategy. The initial conditions are defined as \eqref{4.2.1}, and the other parameters are set as
$\Omega = [0, 128]^{2}, \varepsilon=1, \alpha=1, \beta=1, M=1, T=800, h_{vac} = 5000, B=1e5, S_{cr} = 100.$
Fig.~\ref{fig:4.15} presents numerical results obtained using small, fixed time-step sizes: the phase diagram at $t = 800$ with (a)  $\Delta t = 0.015, S = 100$; (b) $\Delta t = 0.001, S = 0$; (c) $\Delta t = 0.015, S = 0$; and (d) Evolution of discrete energy \eqref{3.21}. Testing revealed that the stabilization term must be activated when the time step exceeds the threshold $\Delta t_{cr} = 0.014$ in this case. Otherwise, instability occurs (see Fig.~\ref{fig:4.15}(c)).

To compare the energy evolution and time steps under different adaptive strategies, we use the adaptive strategy in \cite{47qiao2011adaptive}. 
With $\Delta t_{\text{cr}} = 0.014$, the following parameter combinations are used in Fig.~\ref{fig:4.16}: (a)$\Delta t_{\min} = 0.0001$, $\Delta t_{\max} = 2$, $\alpha_1 = 10^4$; (b)$\Delta t_{\min} = 0.002$, $\Delta t_{\max} = 6$, $\alpha_1 = 10^5$; (c)$\Delta t_{\min} = 0.005$，$\Delta t_{\max} = 10$, $\alpha_1 = 10^6$; and (d)$\Delta t_{\min} = 0.0001$, $\Delta t_{\max} = 7$, $\alpha_1 = 10^6$.
Fig.~\ref{fig:4.16} shows the corresponding results, including the evolution of the discrete energy \eqref{3.21}, the phase diagram at $t=800$, and the adaptive time step evolution.

Owing to the high-order and strongly nonlinear nature of the VMPFC model, the energy is highly sensitive to changes in the time steps, leading to significant numerical oscillations(see the first column of Fig.~\ref{fig:4.16}). 
These oscillations, in turn, induce violent fluctuations in the time steps (third column) and undermine the stability of the adaptive algorithm, as reflected in the disrupted phase structures within the white boxes in the second-column subfigures.
A comparison with the benchmark solution and the reference discrete energy provided in Fig.~\ref{fig:4.15} further confirms noticeable deviations in the numerical results, suggesting that the stability of the algorithm is influenced by parameter selection.

Furthermore, the same experiment is conducted using the newly proposed EV-MA adaptive time-stepping algorithm\eqref{alg:adaptive_timestep}. 
With $\mathit{w_{size}} = 7, \mathit{ratio_{max}}=1.5$, the corresponding  results are depicted in Fig.~\ref{fig:4.17}.
The EV-MA adaptive strategy utilizes a moving average of the discrete energy \eqref{3.21} variation to achieve smoother step changes, which effectively suppresses step oscillations (see the third column of Fig.~\ref{fig:4.17}) and more accurately captures the physical evolution process (second column). Simultaneously, during periods of rapid energy change, the algorithm automatically selects smaller time steps to capture the variation of the numerical solution, whereas larger steps are adopted when the system reaches a steady state, thereby enhancing computational efficiency.

In practical applications, the stable parameter range of the adaptive time algorithm is often unknown a priori. For the adaptive time-stepping algorithm in \cite{47qiao2011adaptive}, selecting an appropriate value of $\alpha_1$ remains challenging in strongly nonlinear problems.
In contrast, the proposed adaptive algorithm demonstrates significantly improved robustness. In most cases, the algorithm's performance shows low sensitivity to the choice of $\alpha_1$, while also allowing the use of smaller $\Delta t_{\min}$ and larger $\Delta t_{\max}$ to accommodate practical simulation requirements. As a result, the proposed strategy substantially widens the effective parameter range and reduces the computational effort associated with parameter tuning. Finally, a comparison of the computational cost between fixed and adaptive time-stepping approaches is summarized in Table~\ref{tab:tab1}. The adaptive method uses the following parameter set: $\mathit{w_{size}} = 7$, $\mathit{ratio_{max}} = 1.5$, $\Delta t_{\min} = 0.0001$, $\alpha_1 = 10^4$, and $\Delta t_{\max} = 2$. 
\section{Conclusion}
In this work, we develop and analyze three unconditionally energy stable numerical schemes  with non-iterative computation for the VMPFC model. We first construct a Crank–Nicolson scheme based on the stabilized-SAV method, and then introduce two improved methods for further optimization. The first method reduces the computational complexity by decreasing the number of constant-coefficient linear equations to be solved at each time step, while the second method not only maintains this reduced complexity but also eliminates the lower-bound restriction.
\clearpage 
\begin{table}[t]
	\centering
	\caption{The CPU time (seconds) comparison for phase transition simulations.}
	\label{tab:tab1}
	\begin{tabular}{c c c c c } 
		\hline
		Fixed $\Delta t=0.001$&Fixed $\Delta t=0.015$& Adaptive method in\cite{47qiao2011adaptive}& EV-MA Adaptive method\\ 
		\hline 
		6194.1  & 655.018  &27.801 & 8.432 \\
		\hline
	\end{tabular}
\end{table}
Although large time steps can achieve correct steady states via unconditional stability, they may fail to capture rapid transient processes of the structure.
The unconditional energy stability of our schemes permits the use of suitable adaptive time strategy to save CPU time without losing accuracy. 
We propose an adaptive time step approach based on the Energy-Variation Moving Average (EV-MA), which improves the robustness of the conventional adaptive time stepping by effectively suppressing step-size oscillations caused by strongly nonlinear properties.
The small time steps are used when the energy variation is large while lager time steps can be used when the energy variation is small.
To our knowledge, this is the first adaptive time-stepping strategy for the VMPFC model. It offers an effective framework for systems highly sensitive to time-step changes during rapid energy decay.
Numerical results indicate that the method performs effectively under various parameter settings, demonstrating satisfactory robustness. 
\section*{Acknowledgments}
This work is supported by National Natural Science Foundation of China(12431014) and Hunan Provincial Innovation Foundation for Postgraduate, China (XDCX2024Y174).

\end{document}